\documentclass[a4paper,12pt]{article}
\usepackage[text={6.8in,9.3in},centering]{geometry}
\usepackage{graphicx,url,subfig}
\RequirePackage[OT1]{fontenc}
\RequirePackage{amsthm,amsmath,amssymb,amscd}
\RequirePackage[numbers]{natbib}
\usepackage{enumerate,enumitem}
\usepackage{datetime}
\usepackage{etex}
\DeclareMathOperator{\sign}{sign}
\RequirePackage[colorlinks,citecolor=blue,urlcolor=blue,pdfencoding=auto]{hyperref}

\makeatother
\numberwithin{equation}{section}
\allowdisplaybreaks

\theoremstyle{plain}
\newtheorem{theorem}{Theorem}[section]

\newtheorem{lemma}[theorem]{Lemma}
\newtheorem{proposition}[theorem]{Proposition}

\theoremstyle{definition}
\newtheorem{definition}[theorem]{Definition}

\newtheorem{remark}[theorem]{Remark}

\newcommand{\E}{\mathbb{E}}
\newcommand{\D}{\mathbb{D}}
\newcommand{\W}{\dot{W}}
\newcommand{\ud}{\ensuremath{\mathrm{d}}}

\newcommand{\sgn}{\text{sgn}}

\newcommand{\Indt}[1]{\one_{\left\{#1 \right\}}}

\newcommand{\Norm}[1]{\left|\left|  #1   \right|\right|}

\newcommand{\spt}[1]{\text{supp}\left(#1\right)}
\newcommand{\InPrd}[1]{\left\langle #1 \right\rangle}

\newcommand{\calB}{\mathcal{B}}

\newcommand{\calF}{\mathcal{F}}
\newcommand{\calG}{\mathcal{G}}
\newcommand{\calK}{\mathcal{K}}
\newcommand{\calH}{\mathcal{H}}
\newcommand{\calI}{\mathcal{I}}
\newcommand{\calL}{\mathcal{L}}

\newcommand{\calN}{\mathcal{N}}

\newcommand{\calP}{\mathcal{P}}

\newcommand{\calS}{\mathcal{S}}

\newcommand{\bbN}{\mathbb{N}}
\newcommand{\bbP}{\mathbb{P}}

\newcommand*{\one}{{{\rm 1\mkern-1.5mu}\!{\rm I}}}

\newcommand{\R}{\mathbb{R}}

\def\myred{black}

\DeclareMathOperator{\Lip}{\mathit{L}}
\DeclareMathOperator{\LIP}{Lip}
\DeclareMathOperator{\lip}{\mathit{l}}

\DeclareMathOperator{\Vip}{\overline{\varsigma}}

\DeclareMathOperator{\vv}{\varsigma}

\newcommand{\Dxa}{{}_x D_\delta^\alpha}

\newcommand{\lMr}[3]{\:{}_{#1} #2_{#3}}

\title{Regularity and strict positivity of densities for \\
the nonlinear stochastic heat equation}

\author{
{\bf Le Chen} and {\bf Yaozhong Hu\footnote{Research partially supported by a grant from the Simons Foundation \#209206.} }
and {\bf David Nualart\footnote{Research partially supported by the NSF grant DMS1512891.} }
\\[1em]
Department of Mathematics\\
University of Kansas
\date{}
\vspace{-1em}
}

\begin{document}
\maketitle
\begin{center}
\begin{minipage}[rct]{5 in}
\footnotesize \textbf{Abstract:}
In this paper, we establish a necessary and sufficient condition 
for the existence and regularity of the density of the solution to a semilinear stochastic (fractional) heat equation with measure-valued initial conditions.
Under a mild cone condition for the diffusion coefficient, 
we establish the smooth joint density at multiple points. The tool we use is Malliavin calculus. The main ingredient is to prove that the solutions to a related stochastic partial differential equation have negative moments of all orders. Because we cannot prove $u(t,x)\in \mathbb{D}^\infty$ for measure-valued initial data, we need a localized version of Malliavin calculus. Furthermore, we prove that the (joint) density is strictly positive in the interior of the support of the law, where we allow both measure-valued initial data and unbounded diffusion coefficient.
The criteria introduced by Bally and Pardoux \cite{BP98} are no longer applicable for the {\it parabolic Anderson model}. We have extended their criteria to a localized version. Our general framework includes the parabolic Anderson model as a special case.

\vspace{2ex}

\vspace{2ex}
\vspace{4ex}
\end{minipage}
\end{center}

\tableofcontents
\setlength{\parindent}{1.5em}



\section{Introduction}

In this paper,  we are interested in   the density of the law of the solution to
the following stochastic fractional heat equation:
\begin{align}\label{E:FracHt}
 \begin{cases}
  \left(\displaystyle\frac{\partial}{\partial t} - \Dxa \right) u(t,x) =
\rho(t,x,u(t,x))  \dot{W}(t,x),& t>0\;,\: x\in\R,\cr
u(0,\cdot) = \mu(\cdot).
 \end{cases}
\end{align}
We shall concentrate on  the regularity and the strict positivity of the  density of the solution $u(t,x)$ as a random variable.
First, we explain the meaning of the terms in above equation.
The operator $\Dxa$ is the {\it Riesz-Feller fractional differential operator} of
order $\alpha$ and skewness $\delta$.  In terms of Fourier transform this operator
 is defined by
\begin{equation}
{\cal F}  (\Dxa  f)(\xi)=-|\xi|^ \alpha   e^{\iota \sign(\xi) \delta \pi/2}\hat f(\xi) \,.
\label{e.1.F-Dxa}
\end{equation}
To study random field solutions, we need to require, and hence will assume throughout  the paper, that
\begin{align}\label{E:alphaDelta}
\alpha\in (1,2]\quad\text{and}\quad
|\delta|\le 2-\alpha.
\end{align}
Under these two conditions and when $\alpha\ne 2$,  we can express the above Riesz-Feller fractional differential operator as 
\begin{align*}
\Dxa  f(x)
=\frac{\Gamma(1+ \alpha)}{\pi}\bigg\{&\sin\left(( \alpha+\delta)\frac{\pi}{2}\right)\int_0^\infty
\frac{f(x+\xi)-f(x)-\xi f'(x)}{\xi^{1+ \alpha}}d\xi\\
+&\sin\left(( \alpha-\delta)\frac{\pi}{2}\right)\int_0^\infty
\frac{f(x-\xi)-f(x)+\xi f'(x)}{\xi^{1+ \alpha}}d\xi\bigg\}\,.
\end{align*}
The nonlinear coefficient $\rho(t,x,z)$ is a continuous function which is differentiable in the third argument with
a bounded derivative and it also satisfies the linear growth condition:
for some $\Lip_\rho>0$ and $\vv\ge 0$,
\begin{align}\label{E:LinGrw}
 \left|\rho(t,x,z)\right| \le
  {\Lip_\rho (\vv+|z|)}
 \quad\text{for all $(t,x,z)\in \R_+\times\R\times\R$}.
\end{align}
Throughout of the paper, denote $\rho' :=\partial\rho/\partial z$ and
$\rho(u(t,x))$ is understood as a short-hand notation for $\rho(t,x,u(t,x))$.
An important case that fits these conditions is the {\it parabolic Anderson model} (PAM)
\cite{CarmonaMolchanov94PAM}: $\rho(u)=\lambda u$, $\alpha=2$ and $\delta=0$.

$\dot{W}$ is the space-time white noise on $\R_+\times\R$.
The initial data $\mu$ is assumed to be a Borel (regular) measure on $\R$ such that
\begin{align}\label{E:InitD}
\begin{cases}
\displaystyle
 \sup_{y\in\R}\int_\R\frac{|\mu|(\ud x)}{1+|y-x|^{1+\alpha}}<+\infty, & \text{if $\alpha\in (1,2)$,}\\[1em]
\displaystyle
 \int_\R e^{-c x^2} |\mu|(\ud x) <+\infty\quad \text{for all $c>0$}, & \text{if $\alpha=2$\,, }
 \end{cases}
\end{align}
where we recall  the {\it Jordan decomposition} of a signed Borel measure $\mu=\mu_+-\mu_-$,
where $\mu_\pm$ are two nonnegative Borel measures with disjoint support and we denote   $|\mu|=\mu_++\mu_-$.  
By ``$\mu>0$'', it means that $\mu$ is a nonnegative and nonvanishing measure.
It is interesting to point out that our assumption allows  the initial data  $\mu$ to be the Dirac measure $\delta$.

Before we state our main results let us recall some relevant works.
Nualart and Quer-Sardanyons \cite{NualartQuer07} proved the existence of a smooth density for the solutions to a general
class of stochastic partial differential equations (SPDE's), including stochastic heat and wave equations.
 They assume that the initial data is vanishing and $\rho$ is $C^\infty$ with bounded derivatives.   Moreover,
the condition $\rho(u)\ge c>0$ is required in their proof, which excludes the case $\rho(u)=\lambda u$.
Mueller and Nualart \cite{MuellerNualart08} later showed
that for the stochastic heat equation (i.e., $\alpha=2$) on $[0,1]$ with Dirichlet boundary conditions,
the condition $\rho(u)\ge c>0$ can be removed.
They require the initial condition to be a H\"older continuous function such that
$\rho(0,x,u(0,x))\ne 0$ for some $x\in\R$.
The existence of absolutely continuous density under a similar setting as \cite{MuellerNualart08} is obtained earlier by Pardoux and Zhang \cite{PardouxZhang93}.
Note that all these results are for densities at a single  time space
point $(t,x)\in(0,\infty)\times\R$.

For solution  at multiple spatial points $(u(t,x_1),\dots,u(t,x_d))$,  Bally and Pardoux \cite{BP98} proved a local result, i.e.,
smoothness of  density on $\{\rho\ne 0\}^d$
for the space-time white noise case and
more recently Hu {\it et al} \cite{HHNS14} proved this result for the spatially colored noise case (which is white in time).
In both \cite{BP98} and \cite{HHNS14}, the function $\rho(t,x,z)=\rho(z)$ does not depend on $(t,x)$.
The initial data are assumed to be  continuous   in \cite{BP98} and  vanishing   in \cite{HHNS14}.

The aim of this paper is to extend the above results with more general initial conditions.
 In particular,
we are interested in proving that under some regularity and non-degeneracy conditions
the solution at a single point  or  at multiple spatial points  has a smooth (joint) density and the density is strictly positive
in the interior of support of the law.
We shall not concern with the existence and uniqueness of the solution
since it   has been studied in \cite{ChenDalang13Heat,ChenDalang14Holder,ChenDalang15FracHeat}.
Notice  that a comparison principle has been obtained recently in \cite{ChenKim14Comparison}. Let us point  out that  the initial measure satisfying  \eqref{E:InitD} poses some serious  difficulties.
For example, the following statement will no longer hold    true:
\begin{align}\label{E:BddNorm}
\sup_{(t,x)\in [0,T]\times (a,b)}\E[|u(t,x)|^p]<+\infty,\qquad \text{for $T>0$ and $-\infty\le a<b\le +\infty$}.
\end{align}
For this reason, more care needs to be taken when dealing with various approximation procedures. This property \eqref{E:BddNorm} is important in the conventional Malliavin calculus. For example, without property \eqref{E:BddNorm}, even in the case of $\rho$ is smooth and its derivatives of all orders are bounded, we are not able to prove the property $u(t,x)\in \D^\infty$. We need to introduce a bigger space $\D^\infty_S$ and carry out some localized analysis to deal with this case; see Theorem \ref{T:Density2}, Proposition \ref{P:D1} and Remark \ref{R:LocalAnalize} below for more details.


\bigskip
Now we state our main results on the regularity of densities.
These results are summarized in the following three theorems, Theorems \ref{T:Single}, \ref{T:Mult} and \ref{T:LocDen}.
The differences lie in different assumptions on initial data and on the function $\rho(t,x,u)$.
The first two theorems (Theorems \ref{T:Single} and \ref{T:Mult}) are global results, while the
third   (Theorem \ref{T:LocDen}) is a local one.

\textcolor{\myred}{ The first theorem, Theorem \ref{T:Single}, gives the necessary and sufficient condition for the existence and smoothness of the density of the solution $u(t,x)$.
It extends the sufficient condition (see Theorem \ref{T_:Single} below) by Mueller and Nualart \cite{MuellerNualart08} from the case where $\alpha=2$ and the initial data is an ordinary function to the case where $\alpha\in (1,2]$ and  the initial data is a  measure.}

\begin{theorem}\label{T:Single}
Suppose that $\rho:[0,\infty)\times\R^2\mapsto\R$ is continuous.
Let $u(t,x)$ be the solution  to \eqref{E:FracHt} starting from an initial measure $\mu$
that satisfies \eqref{E:InitD}.
Then we have the following two statements:
\begin{enumerate}[parsep=0ex,topsep=1ex]
 \item[(a)] If $\rho$ is differentiable in the third argument with bounded Lipschitz continuous derivative,
 then  for all $t>0$ and $x\in\R$,  $u(t,x)$ has an absolutely continuous law with respect to the Lebesgue measure
 if and only if 
 \begin{align}\label{E:IFF}
  t>t_0:=\inf\left\{s>0,\: \sup_{y\in\R}\left|\rho\left(s,y,(G(s,\cdot)*\mu)(y)\right)\right|\ne 0\right\}.
 \end{align}
 \item[(b)] If $\rho$ is infinitely differentiable in the third argument with bounded derivatives,
 then  for all $t>0$ and $x\in\R$,  $u(t,x)$ has a smooth  density if and only if condition \eqref{E:IFF} holds.
\end{enumerate}
\end{theorem}

This theorem is proved in Section \ref{S:Main}.
For the PAM (i.e., $\rho(t,x,u)=\lambda u$) with the Dirac delta initial condition, it is not hard to see that $t_0=0$. Theorem \ref{T:Single} implies both existence and smoothness of the density of the law of $u(t,x)$ for any $t>0$ and $x\in\R$.

\bigskip
In the next theorem, Theorem \ref{T:Mult}, by imposing one additional condition \eqref{E:lip}   on the lower bound of $\rho$,
we are able to extend the above result (and hence the result by Mueller and Nualart \cite{MuellerNualart08})
from the density at a single point to a joint density at multiple spatial points
and with slightly different condition on the initial data from that in Theorem \ref{T:Single}. 
\textcolor{\myred}{In particular, the condition ``$\mu>0$'' and the cone condition \eqref{E:lip} below imply immediately that the critical time $t_0$ defined in \eqref{E:IFF} is equal to zero.}

\begin{theorem}\label{T:Mult}
Let $u(t,x)$ be the solution to \eqref{E:FracHt} starting from a nonnegative measure $\mu>0$ that satisfies \eqref{E:InitD}.
Suppose that for some constants $\beta>0$, $\gamma \in (0,2-1/\alpha)$ and $\lip_\rho>0$,
\begin{align}\label{E:lip}
|\rho(t,x,z)|\ge \lip_\rho \exp\left\{-\beta\left[\log\frac{1}{|z|\wedge 1}\right]^\gamma\right\},
\qquad\text{for all $(t,x,z)\in\R_+\times\R^2$}\,.
\end{align}
Then for any  $x_1<x_2<\dots<x_d$ and $t>0$,
we have the following two statements:
\begin{enumerate}[parsep=0ex,topsep=1ex]
 \item[(a)] If $\rho$ is differentiable in the third argument with bounded Lipschitz continuous derivative,
 then the law of the random vector $\left(u(t,x_1),\dots,u(t,x_d)\right)$ is absolutely continuous
 with respect to the Lebesgue measure on $\R^d$.
 \item[(b)] If $\rho$ is infinitely differentiable in the third argument with bounded derivatives,
 then the random vector $\left(u(t,x_1),\dots,u(t,x_d)\right)$ has a smooth  density on $\R^d$.
\end{enumerate}
\end{theorem}
This theorem is proved in Section \ref{SS:Mult}.
Note that because $\alpha\in (1,2]$, we see that $2-1/\alpha\in (1,3/2]$. When \eqref{E:lip} holds for $\gamma=1$, condition \eqref{E:lip} reduces 
to the following condition
\[
|\rho(t,x,z)|\ge \lip_\rho \min\left(|z|^\beta,1\right)
\qquad\text{for all $(t,x,z)\in\R_+\times\R^2$}\,.
\]
In particular, the PAM (i.e., $\rho(t,x,u)=\lambda u$) satisfies \eqref{E:lip} with $\lip_\rho= |\lambda|$ and $\beta=\gamma=1$.
We also note that the range of the exponent $\gamma$ in \eqref{E:lip} can be improved if one can obtain a better bound for the probability on the left-hand side of \eqref{E:Rate} below; see the work by Moreno Flores \cite{MorenoFlores14} for the case $\alpha=2$.

\bigskip

The next theorem, Theorem \ref{T:LocDen}, extends the local results of  Bally and Pardoux \cite{BP98} by allowing more general initial data and allowing the parabolic Anderson model.
The cone condition \eqref{E:lip} in Theorem \ref{T:Mult} is not required. 
However, one should require $\rho(t,x,u)=\rho(u)$ and the conclusion is weaker in the sense that
only a local result is obtained: the smooth density is established over the domain $\{\rho\ne0\}^d$ instead of $\R^d$.

\begin{theorem}\label{T:LocDen}
Let $u(t,x)$ be the  solution   starting from \textcolor{\myred}{a singed initial measure} $\mu$
that satisfies \eqref{E:InitD}.
Suppose that $\rho(t,x,z)=\rho(z)$ and it is infinitely differentiable with bounded derivatives.
Then for any  $x_1<x_2<\dots<x_d$ and $t>0$,  the law
of the random vector $\left(u(t,x_1),\dots,u(t,x_d)\right)$ admits a smooth joint density $p$ on $\{\rho\ne 0\}^d$.   Namely,
for some $p\in C^\infty(\{\rho\ne 0\}^d;\R)$ it holds that for every $\phi\in C_b(\R^d;\R)$ with
$\spt{\phi}\subseteq\{\rho\ne 0\}^d$,
\[
\E\left[\phi(u(t,x_1),\dots,u(t,x_d))\right]=\int_{\R^d}\phi(z)p(z)\ud z.
\]
\end{theorem}
The proof is given at Section \ref{SS:LocDen}.
Note that when the initial data are vanishing,
one can also prove this theorem by verifying the four assumptions in
Theorem 3.1 and Remark 3.2 of \cite{HHNS14}; see Remark \ref{R:H1-H4} for the verification.

\bigskip

As for the strict positivity of the density, most known results assume the boundedness of the diffusion coefficient $\rho$;
see, e.g., Theorem 2.2 of Bally and Pardoux \cite{BP98},
Theorem 4.1 of Hu {\it et al} \cite{HHNS14}
and Theorem 5.1 of Nualart \cite{Nualart13Density}.
This condition excludes the important case: the parabolic Anderson model $\rho(u)=\lambda u$.
The following theorem will cover this linear case. Moreover, it allows measure-valued initial data in some cases.
\textcolor{\myred}{Before stating the theorem, we first introduce a notion. We say that a Borel measure $\mu$ is {\it proper} at a point $x\in \R$ if
there exists a neighborhood $E$ of $x$ such that
$\mu$ restricted to $E$ is absolutely continuous with respect to
the Lebesgue measure.}

\begin{theorem}\label{T:Pos}
Suppose $\rho(t,x,z)=\rho(z)$ and $\rho\in C^\infty(\R)$ such that all derivatives of $\rho$ are bounded.
Let  $ x_1<x_2<\cdots<x_d $  be some  proper   points  of   $\mu$
with a bounded density on each neighborhood.  Then for any $t>0$,
the joint law of $(u(t,x_1),\dots,u(t,x_d))$ admits a $C^\infty$-density $p(y)$, and $p(y)>0$
if $y$    belongs both  to $\{\rho\ne 0\}^d$ and
  to the interior of the support of the law of $(u(t,x_1),\dots,u(t,x_d))$.
\end{theorem}

This theorem is proved in Section \ref{S:Pos}.

\bigskip
To show our theorems  we need  two auxiliary results, which are interesting by themselves.
So we shall state them explicitly here. The first one is the existence of negative moments
for the solution of the following   SPDE:
\begin{align}\label{E:SPDE-Var}
 \begin{cases}
  \left(\displaystyle\frac{\partial}{\partial t} - \Dxa \right) u(t,x) =
H(t,x) \sigma(t,x,u(t,x))  \dot{W}(t,x),& t>0\;,\: x\in\R,\\[1em]
u(0,\cdot) =\mu(\cdot),
 \end{cases}
\end{align}
which is a key ingredient in the proofs of Theorems \ref{T:Single} and \ref{T:Mult}.

\begin{theorem}
\label{T:NegMom}
Let $u(t,x)$ be the solution to \eqref{E:SPDE-Var} starting from a deterministic and nonnegative measure \textcolor{\myred}{$\mu>0$} that satisfies \eqref{E:InitD}.
Suppose that $H(t,x)$ in \eqref{E:SPDE-Var} is a bounded and adapted process and $\sigma(t,x,z)$ is a measurable and locally bounded function
which is Lipschitz continuous in $z$, uniformly in both $t$ and $x$, satisfying that  $\sigma(t,x,0)=0$.
Let $\Lambda>0$ be a constant such that
\begin{align}\label{E:Lambda}
\left|H(t,x,\omega)\sigma(t,x,z)\right|\le \Lambda |z| \qquad\text{for all $(t,x,z,\omega)\in \R_+\times\R^2\times\Omega$.}
\end{align}
Then for any compact set $K\subseteq\R$ and $t>0$,
there exist finite constant $B>0$ which only depend on $\Lambda$, $\delta$, $K$ and $t>0$ such that for small enough $\epsilon>0$,
\begin{align}\label{E:Rate}
\bbP\left(\inf_{x\in K}u(t,x) <\epsilon \right)&\le \exp\left(-B \left\{|\log(\epsilon)| \log\left(|\log(\epsilon)|\right) \right\}^{2-1/\alpha}\right).
\end{align}
Consequently, for all $p>0$,
\begin{align}\label{E:NonMom1}
\E\left(\left[\inf_{x\in K}u(t,x) \right]^{-p}\right)<+\infty.
\end{align}
\end{theorem}
This theorem is an extension of Theorem 1.4 in \cite{ChenKim14Comparison}; see Remark \ref{R:NegMom} below.
This theorem is proved in Section \ref{S:NegMom} by arguments similar to those in \cite{MuellerNualart08}. The improvement is made through a stopping time argument.

\bigskip
The second result is the following lemma, which is used in the proof of Theorem \ref{T:Single}.
It also illustrates  the subtlety of the measure-valued initial data.
Let $\Norm{\cdot}_p$ denote the $L^p(\Omega)$-norm.

\begin{lemma}\label{L:LimPmom}
Let $u$ be the  solution   with the initial data $\mu$ that satisfies \eqref{E:InitD}.
Suppose there are  $a'< b'$ such that
the measure $\mu$ restricted to $[a',b']$
has a bounded density $f(x)$.  Then for any $a'<a<b<b'$, the following properties hold:
\begin{enumerate}[parsep=0ex,topsep=1ex]
\item[(1)] For all $T>0$, $\sup_{(t,x)\in [0,T]\times [a,b]} \Norm{u(t,x)}_p<+\infty$;
\item[(2)] If $f$ is $\beta$-H\"older continuous on $[a', b']$
for some $\beta\in(0,1)$, then for all $x\in (a,b)$,
 \[
 \Norm{u(t,x)-u(s,x)}_p\le C_{T} |t-s|^{\frac{(\alpha-1)\wedge \beta}{2\alpha}}
 \qquad\text{for all $0\le s\le t\le  T$ and $p\ge 2$}.
 \]
\end{enumerate}
\end{lemma}
As pointed out, e.g., in Example 3.3 and Proposition 3.5 of \cite{ChenDalang14Holder},
without the restriction $x\in(a,b)$, both the limit and the supremum in the above lemma can be equal to infinity.
This lemma is proved in Section \ref{S:LimPmom}.

\bigskip
The paper is organized as follows.  We first give some preliminaries
on the fundamental solutions, stochastic integral and Malliavin calculus
in Section \ref{S:Pre}.
In Section \ref{S:MalliavinD} we study the Malliavin derivatives of $u(t,x)$.
Then we prove Theorem \ref{T:NegMom} on the moments of negative power  in Section \ref{S:NegMom}.
We proceed to prove Lemma \ref{L:LimPmom} in Section \ref{S:LimPmom}.
Then we prove our results on the densities, Theorem \ref{T:Single} (resp. Theorems \ref{T:Mult} and \ref{T:LocDen}),
in Section \ref{S:Main} (resp. Section \ref{S:Mult}).
Finally, the strict positivity of the density, Theorem \ref{T:Pos}, is proved in Section \ref{S:Pos}.
Some technical lemmas are given in Appendix.

\bigskip
Throughout this paper, we  use $C$ to denote a generic constant whose value may vary at different occurrences.

\section{Preliminaries and notation}\label{S:Pre}

\subsection{Fundamental solutions}
Following \cite{ChenDalang15FracHeat}, we denote the fundamental solution to
\eqref{E:FracHt} by $\lMr{\delta}{G}{\alpha}(t,x) $,
or simply $G(t,x)$ since we shall fix  $\alpha$ and $\delta$.
The function $\lMr{\delta}{G}{\alpha}(1,x)$ is the density of the (not necessarily symmetric) $\alpha$-stable distribution.
In particular, when $\alpha=2$ (which forces $\delta=0$),
\begin{align}\label{E:Gaussian}
 \lMr{0}{G}{2}(t,x)=\frac{1}{\sqrt{4\pi t}}\exp\left(-\frac{x^2}{4t}\right).
\end{align}
For general  values of $\alpha$ and $\delta$, there is no explicit expression for $G(t,x)$ in general.  From \eqref{e.1.F-Dxa}  it is easy to see that
it can be equivalently defined through the inverse Fourier transform:
\[
\lMr{\delta}{G}{\alpha}(t,x):=\frac{1}{2\pi}\int_\R\ud \xi \:\exp
\left(
i\xi x -t |\xi|^\alpha e^{-i\pi\delta \sgn(\xi)/2}
\right).
\]
The fundamental solution $G(t,x)$ has the following scaling property
\begin{align}\label{E:ScaleG}
G(t,x) = t^{-1/\alpha} G\left(1,t^{-1/\alpha} x\right).
\end{align}
The following bounds are useful
\begin{align}\label{E:BddG}
\begin{cases}
 \displaystyle \left|G^{(n)}(1,x)\right| \le \frac{K_n}{1+|x|^{1+n+\alpha}},&\text{for $n\ge 0$, if $\alpha\in (1,2]$ and $\delta\le 2-|\alpha|$,}\\[1em]
 \displaystyle G(1,x) \ge \frac{K'}{1+|x|^{1+\alpha}},&\text{if $\alpha\in (1,2)$ and $\delta< 2-|\alpha|$\,. }
\end{cases}
\end{align}
[See (4.2), resp. (5.1), of \cite{ChenDalang15FracHeat} for the upper, resp. lower, bound.]
When we apply space-time convolutions, $G(t,x)$ is understood as $G(t,x)\one_{\{t>0\}}$.
We will use   $\widetilde{G}(t,x)$ to denote the symmetric case (when
$\delta=0$), i.e.,
\[
\widetilde{G}(t,x) := \lMr{0}{G}{\alpha}(t,x).
\]
By the properties of $\lMr{\delta}{G}{\alpha}(t,x)$, we know that for some constants $C_{\alpha,\delta}$ and $C_{\alpha,\delta}'>0$,
\begin{align}\label{E:TildeG}
C_{\alpha,\delta}' \widetilde{G}(t,x)\le
\lMr{\delta}{G}{\alpha}(t,x)\le C_{\alpha,\delta} \widetilde{G}(t,x), \quad\text{for all $t>0$ and $x\in\R$.}
\end{align}
In particular, when $\alpha=2$, the two inequalities in \eqref{E:TildeG} become equality with
$C_{\alpha,\delta}=C_{\alpha,\delta}'=1$.
In \cite{ChenDalang15FracHeat} one can find more properties of $G(t,x)$ related to the calculations in this paper.

\subsection{Some moment bounds and related functions}
Now we define the space-time white noise on a certain  complete probability space $(\Omega,\calF,\bbP)$.   Let
\[
W=\left\{W(A),\: \text{$A$ is a Borel subset of $[0,+\infty)\times\R$ such that $|A|<\infty$}\right\}
\]
be a Gaussian family of random variables with zero mean and covariance
\[
\E\left[W(A)W(B)\right] = |A\cap B|,
\]
where $|A|$ denotes the Lebesgue measure of a Borel sets in $\R^2$.
Let  $\{\calF_t,\, t\ge 0\}$ be the filtration generated by $W$ and augmented by the $\sigma$-field $\calN$ generated
by all $\bbP$-null sets in $\calF$:
\begin{align}\label{E:Ft}
\calF_t = \sigma\left(W([0,s)\times A):0\le s\le
t,A\in\calB_b\left(\R\right)\right)\vee
\calN,\quad t\ge 0,
\end{align}
where $\calB_b\left(\R\right)$ is the set of Borel sets with finite Lebesgue measure.
In the following, we fix this filtered
probability space $\left\{\Omega,\calF,\{\calF_t,t\ge 0\}\:,\bbP\right\}$.
In this setup, $W$ becomes a worthy martingale measure in the sense of Walsh
\cite{Walsh86}. As proved in \cite{ChenDalang13Heat},
for any adapted random field $\left\{X(s,y),\; (s,y)\in\R_+\times\R\right\}$
that is jointly measurable and
\[
\int_0^t \int_\R \E[X(s,y)^2]\ud s\ud y<\infty,
\]
the following stochastic integral
\[
\iint_{[0,t]\times\R}X(s,y) W(\ud s,\ud y)
\]
is well-defined.

For $h\in L^2(\R)$, set $W_t(h):=\iint_{[0,t]\times\R} h(y)W(\ud s,\ud y)$.
Then $\{W_t,t\ge 0\}$ becomes a cylindrical Wiener process in the Hilbert space $L^2(\R)$ with
the following covariance
\begin{align}\label{E:Cov}
\E(W_t(h)W_s(g)) = \min(t,s) \InPrd{h,g}_{L^2(\R)}\qquad\text{for all $h,g\in L^2(\R)$};
\end{align}
see \cite[Chapter 4]{DaPrato}.
With this integral, the solution to \eqref{E:FracHt} is understood in the following {\it mild} sense:
\begin{align}\label{E:mild}
u(t,x)=J_0(t,x)+\iint_{[0,t]\times\R} G(t-s,x-y)\rho(s,y,u(s,y))W(\ud s,\ud y),
\end{align}
where  $J_0(t,x)$ is the solution to the homogeneous equation, i.e.,
\begin{align}\label{E:J0}
J_0(t,x):= (\mu*G(t,\cdot))(x) =\int_\R G(t,x-y)\mu(\ud y).
\end{align}
In \cite{ChenDalang13Heat,ChenDalang15FracHeat}, existence and uniqueness of solutions to \eqref{E:FracHt}
starting from initial data that satisfies \eqref{E:InitD} have been established.
Note that as far as for existence and uniqueness   the change from $\rho(z)$ to $\rho(t,x,z)$ does not pose any problem
because both the Lipschitz continuity and the linear growth condition \eqref{E:LinGrw} in $z$ are uniformly in $t$ and $x$.
The following inequality is a version of the Burkholder-Gundy-Davis inequality (see Lemma 3.7 \cite{ChenDalang13Heat}):
If $X$ is an adapted and jointly measurable random field such that
\[
\int_0^t\int_\R \Norm{X(s,y)}_p^2 \ud s\ud y<\infty,
\]
then
\begin{align}\label{E:BGD}
\Norm{\int_0^t\int_\R X(s,y)W(\ud s\ud y)}_p^2 \le 4 p \int_0^t \int_\R \Norm{X(s,y)}_p^2 \ud s\ud y.
\end{align}
Hence,
\begin{align}\label{E:BGD-U}
 \Norm{u(t,x)}_p^2\le 2 J_0^2(t,x) + 8p \int_0^t \int_\R G(t-s,x-y)^2 \Norm{\rho(s,y,u(s,y))}_p^2 \ud s\ud y,
\end{align}
for all $p\ge 2$, $t>0$ and $x\in\R$.

\bigskip
The moment formula/bounds obtained in \cite{ChenDalang13Heat,ChenDalang15FracHeat} are the key tools in this study.
Here is a brief summary. Denote
\begin{align}\label{E:K}
 \calK_\lambda(t,x) := \sum_{n=0}^\infty \lambda^{2(n+1)} \underbrace{\left(G^2\star\cdots \star G^2\right)}_{\text{$(n+1)$'s}}(t,x),
\end{align}
where ``$\star$'' is convolution in both space and time variables, i.e.,
\[
(h\star g)(t,x)= \int_0^t\ud s \int_\R \ud y \: h(s,y) g(t-s,x-y).
\]
For the heat equation ($\alpha=2$), an explicit formula for this kernel function $\calK_\lambda$ is obtained in \cite{ChenDalang13Heat}
\begin{align}
\label{E:K-2}
\calK_\lambda(t,x) =
\frac1{\sqrt{ 2\pi t}}  \left(\frac{\lambda^2}{\sqrt{8\pi t}}+\frac{\lambda^4}{4}e^{\frac{\lambda^4 t}{8}}\Phi\left(\frac{\lambda^2\sqrt{t}}{2}\right)\right)
\exp\left(-\frac{x^2}{2t} \right),
\end{align}
where  $\Phi(x)=(2\pi)^{-1/2}\int_{-\infty}^x e^{-y^2/2}\ud y$ is
the cumulative distribution function for the standard normal distribution.
Denote
\begin{align}\label{E:calG}
\calG(t,x):=t^{-1/\alpha}G(t,x).
\end{align}
When $\alpha\in (1,2]$ and $|\delta|\le 2-\alpha$,
we have the following upper bound for $\calK_\lambda$  (see \cite[Proposition 3.2]{ChenDalang15FracHeat})
\begin{align}\label{E:K_bound}
\calK_\lambda(t,x)\le C \calG(t,x) \left(1+t^{1/\alpha} e^{C t}\right),
\end{align}
for some constant $C=C(\alpha,\delta,\lambda)$.

One can view Lemma \ref{L:Mom} below as a two-parameter Gronwall's lemma.

\begin{lemma}\label{L:Mom}
If   two functions $f, g : [0,\infty)\times\R\mapsto [0,\infty]$ satisfy the following integral inequality
\begin{align}\label{E:Mom-IntIneq}
g(t,x)\le f(t,x) + \lambda^2 \int_0^t\int_\R G(t-s,x-y)^2 \left[\vv+g(s,y)\right]\ud s\ud y,
\end{align}
for all $t\ge 0$ and $x\in\R$, where $\lambda>0$ and $\vv\ge 0$ are some constants, then for all $t\in [0,T]$ and $x\in\R$,
\[
g(t,x) \le f(t,x) + C_{\lambda,T} \left[\vv+ \int_0^t \int_\R \calG(t-s,x-y) f(s,y)\ud s\ud y\right].
\]
\end{lemma}
\begin{proof}
By the arguments  in \cite{ChenDalang13Heat} and \cite{ChenDalang15FracHeat}, we have
\begin{equation}
g(t,x)\le f(t,x) + ((\vv+f)\star \calK_\lambda)(t,x).
\end{equation}
For $t\in[0,T]$ and $\alpha\in (1,2]$, from \eqref{E:K_bound}, we see that $\calK_\lambda(t,x)\le C t^{-1/\alpha} G(t,x)=C\calG(t,x)$.
Plugging this upper bound in the above inequality proves Lemma \ref{L:Mom}.
\end{proof}
{Under the growth condition \eqref{E:LinGrw}},   we can write \eqref{E:BGD-U}
as the form of \eqref{E:Mom-IntIneq}.  Then  an application of Lemma \ref{L:Mom}
yields  the bounds for
the $p$-th moments ($p\ge 2$)
of the solution:
\begin{align}\label{E:K_p}
\Norm{u(t,x)}_p^2\le 2 J_0^2(t,x)+\left([\varsigma^2+2J_0^2]\star \calK_{2\sqrt{p}\Lip_\rho}\right)(t,x);
\end{align}
see \cite[Theorem 3.1]{ChenDalang15FracHeat}. Then by \eqref{E:K_bound}, for some constant $C$ depending on $\lambda$ and $p$, we have
\begin{align}\label{E:MomentT}
\Norm{u(t,x)}_p^2\le C J_0^2(t,x)+C \int_0^t \int_\R (t-s)^{-1/\alpha}G(t-s,x-y) J_0^2(s,y)\ud s\ud y,
\end{align}
for all $t\in[0,T]$ and $x\in\R$.

We remark that   by setting $C=\Lip_\rho \max(\vv,1)$  the linear growth
condition \eqref{E:LinGrw} may be replaced by a condition of  the following form
\[
|\rho(z)|\le C (1+|z|)\quad\text{for all $z\in\R$}.
\]

\textcolor{\myred}{In the proofs of Theorems \ref{T:Single} and \ref{T:NegMom}, we will apply the strong Markov property for the solution $u(t,x)$. 
For any stopping time $\tau$ with respect to the filtration $\{\calF_t,t\ge 0\}$, the time-shifted white noise,  denoted formally as $\W^*(t,x) = \W(t+\tau,x)$,  is defined as 
\begin{align}\label{E:W*}
W_t^*(h) = W_{t+\tau}(h)-W_t(h),\qquad\text{for all $h\in L^2(\R)$.} 
\end{align}
}
\subsection{Malliavin calculus}\label{SS:Malliavin}
Now we recall some basic facts on  Malliavin calculus associated with $W$.
Denote by $C_p^\infty(\R^n)$ the space of smooth functions with
all their partial derivatives having at most polynomial growth at infinity.
Let $\calS$ be the space of simple functionals of the form
\begin{align}\label{E:Ff}
F=f(W(A_1),\dots,W(A_n)),
\end{align}
where $f\in C_p^\infty(\R^n)$ and
$A_1,\dots, A_n$ are Borel subsets of $[0,\infty)\times\R$ with finite Lebesgue measure.
The derivative of $F$ is a two-parameter stochastic process defined as follows
\[
D_{t,x}F=\sum_{i=1}^n \frac{\partial f}{\partial x_i}\left(W(A_1),\dots,W(A_n)\right)\one_{A_i}(t,x).
\]
In a similar way we define the iterated derivative $D^{k}F$.
The derivative operator $D^{k}$ for positive integers $k\ge 1$ is a closable  operator from $L^p(\Omega)$
into $L^p(\Omega;L^2(([0,\infty)\times\R)^{k})$ for any $p\ge 1$.
Let $k$ be some positive integer.
For any $p>1$, let $\D^{k,p}$ be the completion of $\calS$ with respect to the
norm
\begin{align}
\Norm{F}_{k,p}:=\left(\E\left(|F|^p\right)+\sum_{j=1}^k\E\left[
\left(\int_{([0,\infty)\times\R)^{j}}\left(D_{z_1}\cdots D_{z_j}F\right)^2 \ud
z_1\dots\ud z_j\right)^{p/2}
\right]\right)^{1/p}.
\end{align}
Denote $\D^{\infty}:=\cap_{k,p}\D^{k,p}$.

Suppose that $F=(F^1,\dots,F^d)$ is a $d$-dimensional random vector whose components are in $\D^{1,2}$.
The following random symmetric nonnegative definite matrix
\begin{align}
\sigma_F = \left(
\InPrd{D F^i,DF^j}_{L^2([0,\infty)\times \R)}
\right)_{1\le i,j\le d}
\end{align}
is called the {\it Malliavin matrix} of $F$.
The classical criteria for the existence and regularity of the density are the
following:

\begin{theorem}\label{T:Density}
 Suppose that $F=(F^1,\dots,F^d)$ is a $d$-dimensional random vector whose components are in $\D^{1,2}$. Then
\begin{enumerate}[parsep=0ex,topsep=1ex]
 \item[(1)] If $\det(\sigma_F)>0$ almost surely, the law of $F$ is absolutely continuous with respect to the Lebesgue measure.
 \item[(2)] If $F^i\in\D^\infty$ for each $i=1,\dots,d$ and $\E\left[\left(\det\sigma_F\right)^{-p}\right]<\infty$
 for all $p\ge 1$, then $F$ has a smooth  density.
\end{enumerate}
\end{theorem}

The next lemma gives us a way to establish $F\in\D^\infty$.
\begin{lemma}\label{L:Dinfty0}(Lemma 1.5.3 in \cite{Nualart06})
Let $\{F_m,m\ge 1\}$ be a sequence of random variables converging to $F$ in $L^p(\Omega)$ for some $p>1$. Suppose that $\sup_{m}\Norm{F_m}_{k,p}<\infty$ for some $k\ge 1$. Then $F\in \D^{k,p}$.
\end{lemma}

However, in order to deal with measure-valued initial conditions, we need to extend the above criteria and the lemma as follows; See Remark \ref{R:LocalAnalize} below for the reason that our arguments fail when applying Theorem \ref{T:Density}.
Fix some measurable set $S\subseteq [0,\infty)\times\R$.
For $F$ in the form \eqref{E:Ff}, define the following norm
\begin{align}
\Norm{F}_{k,p,S}:=\left(\E\left(|F|^p\right)+\sum_{j=1}^k\E\left[
\left(\int_{S^{j}}\left(D_{z_1}\cdots D_{z_j}F\right)^2 \ud
z_1\dots\ud z_j\right)^{p/2}
\right]\right)^{1/p},
\end{align}
with respect to which one can define the spaces $\D^{k,p}_S$ as above. 
By convention, when $k=0$, $\Norm{F}_{0,p,S}=\Norm{F}_p$.
Define
$\D^{\infty}_S := \cap_{k,p} \D^{k,p}_S$.
\begin{definition}
We say that a random vector $F=(F^1,\dots,F^d)$ is {\it nondegenerate with
respect to $S$} if it satisfies the following conditions:
\begin{itemize}[parsep=0ex,topsep=1ex]
 \item[(1)] $F^i\in \D^{\infty}_S$ for all $i=1,\dots,d$.
 \item[(2)] The {\it localized  Malliavin matrix}
 \begin{align}
\sigma_{F,S} := \left(
\InPrd{D F^i,DF^j}_{L^2(S)}
\right)_{1\le i,j\le d}
\end{align}
satisfies $\E\left[\left(\det\sigma_{F,S}\right)^{-p}\right]<\infty$ for all
$p\ge 2$.
\end{itemize}
\end{definition}

\begin{theorem}\label{T:Density2}
Suppose that $F=(F^1,\dots,F^d)$ is a random
vector whose components are in $\D^{1,2}$. 
Then
\begin{enumerate}[parsep=0ex,topsep=1ex]
 \item[(1)] If $\det(\sigma_{F,S})>0$ almost surely, the law of $F$ is absolutely
continuous with respect to the Lebesgue measure.
 \item[(2)] If $F^i\in\D^\infty_S$ for each $i=1,\dots,d$ and
$\E\left[\left(\det\sigma_{F,S}\right)^{-p}\right]<\infty$
 for all $p\ge 1$, then $F$ has a smooth  density.
\end{enumerate}
\end{theorem}

\begin{proof}
We will only prove part (2). 
We use $\calB(A)$ to denote all Borel subsets of $A\subseteq(\R^d)$.
Part (1) is a restatement of part (1) of Theorem \ref{T:Density}.
Because $\{W(A), A\in \calB(S)\}$ and $\{W(B), B\in \calB(S^c)\}$ are independent, we may assume that we work on a product space $\Omega_S\times \Omega_{S^c}$. 
Let $D^{S}$ be the Malliavin derivative with respect to $\{W(A), A\in \calB(S)\}$.
For $k\ge 1$ and $p\ge 2$, let $\D^{k,p}_{S,*}$ be the space completed by the {\em smooth cylindrical random variables restricted on $S$} with respect to the norm
\begin{align}
\Norm{F}_{k,p,S,*}:=\left(\E\left(|F|^p\right)+\sum_{j=1}^k\E\left[
\left(\int_{S^{j}}\left(D_{z_1}^S\cdots D_{z_j}^S F\right)^2 \ud
z_1\dots\ud z_j\right)^{p/2}
\right]\right)^{1/p}.
\end{align}
Here, by smooth cylindrical random variables restricted on $S$, we mean any random variable of the following form: 
\[
F=f(W(A_1),\dots,W(A_n)),\quad \text{$A_1,\dots A_n\in \calB(S)$}, 
\]
where $f$ belongs to $C_p^\infty(\R^n)$ ($f$ and all its partial derivatives have polynomial growth).
Similarly, define the space $\D^{\infty}_{S,*}=\cap_{k\ge 1,p\ge 2} \D^{k,p}_{S,*}$.

We first claim that  
\begin{align}\label{E:DoneD}
D^S_z F = \one_S(z) \: D_z F \quad\text{for any $F\in \D^{1,2}$.}
\end{align}
In fact, let $F$ be a smooth and cylindrical random variables without restrictions (i.e., restricted on $[0,\infty)\times\R$) of the following form 
\[
F=f(W(B_1),\dots,W(B_n)), \quad B_1,\dots,B_n\in\calB([0,\infty)\times\R),\: f \in C_p^\infty(\R^n).
\]
Noticing that 
\[
F=f(W(B_1\cap S)+W(B_1\cap S^c),\dots,W(B_n\cap S)+W(B_n\cap S^c)),
\]
we have that
\[
D^S_z F = \sum_{i=1}^n \frac{\partial f}{\partial x_i} 
\left(W(B_1),\dots, W(B_n)\right) \one_{B_i\cap S}(z)
= \one_S(z) \: D_z F.
\]
As a consequence of \eqref{E:DoneD}, we have that for all $F\in \D^{k,p}_S$,
\begin{align}
\Norm{F}_{k,p,S,*}:=\left(\E^S\left(|F|^p\right)+\sum_{j=1}^k\E^S\left[
\left(\int_{S^{j}}\left(D_{z_1}\cdots D_{z_j} F\right)^2 \ud
z_1\dots\ud z_j\right)^{p/2}
\right]\right)^{1/p},
\end{align}
and 
\begin{align}\label{E:TwoNorms}
 \Norm{F}_{k,p,S}^p = \E \left[\Norm{F}_{k,p,S,*}^p\right],
\end{align}
where $\E^{S}$ is the expectation with respect to $\Omega_{S}$.

Now we claim that $F^j(\cdot,\omega')\in \D_{S,*}^\infty$ for almost all $\omega'\in \Omega_{S^c}$.
Actually, since $F\in \D^{k,p}_S$, one can find a sequence of smooth and cylindrical random variables $\{F^j_{n},n\ge 1\}$ such that $F^j_{n}$ converges to $F^j$ in the norm $\Norm{\cdot}_{k,p,S}$. By \eqref{E:TwoNorms}, one can find a subsequence $\{F^j_{n_k},k\ge 1\}$ such that $F^j_{n_k}$ converges to $F^j$ in the norm $\Norm{\cdot}_{k,p,S,*}$ almost surely on $\Omega_{S^c}$ as $k\rightarrow\infty$. Since the operator $D^S$ is closable from $L^p(\Omega_S)$ to $L^p(\Omega_S;L^2(S))$, one can conclude that $F^j\in \D_{S,*}^{k,p}$ almost surely on $\Omega_{S^c}$.
Finally, because $k\ge 1$ and $p\ge 2$ are arbitrary, this claim follows.

The condition $\E\left[(\det \sigma_{F,S})^{-p}\right]<\infty$ implies that $\E^{S}\left[(\det \sigma_{F,S})^{-p}\right]<\infty$ for almost all $\omega'\in \Omega_{S^c}$. Now we can apply the usual criterion, i.e., part (2) of Theorem \ref{T:Density}, to obtain a smooth density for $F(\cdot,\omega')$ for almost all $\omega'$. Finally, integrating with respect to $\omega'$ produces the desired smooth density for $F$.
This completes the proof of Theorem \ref{T:Density2}.
\end{proof}

The following lemma provides us some sufficient conditions for proving $F\in\D^{n,p}_S$. It is an extension of Lemma \ref{L:Dinfty0}. 
We need to introduce some operators.
Notice that one can obtain the orthogonal decomposition $L^2(\Omega_S)=\oplus_{n=0}^\infty \calH_{n,S}$ where $\calH_{n,S}$ is the $n$-th Wiener Chaos associated to the Gaussian family $\{W(A),A\in\calB(S)\}$. Let $J_{n,S}$ be the orthogonal projection on the $n$-th Wiener chaos. Define the {\em Ornstein-Uhlenbeck semigroup restricted on $S$} on $L^2(\Omega_S)$ by 
\[
T_{t,S}(F) =\sum_{n=0}^\infty  e^{-nt} J_{n,S} F.
\]
Or one can equivalently define this operator through {\it Mehler's formula} (see, e.g., \cite[Proposition 3.1]{Nualart09}):
\[
T_{t,S}(F) =\E'\left(F\left(e^{-t}W+\sqrt{1-e^{-2t}}W'\right)\right),\quad\text{for all $F\in L^2(\Omega_S)$},
\]
where $W = \{W(A),A\in\calB(S)\}$, $W'$ is an independent copy of $W$, and $\E'$ is the expectation with respect to $W'$.

\begin{lemma}\label{L:Dinfty}
Let $\{F_m,m\ge 1\}$ be a sequence of random variables converging to $F$ in
$L^p(\Omega)$ for some $p>1$. Suppose that $\sup_{m}\Norm{F_m}_{n,p,S}<\infty$
for some integer $n\ge 1$. Then $F\in\D_S^{n,p}$.
\end{lemma}
\begin{proof}
Fix some integer $n\ge 1$. Set $\calH=L^2(S)$.
The condition $\sup_{m}\Norm{F_m}_{n,p,S}<\infty$ implies that
$\{D^k F_m, m\ge 1\}$ is bounded in $L^p(\Omega;\calH^{\otimes k})$ for all $k=1,\dots, n$.
Hence, one can find a subsequence $\{F_{m_i},i\ge 1\}$ such that 
$D^kF_{m_i}$ converges weakly to $\Theta_k\in L^p(\Omega;\calH^{\otimes k})$ for all $k=1,\dots,n$, i.e., for all $\beta_k\in L^q(\Omega;\calH^{\otimes k})$, $1/p+1/q=1$, 
\begin{align}\label{E:betaDF}
\lim_{i\rightarrow\infty} 
\E\left[\InPrd{\beta_k, D^kF_{m_i}}_{\calH^{\otimes k}}\right]
= \E\left[\InPrd{\beta_k,\Theta_k}_{\calH^{\otimes k}}\right],
\quad\text{for $k=1,\dots, n$.}
\end{align}
In the following, we will use the sequence $\{F_m,m\ge 1\}$ itself to denote the subsequence $\{F_{m_i},i\ge 1\}$ for simplicity.

Let $t>0$. It is clear that $T_{t,S}F_m$ converges to $T_{t,S}F$ in $L^p(\Omega)$ as $m\rightarrow\infty$. 
On the other hand, by Proposition 3.8 of \cite{Nualart09}, for all $k\in\{1,\dots,n\}$, 
\[
\Norm{\: \Norm{D^{k,S}\left(T_{t,S}F_n-T_{t,S}F_m\right)}_{\calH^{\otimes k}}}_p
\le C_p t^{-k/2} \Norm{F_n-F_m}_p\rightarrow 0\quad\text{as $n,m\rightarrow\infty$}.
\]
Hence, one can conclude that $T_{t,S}F\in \D^{n,p}_S$ and 
\begin{align}\label{E:DkSTtS}
D^{k,S} T_{t,S}F = \lim_{m\rightarrow\infty}D^{k,S} T_{t,S} F_m \qquad\text{in $L^p(\Omega)$ for all $k=1,\dots,n.$}
\end{align}

By Proposition 3.7 of \cite{Nualart09}, we have that $D^{k,S}  T_{t,S}  F_m= e^{-k t}T_{t,S}  D^{k,S}  F_m$ for all $k=1,\dots,n$. 
Then, for all $\beta_k\in L^q(\Omega;\calH^{\otimes k})$, $1/p+1/q=1$, $k=1,\cdots, n$, we have that
\begin{eqnarray*}
\E\left[\InPrd{D^{k,S}T_{t,S} F, \beta_k}_{\calH^{\otimes k}}\right] 
&\stackrel{\eqref{E:DkSTtS}}{=}&\lim_{m\rightarrow\infty}
\E\left[\InPrd{D^{k,S}T_{t,S} F_m, \beta_k}_{\calH^{\otimes k}}\right] \\
&=&
\lim_{m\rightarrow\infty}
\E\left[\InPrd{e^{-kt}T_{t,S} D^{k,S} F_m, \beta_k}_{\calH^{\otimes k}}\right] \\
&=&
\lim_{m\rightarrow\infty}
\E\left[\InPrd{D^{k,S} F_m, e^{-kt}T_{t,S} \beta_k}_{\calH^{\otimes k}}\right] \\
&\stackrel{\eqref{E:betaDF}}{=}&
\E\left[\InPrd{\Theta_k, e^{-kt}T_{t,S} \beta_k}_{\calH^{\otimes k}}\right] \\
&=&
\E\left[\InPrd{e^{-kt}T_{t,S}\Theta_k,  \beta_k}_{\calH^{\otimes k}}\right] .
\end{eqnarray*}
Therefore, 
\[
D^{k,S} T_{t,S}F = e^{-kt}T_{t,S} \Theta_k,\quad\text{for all $k=1,\dots,n$.}
\]
Because $T_{1/m,S}F$ converges to $F$ as $m\rightarrow \infty$ in $L^p(\Omega)$ and 
\[
D^{k,S} T_{1/m,S} F =e^{-k/m} T_{1/m,S}\Theta_k \rightarrow \Theta_k \quad\text{as $m\rightarrow\infty$ in $L^p(\Omega)$ for all $k=1,\dots,n$,}
\]
we see that $D^{k,S}F=\Theta_k\in L^p(\Omega;\calH^{\otimes k})$ for all $k=1,\dots,n$, which proves that $F\in\D_S^{n,p}$.
\end{proof}

Finally, we refer to Nualart \cite{Nualart06} for a complete presentation of the Malliavin calculus and its applications and to Hu \cite{Hu16Book} for the general analysis on Gaussian space.


\section{Nonnegative moments: Proof of Theorem \ref{T:NegMom}}\label{S:NegMom}

In this section, we will prove Theorem \ref{T:NegMom}, which is an improvement of part (2) of the following theorem. Part (1) of this theorem will also be used in our arguments.
\begin{theorem}[Theorem 1.4 of \cite{ChenKim14Comparison}]
\label{T:NegMom-old}
Let $u(t,x)$ be the solution to \eqref{E:SPDE-Var} starting from a deterministic and nonnegative measure $\mu$ that satisfies \eqref{E:InitD}. Then we have the following two statements:\\
(1) If $\mu>0$, then for any compact set $K\subseteq \R_+^*\times\R$,
there exists some finite constant $B>0$ which only depends on $K$ such that for small enough $\epsilon>0$,
\begin{align}\label{E:Rates1}
P\left(\inf_{(t,x)\in K}u(t,x) <\epsilon \right)&\le \exp\left(-B |\log(\epsilon)|^{1-1/a} \log\left(|\log(\epsilon)|\right)^{2-1/a}\right).
\end{align}
(2) If  $\mu(\ud x)=f(x)\ud x$ with $f\in C(\R)$, $f(x)\ge 0$ for all $x\in\R$ and $\spt{f}\neq \emptyset$,
then for any compact set $K\subseteq\spt{f}$ and any $T>0$,
there exists finite constant $B>0$ which only depends on $K$ and $T$ such that for all small enough $\epsilon>0$,
\begin{align}
P\left(\inf_{(t,x)\in \: ]0,T]\times K}u(t,x) <\epsilon\right)\le \exp\left(-B \left\{ |\log(\epsilon)|\cdot \log\left(|\log(\epsilon)|\right)\right\}^{2-1/a}\right).
\end{align}
\end{theorem}
\begin{remark}\label{R:NegMom}
As a consequence of  part (2) of Theorem \ref{T:NegMom-old}, we have that 
\begin{align}\label{E_:PinfK}
P\left(\inf_{x\in K}u(T,x) <\epsilon\right)\le \exp\left(-B \left\{ |\log(\epsilon)|\cdot \log\left(|\log(\epsilon)|\right)\right\}^{2-1/a}\right).
\end{align}
Theorem \ref{T:NegMom} improves both \eqref{E_:PinfK} and \eqref{E:Rates1} by both removing the condition $K\subseteq \spt{f}$ and increasing the power of $\log(1/\epsilon)$ from $1-1/\alpha$ to $2-1/\alpha$, respectively. This is achieved through a stopping time argument.
\end{remark}

Let $\Phi_\alpha(x)$ be the cumulative distribution function of $G(1,x)$, i.e.,
$\Phi_\alpha(x):=\int_{-\infty}^x G(1,y)\ud y$.
We need a slight extension of the weak comparison principle in \cite{ChenKim14Comparison}
with $\sigma(u)$ replaced by $H(t,x)\sigma(t,x,u)$.
The proof of following lemma \ref{L:WeakComp} is similar to that in \cite{ChenKim14Comparison}. We leave the details to the interested readers.

\begin{lemma}[Weak comparison principle]\label{L:WeakComp}
Let $u_i(t,x)$, $i=1,2$, be two solutions to \eqref{E:SPDE-Var} starting from initial measures $u_{0,i}$ that satisfy \eqref{E:InitD}, respectively.
If $u_{0,1}-u_{0,2}\ge 0$, then
\[
\bbP\Big(u_1(t,x)\ge u_2(t,x)\:\:\text{for all $t\ge 0$ and $x\in\R$}\Big)=1.
\]
\end{lemma}

The following lemma is used to initialize the iteration in the proof of Theorem \ref{T:NegMom}.

\begin{lemma}\label{L:1Init}
 For any $\alpha\in [1,2]$, $a<b$, $\gamma\ge 0$ and $T>0$ with $b-a>\gamma T$,  it holds that
 \begin{align*}
 0<\inf_{0\le t+s\le T} \inf_{a-\gamma(t+s)\le x\le b+\gamma (t+s)}(\one_{[a-\gamma s,b+\gamma s]}*G(t,\cdot))(x)
  \le 1.
 \end{align*}
\end{lemma}
\begin{proof}
The upper bound is true for all $t>0$ and $x\in\R$: $(\one_{[a-\gamma s,b+\gamma s]}*G(t,\cdot))(x)\le \int_\R G(t,x-y)\ud y=1$.
As for the lower bound,
fix $t,s\ge 0$ such that $t+s\le T$. Notice that
\[
(\one_{[a-\gamma s,b+\gamma s]}*G(t,\cdot))(x) =
\Phi_\alpha\left(\frac{x-a+\gamma s}{t^{1/\alpha}}\right)
-\Phi_\alpha\left(\frac{x-b-\gamma s}{t^{1/\alpha}}\right).
\]
For any $x\in [a-\gamma(t+s),b+\gamma(t+s)]$, the above quantity achieves the minimum at one of
the two end points, i.e.,
\begin{align*}
(\one_{[a-\gamma s,b+\gamma s]}*G(t,\cdot))(x) \ge&
\min\Bigg\{
\Phi_\alpha\left(-\gamma t^{1-1/\alpha}\right)
-\Phi_\alpha\left(\frac{(a-b)-\gamma t-2\gamma s}{t^{1/\alpha}}\right),\\
&
\hspace{3em}\Phi_\alpha\left(\frac{(b-a)+\gamma t+2\gamma s}{t^{1/\alpha}}\right)
-\Phi_\alpha\left(\gamma t^{1-1/\alpha}\right)
\Bigg\}\\
\ge &
\min\Bigg\{
\Phi_\alpha\left(-\gamma t^{1-1/\alpha}\right)
-\Phi_\alpha\left(\frac{a-b}{t^{1/\alpha}}\right),\\
&
\hspace{3em}\Phi_\alpha\left(\frac{b-a}{t^{1/\alpha}}\right)
-\Phi_\alpha\left(\gamma t^{1-1/\alpha}\right)
\Bigg\}\\
\ge &
\min\Bigg\{
\Phi_\alpha\left(-\gamma T^{1-1/\alpha}\right)
-\Phi_\alpha\left(\frac{a-b}{T^{1/\alpha}}\right),\\
&
\hspace{3em}\Phi_\alpha\left(\frac{b-a}{T^{1/\alpha}}\right)
-\Phi_\alpha\left(\gamma T^{1-1/\alpha}\right)
\Bigg\}>0,
\end{align*}
where the last inequality is due to $\gamma T<b-a$.
Note that in the above inequalities, we have used the fact that
$\alpha\ge 1$ (so that $1-1/\alpha\ge 0$).
This proves Lemma \ref{L:1Init}.
\end{proof}

\bigskip
\begin{proof}[Proof of Theorem \ref{T:NegMom}]
Let $\calF_t$ be the natural filtration generated by the white noise (see \eqref{E:Ft}). 
Fix an arbitrary compact set $K\subset\R$ and let $T>0$.
We are going to prove Theorem \ref{T:NegMom} for $\inf_{x\in K} u(T,x)$ in two steps.

{\bigskip\bf\noindent Case I.~}
In this case, we make the following assumption: 
\begin{itemize}
 \item[(H)] Assume that for some nonempty interval $(a,b)$ and some nonnegative function $f$ the initial measure $\mu$ satisfies that $\mu(\ud x) = f(x)\ud x$. Moreover, for some $c>0$, $f(x)\ge c \one_{(a,b)}(x)$ for all $x\in\R$. 
\end{itemize}
Thanks to the weak comparison principle (Lemma \ref{L:WeakComp}), we may assume that $f(x)= \one_{[a,b]}(x)$.
Choose $\gamma\ge 0$ such that $K\subseteq\calI_T$. 
For any $t>0$, denote
\[
\calI_t:=[ a-\gamma t,\:b+\gamma t].
\]
Denote
\[
\beta:=\frac{1}{2}\: \inf_{0\le t+s\le T} \: \inf_{x\in \calI_{t+s}}(\one_{[a-\gamma s,b+\gamma s]}*G(t,\cdot))(x) .
\]
By Lemma \ref{L:1Init}, we see that $\beta\in (0,1/2)$.

Define a sequence $\{\:T_n, n\ge 0\}$ of $\{\calF_t,t\ge 0\}$-stopping times as follows:
let $T_0\equiv 0$ and
\[
T_k:=\inf\left\{
t>T_{k-1}: \inf_{x\in\calI_t} u(t,x)\le \beta^k
\right\},\qquad k\ge 1,
\]
where we use the convention that $\inf\emptyset =\infty$.
By these settings, one can see that $T_1>0$ a.s.

Let $\W_k(t,x)=\W(t+T_{k-1},x)$ be the time-shifted space-time white noise (see \eqref{E:W*}) and
similarly, let $H_k(t,x)=H(t+T_{k-1},x)$.
For each $k$, let $u_k(t,x)$ be the unique solution to \eqref{E:SPDE-Var} starting from
$u_k(0,x)= \beta^{k-1} \one_{[a-\gamma T_{k-1},b+\gamma T_{k-1}]}(x)$.
Then
\[
w_k(t,x):= \beta^{1-k}u_k(t,x)
\]
solves the following SPDE
\begin{align*}
 \begin{cases}
\left(\displaystyle\frac{\partial}{\partial t} - \Dxa \right) w_k(t,x) =
H_k(t,x)\sigma_k(t,x,w_k(t,x))  \dot{W}_k(t,x),& t>0\;,\: x\in\R,\cr
w_k(0,x) = \one_{[a-\gamma T_{k-1},b+\gamma T_{k-1}]}(x),
 \end{cases}
\end{align*}
with $\sigma_k(t,x,z):=\beta^{1-k}\sigma(t,x,\beta^{k-1} z)$.
Note that $\sigma_k$ has the same Lipschitz constant as $\sigma$ in the third argument.
Let $S(t_1,t_2)$ be a space-time cone defined as
\[
S(t_1,t_2):=
\Big\{
  (t,x): \: t\in [0,t_2-t_1], \: x\in [a-\gamma(t_1+t),b+\gamma(t_1+t)]
\Big\}.
\]
Set
\[
\tau_n:=\frac{2T}{n},
\]
and define the events
\[
\mathfrak{D}_{k,n}:=\left\{
T_{k}-T_{k-1}\le \tau_n
\right\},\quad\text{for $1\le k\le n$.}
\]
By the definition of the stopping times $T_k$,
\[
\beta^{1-k}u(T_{k-1},x) \ge \beta^{1-k}u_k(0,x) =w_k(0,x)\quad
\text{for all $x\in\R$, a.s. on $\{T_{k-1}< +\infty\}$, }
\]
for all $k\ge 1$.
Therefore, by the strong Markov property and the weak comparison principle
(Lemma \ref{L:WeakComp}), we see that on $\{T_n<T\}$, for $k\le n$,
\[
\bbP\left(\mathfrak{D}_{k,n}\middle|\calF_{T_{k-1}}\right)=
\bbP\left(\mathfrak{D}_{k,n}\bigcap \left\{ \sup_{(t,x)\in S(T_{k-1},T_k)} \left|w_k(t,x)-w_k(0,x)\right|\ge 1-\beta\right\}\middle|\calF_{T_{k-1}}\right).
\]
Write $w_k(t,x)$ in the mild form
\begin{align*}
w_k(t,x)=&\quad (\one_{[a-\gamma T_{k-1},b+\gamma T_{k-1}]}*G(t,\cdot))(x)\\
& + \int_0^t \int_\R H_k(s,y)\sigma_k(s,y,w_k(s,y))G(t-s,x-y) W_k(\ud s,\ud y)\\
=: & J_k(t,x)+I_k(t,x).
\end{align*}
Notice that
\[
\left|w_k(t,x)-w_k(0,x)\right|\le
\left|J_k(t,x)-\one_{[a-\gamma T_{k-1},b+\gamma T_{k-1}]}(x)\right|+
\left|I_k(t,x)\right|.
\]
Hence, on $\{T_n<T\}$, for $k\le n$,
\begin{align*}
 &\hspace{-4em}
 \bbP\left(\mathfrak{D}_{k,n}\bigcap\left\{\sup_{(t,x)\in S(T_{k-1},T_k)} \left|w_k(t,x)-w_k(0,x)\right|\ge 1-\beta\right\} \middle| \calF_{T_{k-1}} \right)\\
 \le&\quad \bbP\left(\sup_{(t,x)\in S(T_{k-1},T_k)} \left|J_k(t,x)-\one_{[a-\gamma T_{k-1},b+\gamma T_{k-1}]}(x)\right|> 1-2\beta \:\middle|
 \calF_{T_{k-1}}\right)
 \\
 &+
 \bbP\left(\mathfrak{D}_{k,n}\bigcap\left\{\sup_{(t,x)\in S(T_{k-1},T_k)} \left|I_k(t,x)\right|\ge \beta \right\}\middle|\calF_{T_{k-1}}\right).
\end{align*}
By Lemma \ref{L:1Init},
\[
\sup_{(t,x)\in S(T_{k-1},T_k)} \left|J_k(t,x)-\one_{[a-\gamma T_{k-1},b+\gamma T_{k-1}]}(x)\right|\le
1- 2\beta,\quad
\text{a.s. on $\{T_{n}\le T\}$.}
\]
Hence, an application of Chebyschev's inequality shows that,
on $\{T_n<T\}$, for $k\le n$,
\begin{align}\notag
 &\hspace{-3em}\bbP\left(\mathfrak{D}_{k,n}\bigcap\left\{\sup_{(t,x)\in S(T_{k-1},T_k)} \left|w_k(t,x)-w_k(0,x)\right|\ge 1-\beta \right\} \middle| \calF_{T_{k-1}}\right)\\
 \notag
 &\le
 \bbP\left(\mathfrak{D}_{k,n}\bigcap\left\{\sup_{(t,x)\in S(T_{k-1},T_k)} \left|I_k(t,x)\right|\ge \beta\right\} \middle|\calF_{T_{k-1}}\right)\\
 &\le \notag
 \bbP\left(\sup_{(t,x)\in [0,\tau_n]\times \calI_T} \left|I_k(t,x)\right|\ge \beta \middle|\calF_{T_{k-1}}\right)\\
&\le
 \beta^{-p}\:
 \E\left[\sup_{(t,x)\in [0,\tau_n]\times \calI_T } \left|I_k(t,x)\right|^p\middle| \calF_{T_{k-1}}\right]
 \label{E_:bE},
\end{align}
where we have used the fact that
$S(T_{k-1},T_k)\subseteq [0,\tau_n]\times \calI_T$ a.s. on $\mathfrak{D}_{k,n}$.

\bigskip
Next we will find a deterministic upper bound for the conditional expectation in \eqref{E_:bE}.
By the Burkholder-Davis-Gundy inequality  and Proposition 4.4 of \cite{ChenDalang15FracHeat},
for some universal constant $C_1>0$ (see Proposition 4.4 of \cite{ChenDalang15FracHeat} for the value of this constant), it holds that
\[
\E\left[|I_k(s,x)-I_k(s',x)|^p\right]\le C_1  \left(|s-s'|^{\frac{\alpha-1}{2\alpha}}\right)^{p/2} \sup_{(t,y)\in [0,\tau_n]\times\R} \|w_k(t,y)\|_p^p\:,
\]
for all $(s,s',x)\in [0,\tau_n]^2\times\calI_T$.
Because $\sigma(t,x,0)=0$, Lemma \ref{L:WeakComp} and Lemma 4.2 in \cite{ChenKim14Comparison} together imply that
\[
\sup_{(t,y)\in [0,\tau_n]\times\R} \|w_k(t,y)\|_p^p \le Q^p \exp\left(Q p^{\frac{2\alpha-1}{\alpha-1}} \tau_n\right)=:C_{p,\tau_n},
\]
for some constant $Q:=Q(\Lambda)$ and for all $p\ge 2$.
Then by the Kolmogorov continuity theorem (see Theorem \ref{T:KolCont}),
for some constant $C>0$ and for all $0<\eta< 1-\frac{2(\alpha+1)}{p(\alpha-1)}$, we have that
\begin{align}
\E\left[
\sup_{(t,x)\in [0,\tau_n]\times \calI_T }\left|\frac{I_k(t,x)}{\tau_n^{\frac{\alpha-1}{2\alpha}\eta}}\right|^p\right]&\le
\E\left[
\sup_{(s,s',x')\in [0,\tau_n]^2\times \calI_T}\left|\frac{I_k(s,x)-I_k(s',x)}{|s-s'|^{\eta\frac{\alpha-1}{2\alpha}}}\right|^p\right]
\le C^p C_{p,\tau_n}\,,
\label{E_:Holder}
\end{align}
where we have used the fact that $I(0,x)\equiv 0$ for all $x\in\R$,  a.s.

We are interested in the case where $p=O\left([n\log n]^{1-1/\alpha}\right)$ as $n\rightarrow \infty$ (see \eqref{E:pA} below).
In this case, we have $p^{\alpha/(\alpha-1)}\tau_n=O(\log n)$ as $n\rightarrow\infty$ since $\tau_n=2T/n$.
This implies that there exists some constant $Q':=Q'(\beta, \Lambda, T)$ such that
\begin{align*}
\beta^{-p}\: \E\left[
\sup_{(s,x)\in [0,\tau_n]\times\calI_T}\left|I_k(s,x)\right|^p\right]
&\le
Q'\, \tau_n^{\frac{(\alpha-1)\eta}{2\alpha}p}\, \exp\left(Q' p^{\frac{2\alpha-1}{\alpha-1}} \tau_n\right)\\
&=
Q' \exp\left(Q' p^{\frac{2\alpha-1}{\alpha-1}} \tau_n+\frac{(\alpha-1)\eta}{2\alpha}p \log(\tau_n)\right).
\end{align*}
By denoting $\eta = \theta \left(1-\frac{2}{p}\:
\frac{\alpha+1}{\alpha-1}\right)$ with $\theta \in (0,1)$, the above exponent becomes
\[
f(p):=Q' \tau_n  p^{\frac{2\alpha-1}{\alpha-1}}+\frac{\log (\tau_n )\:\theta\: \left(p[\alpha-1]-2 [\alpha+1]\right)}{2 \alpha}.
\]
Some calculations show that $f(p)$ for $p\ge 2$ is minimized at
\begin{align*}
p= & \left(\frac{(\alpha-1)^2 \theta  \log (1/\tau_n )}{2 \alpha
   (2 \alpha-1) Q' \tau_n }\right)^{1-1/\alpha}=
\left(\frac{(\alpha-1)^2 \theta \: n \log(n/(2T))}{4 \alpha (2 \alpha-1) Q' T }\right)^{1-1/\alpha}.
\end{align*}
Thus, for some constants $A:=A(\beta, \Lambda,T)$ and $Q'':=Q''(\beta, \Lambda, T)$,
\begin{align}\label{E:pA}
\min_{p\ge 2} f(p) \le f(p') = - Q'' n^{1-1/\alpha} [\log(n)]^{2-1/\alpha}\qquad\text{with $p'= A  \left[n \log(n)\right]^{1-1/\alpha}.$}
\end{align}
Therefore, for some finite constant $Q:=Q(\beta,\Lambda,T)>0$,
\begin{equation}\label{E:CondProb}
\bbP\left(\mathfrak{D}_{k,n} \middle| \calF_{T_{k-1}}\right)
\leq Q \exp\left(-Q \,n^{(\alpha-1)/\alpha}(\log n)^{(2\alpha-1)/\alpha}\right).
\end{equation}
Note that the above upper bound is deterministic.

\bigskip
For convenience, assume that $n=2m$ is even.
Let $\Xi_n\subseteq\mathbb{N}^m$ be defined as follows
\[
(i_1,\dots,i_m)\in\Xi_n \quad\text{if and only if}\quad
1\le i_1<i_2<\dots<i_m\le n.
\]
The cardinality of $\Xi_n$ satisfies that $|\Xi_n|=\binom{n}{m}\le 2^n$. Then
\begin{align*}
\bbP\left(\inf_{x\in K}u(T,x)\le  \beta^{n}\right)& \le \bbP\left(\inf_{x\in \calI_T} u(T,x)\le  \beta^{n}\right)\\
&\le \bbP\left(T_n\leq T\right)\\
&\le \bbP\left(\bigcup_{(i_1,\dots,i_m)\in\Xi_n}\bigcap_{j=1}^m \mathfrak{D}_{i_j,n}\right)
\\ &
\le
\sum_{(i_1,\dots,i_m)\in\Xi_n} \bbP\left(\bigcap_{j=1}^m \mathfrak{D}_{i_j,n}\right).
\end{align*}
Applying \eqref{E:CondProb} for $m$ times, we see that
\begin{align*}
\bbP\left(\bigcap_{j=1}^m \mathfrak{D}_{i_j,n}\right) & =
\E\left[\prod_{j=1}^{m-1} \Indt{\mathfrak{D}_{i_j,n}}\E\left(\Indt{\mathfrak{D}_{i_m,n}} \middle| \calF_{T_{i_m-1}}\right)\right]\\
&\le
Q \exp\left(-Q \,n^{(\alpha-1)/\alpha}(\log n)^{(2\alpha-1)/\alpha}\right)
\bbP\left(\bigcap_{j=1}^{m-1} \mathfrak{D}_{i_j,n}\right)\\
&\le
Q^m \exp\left(-mQ \,n^{(\alpha-1)/\alpha}(\log n)^{(2\alpha-1)/\alpha}\right).
\end{align*}
Hence,
\begin{align*}
\bbP\left(\inf_{x\in K}u(T,x)\le  \beta^{n}\right)\le
2^n Q^m \exp\left(-m Q \, n^{(\alpha-1)/\alpha}(\log n)^{(2\alpha-1)/\alpha}\right).
\end{align*}
Finally, by adapting the value of $Q$ in the above inequality and setting $\epsilon=\beta^n$, we can conclude that for some constant $B=B(\Lambda,T)>0$ such that for $\epsilon>0$ small enough,
\begin{align}\label{E:CaseI}
\bbP\left(\inf_{x\in K} u(T,x) <\epsilon \right)\le \exp\left(-B(\Lambda,T)\left\{|\log(\epsilon)| \log\left(|\log(\epsilon)|\right) \right\}^{2-1/\alpha}\right).
\end{align}

{\bigskip\bf\noindent Case II.~}
Now we consider the general initial data. Choose and fix an arbitrary $\theta\in (0,T\wedge 1)$. Set 
\[
\Theta(\omega):=1\wedge \inf_{x\in (a,b)}  u(\theta,x,\omega).
\]
Since $u(\theta,x)>0$ for all $x\in\R$ a.s. (see Theorem \ref{T:NegMom-old}) and  $u(\theta,x,\omega)$ is continuous in $x$, we see that 
$\Theta>0$ a.s.
Hence, 
$u(\theta,x,\omega)\ge \Theta(\omega) \one_{(a,b)}(x)$ for all $x\in\R$.
Denote $V(t,x,\omega):=\Theta(\omega)^{-1} u(t+\theta,x,\omega)$. By the Markov property, $V(t,x)$ solves the following time-shifted SPDE 
\begin{align}
 \begin{cases}
\left(\displaystyle\frac{\partial}{\partial t} - \Dxa \right) V(t,x) =
\widetilde{H}(t,x) \widetilde{\sigma}(t,x,u(t,x))  \dot{W}_\theta(t,x),& t>0\;,\: x\in\R,\\[1em]
V(0,x) = \Theta^{-1} u(\theta,x),
 \end{cases}
\end{align}
where $\W_\theta (t,x)=\W(t+\theta,x)$ (see \eqref{E:W*}), $\widetilde{H}(t,x) = H(t+\theta,x)$ and 
\[
\widetilde{\sigma}(t,x,z,\omega)=\Theta(\omega)^{-1}\sigma\left(t+\theta,x,\Theta(\omega) z\right).
\]
A key observation is that condition \eqref{E:Lambda} is satisfied by $\widetilde{H}$ and $\widetilde{\sigma}$ with the same constant $\Lambda$, that is, 
\[
\left|\widetilde{H}(t,x,\omega)\widetilde{\sigma}(t,x,z,\omega)\right|\le \Lambda |z| \qquad\text{for all $(t,x,z,\omega)\in \R_+\times\R^2\times\Omega$.}
\]
The initial data $V(0,x)$ satisfies assumption (H) in Case I. 
Hence, we can conclude from \eqref{E:CaseI} that for $\epsilon>0$ small enough,
\begin{align*}
\bbP\left(\inf_{x\in K} u(T+\theta,x) <\epsilon \right) & \le
\bbP\left(\inf_{x\in K} V(T,x) <\epsilon \right)\\
&\le \exp\left(-B(\Lambda,T)\left\{|\log(\epsilon)| \log\left(|\log(\epsilon)|\right) \right\}^{2-1/\alpha}\right).
\end{align*}
Finally, using the fact that $(t,x)\mapsto u(t,x)$ is continuous a.s., by 
letting $\theta$ go to zero, we can conclude that \eqref{E:Rate} holds for $\epsilon>0$ small enough.
As a direct consequence of \eqref{E:Rate} and Lemma \ref{L:NegMom}, we establish the existence of the negative moments.
This completes the proof of Theorem \ref{T:NegMom}.
\end{proof}

\section{Proof of Lemma \ref{L:LimPmom}} \label{S:LimPmom}

We first prove a lemma.

\begin{lemma}\label{L:J0f}
If for some interval $[a',b']\subset \R$, $a'<b'$, the measure $\mu$ restricted to this interval
has a density $f(x)$ with $f(x)$ being $\beta$-H\"older continuous
for some $\beta\in (0,1)$, then for any interval $[a,b]\subset (a',b')$, $a<b$, there exists some finite constant $C:=C(a,b,\beta)>0$, it holds that
 \[
 \left|\int_\R G(t,x-y) \mu(\ud y) -f(x)\right| \le C t^{\beta/\alpha} \qquad\text{for all $t>0$ and $x\in (a,b)$.}
 \]
\end{lemma}
\begin{proof}
Fix $x\in(a,b)$. Notice that
\begin{align*}
\left|\int_\R G(t,x-y) \mu(\ud y) -f(x)\right| \le &
\quad \int_{a'}^{b'} G(t,x-y)\left|f(y)-f(x)\right|\ud y\\
 &+\int_{[a',b']^c} G(t,x-y) |\mu|(\ud y)+ |f(x)|\int_{[a',b']^c} G(t,x-y) \ud y\\
 =:& I_1+I_2+ I_3.
\end{align*}
By the H\"older continuity of $f$ and properties of $G(t,x)$ (see \eqref{E:ScaleG} and \eqref{E:BddG}),
\begin{align*}
I_1 &\le C \int_{a'}^{b'} G(t,x-y) |y-x|^\beta \ud y\\
&\le C \int_\R t^{-1/\alpha}G\left(1,\frac{y}{t^{1/\alpha}}\right) |y|^\beta \ud y\\
&= C \int_\R t^{-\beta/\alpha}G\left(1,z\right) |z|^\beta \ud z\\
&\le C t^{\beta/\alpha} \int_\R \frac{|z|^\beta}{1+|z|^{1+\alpha}}\ud z\\
&= C t^{\beta/\alpha},
\end{align*}
where in the last equality we have used the fact that $\alpha>1\ge \beta$.
Note that the above equality is also true for $\alpha=2$.
As for $I_3$, by the properties of $G(t,x)$, for all $\alpha\in (1,2]$,
\[
I_3\le |f(x)|\int_{[a',b']^c} \frac{C t}{t^{1+1/\alpha}+|x-y|^{1+\alpha}} \ud y
\le C |f(x)| t \int_{[a',b']^c} \frac{\ud y}{|x-y|^{1+\alpha}}
=C_{a,b} |f(x)| t,
\]
where in the last inequality we have used the fact that $x\in (a,b)$.
Similarly,
\[
I_2\le \int_{[a',b']^c} \frac{C t}{t^{1+1/\alpha}+|x-y|^{1+\alpha}} |\mu|(\ud y)
\le C t \int_{[a',b']^c} \frac{|\mu|(\ud y)}{|x-y|^{1+\alpha}}
=C_{a,b} t.
\]
This proves Lemma \ref{L:J0f}.
\end{proof}

\bigskip
\begin{proof}[Proof of Lemma \ref{L:LimPmom}]
Fix $x\in [a,b]$ and $T>0$.
Let $f(x)$, supported on $[a,b]$, be the density satisfying that $\mu (\ud x) \one_{[a,b]} = f(x)\ud x$.
Denote $\hat{\mu}:=\mu-f=\mu \one_{[a',b']^c}$.
Let $\Lip_\rho$ be the Lipschitz constants of $\rho$.
We claim that 
\begin{align}\label{E:Claim1}
\limsup_{t\rightarrow 0}\Norm{u(t,x)}_p <+\infty,
\quad\text{for all $x\in [a,b]$ and $p\ge 2$.}
\end{align}
Then part (1) is a direct consequence of \eqref{E:Claim1}.
Because the conditions in \eqref{E:InitD} for the cases $\alpha=2$ and $\alpha\in (1,2)$ are different in nature,
we need to consider the two cases separately.

{\bigskip\noindent\bf Case I.~}
When $\alpha\in(1,2)$, by \eqref{E:MomentT},
\begin{align*}
\Norm{u(t,x)}_p^2
\le C |J_0(t,x)|^2 + C R(t,x),
\end{align*}
for all $t>0$ and $x\in\R$, where $J_0(t,x)=(\mu* G(t,\cdot))(x)$ and
\begin{align}\label{E:Rtx}
R(t,x)=\int_0^t \ud s\int_\R \ud y\:
|J_0(s,y)|^2 \frac{G(t-s,x-y)}{(t-s)^{1/\alpha}}.
\end{align}
To prove \eqref{E:Claim1}, we may assume that $t\in(0,1)$.
It is clear that $\lim_{t\rightarrow 0} |J_0(t,x)| = |f(x)|$ for a.e. $x\in [a,b]$.
By \eqref{E:ScaleG} and \eqref{E:BddG},
\begin{align}\label{E:Gbdd}
G(t,x)\le \frac{C \: t}{t^{1+1/\alpha}+|x|^{1+\alpha} } .
\end{align}
Thus,
\begin{align}\notag
|J_0(s,y)| &\le \frac{1}{s^{1/\alpha}} \int_\R \frac{C'}{1+|s^{-1/\alpha} (y-z)|^{1+\alpha} }|\mu|(\ud z)
\\
&\le
\frac{1}{s^{1/\alpha} }\sup_{y\in\R}\int_\R \frac{C'}{1+|y-z|^{1+\alpha} }|\mu|(\ud z)
=: \frac{C}{s^{1/\alpha}},
\label{E:J0Bd}
\end{align}
where the second inequality is due to the fact that $s\in (0,1)$.
So,
\begin{align}
|J_0(s,y)|^2 &\le 2 \left(\int_\R G(t,y-z)|\hat{\mu}|(\ud z)\right)^2 +2 \left(\int_\R G(t,y-z)|f(z)| \ud z\right)^2\notag\\
&\le \frac{C}{s^{1/\alpha}}\int_\R G(t,y-z)|\hat{\mu}|(\ud z) + 2\Norm{f}_{L^\infty(\R)}^2 .\label{E:BoundJ0}
\end{align}
Hence, by \eqref{E:Gbdd},
\begin{align}\notag
R(t,x)\le & C \int_0^t \frac{\ud s}{[(t-s)s]^{1/\alpha} }\int_\R \ud y \: G(t-s,x-y) \int_\R |\hat{\mu}|(\ud z) G(s,y-z)
+ C t^{1-1/\alpha}  \notag \\
\notag
\le &C \: t^{1-2/\alpha}
\int_{\R} \frac{ t}{t^{1+1/\alpha}+|x-z|^{1+\alpha}}|\hat{\mu}|(\ud z)+ C t^{1-1/\alpha}\\
\le
&
 C \int_{[a',b']^c}\frac{1}{|x-z|^{1+\alpha}}|\mu|(\ud z)  \: t^{2(1-1/\alpha)}+ C t^{1-1/\alpha}
\label{E:BoundR}
\\
\rightarrow &\: 0,\quad\text{as $t\rightarrow 0$,}
\notag
\end{align}
where the integral on the right-hand of \eqref{E:BoundR} is finite because $x\in [a,b]\subset[a',b']$.
Therefore, for all $p\ge 2$ and a.e. $x\in [a,b]$,
\[
\Norm{u(t,x)}_p^2 \le C\: \left( |f(x)|^2 + t^{2(1-1/\alpha)}+t^{1-1/\alpha}\right)\rightarrow C f(x)^2,\quad\text{as $t\rightarrow 0$,}
\]
which proves both \eqref{E:Claim1} and part (1) for $\alpha\in(1,2)$.

As for part (2),
we need only consider the case that $s=0$. The case when $s>0$ is covered by Theorem 1.6 of \cite{ChenKim14Comparison}.
By the Burkholder-Davis-Gundy inequality \eqref{E:BGD-U}, for $x\in(a,b)$,
\begin{align}\label{E_:UF}
\Norm{u(t,x)-f(x)}_p^2
\le& C \left|J_0(t,x)-f(x)\right|^2 + C\int_0^t\ud r\int_\R \ud y
\: \Norm{u(r,y)}_p^2 G(t-r,x-y)^2.
\end{align}
By Lemma \ref{L:Mom} and \eqref{E:BoundR},
\[
\int_0^t\ud r\int_\R \ud y
\: \Norm{u(r,y)}_p^2 G(t-r,x-y)^2\le
C' R(t,x) \le C t^{1-1/\alpha}.
\]
Then applying  Lemma \ref{L:J0f} to bound the first term on the right-hand side of \eqref{E_:UF} yields the lemma.

{\bigskip\noindent\bf Case II.~}
If $\alpha=2$, $G(t,x)$ is the Gaussian kernel; see \eqref{E:Gaussian}.
Apply Lemma 3.9 (or equation (3.12)) in \cite{ChenDalang13Heat} with $\nu=2$ and $\lambda=\Lip_\rho$ to obtain that
\begin{align}\label{E_:NpL}
\Norm{u(t,x)}_p^2\le C \left[(|\mu|*G(2t,\cdot))(x)\right]^2
\end{align}
for all $t\in [0,T]$.
Now it is clear that the right-hand side of \eqref{E_:NpL} converges to $C|f(x)|^2$ for a.e. $x\in[a,b]$ as $t\rightarrow 0$.
This proves \eqref{E:Claim1} and hence part (1).
As for part (2),
{the case $s>0$ is covered by Theorem 3.1 of \cite{ChenDalang14Holder}.
Now we consider the case when $s=0$.
Notice that $G(t,x)^2=(8\pi t)^{-1/2}G(t/2,x)$.}
Hence, by \eqref{E_:NpL} and Lemma 3.9 of \cite{ChenDalang13Heat},
\[
\int_0^t\ud r\int_\R \ud y
\: \Norm{u(r,y)}_p^2 G(t-r,x-y)^2\le C \:t^{1/2}.
\]
Then we may apply the same arguments as those in Case I.
This completes the proof of Lemma \ref{L:LimPmom}.
\end{proof}

From the above proof for part (1) of Lemma \ref{L:LimPmom}, one can see that we have actually proved the following slightly stronger  result. Recall that the function $\calG$ is defined in \eqref{E:calG}.
\begin{lemma}\label{L:LimPmom2}
Suppose there are  $a'< b'$ such that
the initial measure $\mu$ restricted to $[a',b']$
has a bounded density $f(x)$.  Then for any $a'<a<b<b'$, 
and for all $T>0$, 
\begin{align}
\sup_{(t,x)\in [0,T]\times [a,b]} \left(J_0^2(t,x)+(J_0^2\star\calG)(t,x)\right)<+\infty.
\end{align}

\end{lemma}

\section{Malliavin derivatives of \texorpdfstring{$u(t,x)$}{TEXT}}
\label{S:MalliavinD}
In this section, we will prove that $u(t,x)\in \D^{1,p}$ for all $p\ge 1$ and
that $D_{\theta,\xi}u(t,x)$ satisfies certain stochastic integral equation.
When the initial data is bounded, one can find a proof, e.g.,
in \cite[Proposition 2.4.4]{Nualart06} and \cite[Proposition 5.1]{NualartQuer07}.
For higher derivatives, we will show that under usual conditions on $\rho$ and if one can find some subset $S\subseteq [0,T]\times\R$ such that \eqref{E:supSNormU} below is satisfied, then we can establish the property that $u(t,x)\in \D^{k,p}_S$ for all $k\ge 1$ and $p\ge 2$.
Denote $\calH_T:=L^2([0,T]\times\R)$. 

\begin{proposition}\label{P:D1}
Suppose that $\rho$ is a $C^1$ function with bounded Lipschitz
continuous derivative. Suppose that the initial data $\mu$ satisfies
\eqref{E:InitD}. Then
\begin{enumerate}[parsep=0ex,topsep=1ex]
\item[(1)] For any
$(t,x)\in [0,T]\times \R$,
$u(t,x)$ belongs to $\D^{1,p}$ for all $p\ge 1$.
\item[(2)] The Malliavin derivative $Du(t,x)$ defines an $\calH_T$-valued
process that satisfies the following linear stochastic differential equation
\begin{equation}\label{E:SPDEM}
\begin{aligned}
D_{\theta,\xi}u(t,x) =&
\quad\rho(u(\theta,\xi)) G(t-\theta,x-\xi)\\
&+
\int_\theta^t \int_\R G(t-s,x-y)\rho'(u(s,y))D_{\theta,\xi}u(s,y) W(\ud s, \ud
y),
\end{aligned}
\end{equation}
for all $(\theta,\xi)\in[0,T]\times \R$.
\item[(3)] If $\rho\in C^\infty(\R)$ and it has bounded derivatives of all orders,
and if for some measurable set $S\subset [0,t]\times \R$ the initial data satisfies the following condition
\begin{align}\label{E:supSNormU}
\sup_{(t,x)\in S} \left(J_0^2(t,x)+(J_0^2\star\calG )(t,x)\right) <\infty,
\end{align}
then $u(t,x)\in \D_S^\infty$; recall that $\calG$ is defined in \eqref{E:calG} and $\D_S^\infty$ is defined in Section \ref{SS:Malliavin}.
\end{enumerate}
\end{proposition}

\begin{remark}\label{R:D1}
We first give some sufficient conditions for \eqref{E:supSNormU}:
\begin{enumerate}[parsep=0ex,topsep=1ex]
\item[(1)] If there exists some compact set $K\subset\R$ such that the Lebesgue measure of $K$ is strictly positive and the initial data restricted on $K$ has a bounded density,
then by Lemma \ref{L:LimPmom2} and \eqref{E:calG}, condition \eqref{E:supSNormU} is satisfied for $S=[0,t]\times K$. This is how we choose $S$ in the proof of a special case of Theorem \ref{T:Single} -- Theorem \ref{T_:Single} by utilizing the properness of the initial data.
\item[(2)] If $S$ is a compact set on $[0,T]\times\R$ that is away from $t=0$, then condition \eqref{E:supSNormU} is trivially satisfied. In the proof of Theorems \ref{T:Mult} and \ref{T:LocDen}, this $S$ is chosen to be $[t/2,t]\times K$ where $K$ is some compact set in $\R$. 
\end{enumerate}
\end{remark}

In the following, we first establish parts (1) and (2) of Proposition \ref{R:D1}.
The proof of part (3) is more involved. We need to introduce some notation and prove some lemmas.
Then we prove part (3). At the end of this section, we point out the reason that we need to resort to the localized Malliavin analysis in Remark \ref{R:LocalAnalize}.

\begin{proof}[Proof of parts (1) and (2) of Proposition \ref{P:D1}]
Fix $p\ge 2$.
Consider the Picard approximations $u_m(t,x)$ in the proof of the existence of the random field solution in \cite{ChenDalang13Heat,ChenDalang15FracHeat}, i.e.,
$u_0(t,x)=J_0(t,x)$, and for $m\ge 1 $,
\begin{align}\label{E:Picard}
u_m(t,x)= J_0(t,x) + \int_0^t\int_\R G(t-s,x-y)\rho(u_{m-1}(s,y))W(\ud s\ud y).
\end{align}
It is proved in \cite{ChenDalang13Heat,ChenDalang15FracHeat} that
$u_m(t,x)$ converges to $u(t,x)$ in $L^p(\Omega)$ as $m\rightarrow\infty$ and
\begin{align}\label{E:unBds}
\Norm{u_m(t,x)}_p^2 \le 2 J_0^2(t,x)+
2\Lip_\rho^2 \left([1+J_0^2] \star \calK_\lambda \right)(t,x),
\end{align}
for all $t\in [0,T]$, $x\in\R$ and $m\in\bbN$,
where the kernel function $\calK_\lambda(t,x)$ is defined in \eqref{E:K} and $\lambda=2\sqrt{2}\Lip_\rho$.

We claim that for all $m\ge 0$ and for all $t\in [0,T]$ and $x\in\R$, 
\begin{align}\label{E:DunBds}
 \E\left(\Norm{D u_m(t,x)}_{\calH_T}^p\right)
 \le C
 \left[\left(
 [1+J_0^2] \star \calK_\lambda \right) (t,x)\right]^{p/2}.
\end{align}
It is clear that $0\equiv D u_0(t,x)$ satisfies \eqref{E:DunBds}.
Assume that $D u_k(t,x)$ satisfies \eqref{E:DunBds} for all $k<m$.
Now we shall show that $D u_m(t,x)$ satisfies the moment bound in \eqref{E:DunBds}.
Notice that
\begin{align}\notag
D_{\theta,\xi}u_m(t,x)
= &  \quad G(t-\theta,x-\xi)\rho\left(u_{m-1}(\theta,\xi)\right) \\
\label{E:DunSPDE}
& +
\int_\theta^t\int_\R G(t-s,x-y)\rho'\left(u_{m-1}(s,y)\right)
D_{\theta,\xi}u_{m-1}(s,y)W(\ud s\ud y)\\
\notag
=:& A_1 + A_2.
\end{align}
We first consider $A_1$. It is clear that by \eqref{E:unBds},
\begin{align*}
\E\left(\Norm{A_1}_{\calH_T}^p\right) =&
\E\left(\left[
\int_0^t\int_\R G(t-\theta,x-\xi)^2 \rho(u_{m-1}(\theta,\xi))^2\ud \theta\ud \xi
\right]^{p/2}\right)\\
\le&C \left(
\int_0^t\int_\R \lambda^2 G(t-\theta,x-\xi)^2 \left[1+ \Norm{u_{m-1}(\theta,\xi)}_p^2\right]\ud \theta\ud \xi
\right)^{p/2}\\
\le&C \left[\left(\left[1+J_0^2+\left(1+J_0^2\right)\star\calK_\lambda\right]\star \lambda^2 G^2
\right)(t,x)\right]^{p/2}.
\end{align*}
By the following recursion formula which is clear from the definition of $\calK_\lambda$ in \eqref{E:K},
\begin{align}\label{E:K_Recursion}
\left(\calK_\lambda\star \lambda^2 G^2\right)(t,x) = \calK_\lambda(t,x) - \lambda^2 G^2(t,x),
\end{align}
we see that
\[
\E\left(\Norm{A_1}_{\calH_T}^p\right) \le C \left[\left([1+J_0^2]\star\calK_\lambda\right)(t,x)\right]^{p/2}.
\]
As for $A_2$, by the Burkholder-Davis-Gundy inequality and by the boundedness of $\rho'$,
\begin{align*}
\E\left(\Norm{A_2}_{\calH_T}^p\right) \le &
C\; \E\left(\left[
\int_0^t\int_\R G(t-s,x-y)^2\Norm{D u_{m-1}(s,y)}_{\calH_T}^2\ud s\ud y
\right]^{p/2}\right)\\
\le &
C \left(
\int_0^t\int_\R \lambda^2 G(t-s,x-y)^2\Norm{\Norm{D u_{m-1}(s,y)}_{\calH_T}}_p^2 \ud s\ud y
\right)^{p/2}.
\end{align*}
Then by the induction assumption \eqref{E:DunBds},
\begin{align*}
\E\left(\Norm{A_2}_{\calH_T}^p\right) &\le
C \left[\left(
([1+J_0^2]\star\calK_\lambda)\star\lambda^2 G^2
\right)(t,x)\right]^{p/2}\\
&=
C \left[\left(
[1+J_0^2]\star\calK_\lambda
\right)(t,x)-\left(
[1+J_0^2]\star\lambda^2 G^2
\right)(t,x)\right]^{p/2}\\
&\le
C \left[\left(
[1+J_0^2]\star\calK_\lambda
\right)(t,x)\right]^{p/2},
\end{align*}
where we have applied \eqref{E:K_Recursion}.
Combining these two bounds shows that $D u_m(t,x)$ satisfies the moment bound in \eqref{E:DunBds}.
Therefore, 
\[
\sup_{m\in\bbN}\E\left(\Norm{Du_m(t,x)}_{\calH_T}^p\right)<\infty.
\]
By Lemma \ref{L:Dinfty0}, we can conclude that $u(t,x)\in \D^{1,p}$.
This proves part (1) of Proposition \ref{P:D1}.

\bigskip
Now we shall show that $D_{\theta,\xi}u(t,x)$ satisfies \eqref{E:SPDEM}.
Lemma 1.2.3 of \cite{Nualart06} implies that $D_{\theta,\xi}u_m(t,x)$ converges
to $D_{\theta,\xi}u(t,x)$
in the weak topology of $L^2\left(\Omega;\calH_T\right)$, namely, for any $h\in \calH_T$ and
any square integrable random variable $F\in \calF_t$,
\[
\lim_{n\rightarrow\infty}
\E\left(\InPrd{D_{\theta,\xi}u_m(t,x)-D_{\theta,\xi}u(t,x),h}_{\calH_T}F\right)=0.
\]
We need to show that the right-hand side of \eqref{E:DunSPDE} converges to the right-hand side of \eqref{E:SPDEM}
in this weak topology of $L^2\left(\Omega;\calH_T\right)$ as well.
Notice that by the Cauchy-Schwarz inequality, 
\begin{multline*}
\E\left(\left|\InPrd{\left(\rho(u)-\rho(u_m)\right)G(t-\cdot,x-\cdot),h}_{\calH_T}F\right|\right)
\\
\le \LIP_\rho\Norm{h}_{\calH_T}\Norm{F}_2
\left(
\int_0^t\int_\R \Norm{u(s,y)-u_m(s,y)}_2^2 G^2(t-s,x-y)\ud s\ud y
\right)^{1/2}.
\end{multline*}
We need to find bounds for $\Norm{u(t,x)-u_m(t,x)}_2^2$.
Let $\lambda=\max(\Lip_\rho,\LIP_\rho)$.
It is clear that
\begin{align*}
\Norm{u(t,x)-u_0(t,x)}_2^2 \le &
\Lip_\rho^2 \int_0^t\int_\R G^2(t-s,x-y) \Norm{u(s,y)}_2^2\ud s\ud y\\
\le & (\lambda^2 G^2 \star (J_0^2 + (1+J_0^2)\star\calK_\lambda))(t,x)\\
\le & ((1+J_0^2)\star\calK_\lambda)(t,x),
\end{align*}
where in the last step we have applied \eqref{E:K_Recursion}.
For $m\ge 1$, we have that
\begin{align*}
\Norm{u(t,x)-u_m(t,x)}_2^2 \le &
\LIP_\rho^2 \int_0^t\int_\R G^2(t-s,x-y) \Norm{u(s,y)-u_{m-1}(s,y)}_2^2\ud s\ud y\\
\le & (\lambda^2 G^2 \star \Norm{u-u_{m-1}}_2^2)(t,x)\\
\le & (\calL_m \star \Norm{u-u_0}_2^2)(t,x),
\end{align*}
where
\[
\calL_m(t,x) = \underbrace{\left(\lambda^2 G^2 \star \cdots \star \lambda^2 G^2\right)}_{\text{$m$'s $\lambda^2 G^2$}}(t,x).
\]
By Proposition 4.3 of \cite{ChenDalang15FracHeat}, we see that 
\[
\calL_m(t,x) \le \frac{C^m}{\Gamma(m(1-1/\alpha))}  G(t,x)
\le \frac{C^m}{\Gamma(m(1-1/\alpha))}  \lambda^2 G^2(t,x),
\]
for $t\in[0,T]$ and $x\in\R$.
Hence,
\begin{align*}
\int_0^t\int_\R &\Norm{u(s,y)-u_m(s,y)}_2^2 G^2(t-s,x-y)\ud s\ud y\\
\le&
C \left((1+J_0^2)\star\calK_\lambda \star\calL_m\star \lambda^2 G^2\right)(t,x)\\
\le &
\frac{C^m}{\Gamma(m(1-1/\alpha))}
\left((1+J_0^2)\star\calK_\lambda \star \lambda^2 G^2 \star \lambda^2 G^2\right)(t,x)\\
\le & \frac{C^m}{\Gamma(m(1-1/\alpha))}
\left((1+J_0^2)\star\calK_\lambda\right)(t,x),
\end{align*}
where in the last step we have applied twice the inequality $(\lambda^2G^2\star\calK_\lambda)(t,x)\le\calK_\lambda(t,x)$;
see \eqref{E:K_Recursion}.
By Stirling's approximation, one sees that 
\begin{align}\label{E:CG-Lim}
 \lim_{m\rightarrow\infty}C^m\Gamma(m(1-1/\alpha))^{-1}=
 0.
\end{align}
Therefore,
\[
\lim_{m\rightarrow\infty}
\E\left(\InPrd{\left(\rho(u)-\rho(u_m)\right)G(t-\cdot,x-\cdot),h}_{\calH_T}F\right)=0.
\]

Denote the second term on the right-hand side of \eqref{E:DunSPDE} (resp. \eqref{E:SPDEM}) by
$I_m(t,x)$ (resp. $I(t,x)$). It remains to show that
\begin{align}
\lim_{m\rightarrow\infty}\E\left(\InPrd{I(t,x)-I_m(t,x),h}_{\calH_T} F\right)=0.
\end{align}
Notice that
\begin{align*}
I(t,x)&-I_m(t,x)\\
=&\int_0^t\int_\R G(t-s,x-y)\left(\rho'(u(s,y))-\rho'(u_{m-1}(s,y))\right)
D_{\theta,\xi}u_{m-1}(s,y)W(\ud s\ud y)\\
&+\int_0^t\int_\R G(t-s,x-y)\rho'(u(s,y))\left(D_{\theta,\xi}u(s,y)-D_{\theta,\xi}u_{m-1}(s,y))\right)
W(\ud s\ud y)\\
=:& B_1(t,x) + B_2(t,x).
\end{align*}
Since $F$ is square integrable, it is known that for some adapted random field
$\{\Phi(s,y),s\in[0,T],y\in\R \}$ with $\int_0^t\int_\R \E\left[\Phi^2(s,y)\right]\ud s\ud y<\infty$ it holds that 
\[
F= \E[F] + \int_0^t\int_\R \Phi(s,y)W(\ud s\ud y);
\]
see, e.g., Theorem 1.1.3 of \cite{Nualart06}.
Hence,
\begin{align*}
\E\left(\InPrd{B_1(t,x),h}_{\calH_T}F\right)
=&\int_0^t\ud s \int_\R\ud y\: 
G(t-s,x-y)\\
&\times
\E\left[\Phi(s,y)\left(\rho'(u(s,y))-\rho'_{m-1}(u(s,y))\right)
\InPrd{Du_{m-1}(s,y),h}_{\calH_T}\right].
\end{align*}
Note that it suffices to consider the case when  $\Phi$ is bounded uniformly in
$(t,x,\omega)$ since these random fields are dense in the set of all adapted
random fields such that $\int_0^T\int_\R \E[\Phi^2(s,y)]\ud s\ud y$ is finite.
By the Cauchy-Schwarz inequality,
\begin{align*}
\E\left(\left|\InPrd{B_1(t,x),h}_{\calH_T}F\right|\right)
\le& \LIP_{\rho'} \Norm{\Phi}_{L^\infty(\R_+\times\R\times\Omega)} \Norm{h}_{\calH_T}
\int_0^t\int_\R G(t-s,x-y)\\
&\times 
\E\left(
|u(s,y)-u_{m-1}(s,y)| \Norm{Du_{m-1}(s,y)}_{\calH_T}
\right)
\ud s\ud y\\
\le& C \int_0^t\int_\R G(t-s,x-y)\\
&\times
\Norm{u(s,y)-u_{m-1}(s,y)}_2\:
\E\left(
\Norm{Du_{m-1}(s,y)}_{\calH_T}^2
\right)^{1/2}
\ud s\ud y.
\end{align*}
By the same arguments as above, we see that 
\begin{align*}
\Norm{u(s,y)-u_m(s,y)}_2^2
& \le
\frac{C^m}{\Gamma(m(1-1/\alpha))}
\left((1+J_0^2)\star\calK_{\lambda}\right)(s,y).
\end{align*}
Together with \eqref{E:DunBds}, we have that
\begin{align*}
\E\left(\left|\InPrd{B_1(t,x),h}_{\calH_T}F\right|\right)^2
\le & C^{m/2}\Gamma\left(m(1-1/\alpha)\right)^{-1/2}\\
&\times \int_0^t \ud s\int_\R\ud y\:
G(t-s,x-y)\left((1+J_0^2)\star\calK_{\lambda}\right)(s,y).
\end{align*}
By \eqref{E:K_bound}, we see that
\[
\left((1+J_0^2)\star\calK_{\lambda}\right)(s,y)
\le C \left((1+J_0^2)\star\calG\right)(s,y),
\]
where $\calG(t,x)$ is defined in \eqref{E:calG}.
Thanks to \eqref{E:CG-Lim}, we only need to prove that
\begin{align*}
M:=\int_0^t \ud s\int_\R\ud y\: G(t-s,x-y) \left((1+J_0^2)\star\calG\right)(s,y)<\infty.
\end{align*}
Notice that $M=\left((1+J_0^2)\star \calG\star G\right)(t,x)$ and for all $t\in [0,T]$, 
\[
(G\star\calG)(t,x)
=\int_0^t \ud s \: s^{-1/\alpha}\int_\R \ud y
\: G(s,y) G(t-s,x-y)
\le C T^{1-1/\alpha} G(t,x)
\le C \calG(t,x).
\]
Hence, $M\le C ((1+J_0^2)\star\calG)(t,x)$, which is clearly finite.  Therefore, we can conclude that
\[
\lim_{m\rightarrow\infty}
\E\left(\InPrd{B_1(t,x),h}_{\calH_T}F\right)=0.
\]

Similarly, for $B_2(t,x)$, we have that
\begin{align*}
\E\left(\InPrd{B_2(t,x),h}_{\calH_T}F\right)
=&\int_0^t\ud s \int_\R\ud y\: 
G(t-s,x-y)\\
&\times
\E\left[\Phi(s,y)\rho'(u(s,y))\InPrd{D(u(s,y)-u_{m-1}(s,y)),h}_{\calH_T}\right].
\end{align*}
The boundedness of $\rho'$ implies that $\Phi(s,y)\rho'(u(s,y))$ is again an $\calF_s$-measurable and square integral
random variable.
Since $D u_{m}(s,y)$ converges to $D u(s,y)$ in the weak topology as specified above, we have that
\[
\lim_{m\rightarrow\infty}
\E\left[\Phi(s,y)\rho'(u(s,y))\InPrd{D(u(s,y)-u_{m-1}(s,y)),h}_{\calH_T}\right] =0.
\]
Then an application of the dominated convergence theorem implies that
\[
\lim_{m\rightarrow\infty}
\E\left(\InPrd{B_2(t,x),h}_{\calH_T}F\right)=0.
\]
This proves part (2) of Proposition \ref{P:D1}.
\end{proof}

\bigskip
In order to prove part (3) of the proposition, we need several lemmas. We first introduce some notation. Let $\Lambda(n,k)$, $n\ge k\ge 1$, be the set of partitions of the integer $n$ of length $k$, that is, if $\lambda \in\Lambda(n,k)$, then $\lambda\in \bbN^k$, and by writing $\lambda=(\lambda_1,\dots,\lambda_k)$ it satisfies $\lambda_1\ge \dots\ge \lambda_k\ge 1$ and $\lambda_1+\dots+\lambda_k=n$.
For $\lambda\in \Lambda(n,k)$, let $\calP(n,\lambda)$ be all partitions 
of $n$ ordered objects, say $\{\theta_1,\dots,\theta_n\}$ with $\theta_1\ge \cdots \ge \theta_n$, into $k$ groups $\{\theta^1_1,\dots,\theta^1_{\lambda_1}\},\dots, \{\theta^k_1,\dots,\theta^k_{\lambda_k}\}$  such that within each group the elements are ordered, i.e., $\theta_1^j\ge \dots\ge \theta^j_{\lambda_j}$ for $1\le j\le k$.
It is clear that the cardinality of the set $\calP(n,\lambda)$ is equal to $\binom{n}{\lambda_1,\dots,\lambda_k}=\frac{n!}{\lambda_1!\dots\lambda_k!}$.

\begin{lemma}\label{L:Leibniz}
For $n\ge 1$, if $\rho\in C^n(\R)$ and $u\in \D^{n,2}$, then
\begin{align}\label{E:Leibniz}
 D^n_{\theta_1,\dots,\theta_n} \rho(u) = \sum_{k=1}^n \rho^{(k)}(u)\sum_{\lambda\in \Lambda(n,k)}
 \sum_{\calP(n,k)}
 \prod_{j=1}^k
  D^{\lambda_j}_{\theta_1^j,\dots,\theta_{\lambda_j}^j} u.
\end{align}
\end{lemma}

\begin{remark}
Lemma \ref{L:Leibniz} is simply a consequence of the Leibniz rule of differentials. We will leave the proof to the interested readers. Here we list several special cases instead:
\begin{align*}
D_{\theta_1}\rho(u) = &\rho'(u) \: D_{\theta_1}u, \\
D_{\theta_1,\theta_2}^2\rho(u) = &\rho''(u)\: D_{\theta_1}u\: D_{\theta_2}u + \rho'(u)\: D^2_{\theta_1,\theta_2}u,\\
D_{\theta_1,\theta_2,\theta_3}^3\rho(u) = &\rho^{(3)}(u) \: D_{\theta_1}u\: D_{\theta_2}u\: D_{\theta_3}u \\
&+ \rho^{(2)}(u)
\left(D^2_{\theta_1,\theta_2}u\: D_{\theta_3} u +D^2_{\theta_1,\theta_3}u\:  D_{\theta_2} u +D^2_{\theta_2,\theta_3}u \: D_{\theta_1} u \right)\\
& + \rho'(u)\: D^3_{\theta_1,\theta_2,\theta_3}u,\\
D_{\theta_1,\theta_2,\theta_3,\theta_4}^4\rho(u) = &\rho^{(4)}(u) \: D_{\theta_1}u\: D_{\theta_2}u\: D_{\theta_3}u\: D_{\theta_4}u \\
&+ \rho^{(3)}(u)
\Bigg[D^2_{\theta_1,\theta_2}u D_{\theta_3} uD_{\theta_4} u +D^2_{\theta_1,\theta_3}u  D_{\theta_2} u D_{\theta_4} u +D^2_{\theta_1,\theta_4}u D_{\theta_2} uD_{\theta_3} u\\
& \hspace{4em}+D^2_{\theta_2,\theta_3}u D_{\theta_1} uD_{\theta_4} u
+D^2_{\theta_2,\theta_4}u D_{\theta_1} uD_{\theta_3} u+D^2_{\theta_3,\theta_4}u D_{\theta_1} uD_{\theta_2} u\Bigg]\\
& + \rho^{(2)}(u)\Bigg[ \Big( 
D^3_{\theta_1,\theta_2,\theta_3}uD_{\theta_4}u+
D^3_{\theta_1,\theta_2,\theta_4}uD_{\theta_3}u+
D^3_{\theta_1,\theta_3,\theta_4}uD_{\theta_2}u+
D^3_{\theta_2,\theta_3,\theta_4}uD_{\theta_1}u
\Big)\\
&\hspace{4em}+
\Big( 
D^2_{\theta_1,\theta_2}uD^2_{\theta_3,\theta_4}u+
D^2_{\theta_1,\theta_3}uD^2_{\theta_2,\theta_4}u+
D^2_{\theta_1,\theta_4}uD^2_{\theta_2,\theta_3}u
\Big)
\Bigg]\\
&+\rho'(u)\: D_{\theta_1,\theta_2,\theta_3,\theta_4}^4 u.
\end{align*}
\end{remark}

\begin{lemma}\label{L:LeibnizMoment}
Suppose $\rho\in C^\infty(\R)$ with all derivatives bounded and $u$ is a smooth and cylindrical random variable.
Let $\calH$ be the corresponding Hilbert space. Then 
\begin{align}\label{E:LeibnizMoment}
\Norm{\: \Norm{D^n\rho(u)}_{\calH^{\otimes n}}}_p^2
\le C_n \sum_{k=1}^n \sum_{\lambda\in\Lambda(n,k)}
\binom{n}{\lambda_1,\dots,\lambda_k}
\prod_{j=1}^k \Norm{\:\Norm{D^{\lambda_j} u}_{\calH^{\otimes \lambda_j}}}_{kp}^2.
\end{align}
\end{lemma}
\begin{proof}
By Lemma \ref{L:Leibniz} and by the boundedness of derivatives of $\rho$, we see that
\begin{align*}
 \Norm{D^n\rho(u)}_{\calH^{\otimes n}}
 &\le C \sum_{k=1}^n \sum_{\lambda\in\Lambda(n,k)}
 \binom{n}{\lambda_1,\dots,\lambda_k}
 \prod_{j=1}^k \Norm{D^{\lambda_j} u}_{\calH^{\otimes\lambda_j}},
\end{align*}
where we have used the fact that $|\calP(n,\lambda)|=\binom{n}{\lambda_1,\dots,\lambda_k}$.
Hence, by the Minkowski inequality and the H\"older inequality, 
\begin{align*}
\Norm{\: \Norm{D^n\rho(u)}_{\calH^{\otimes n}}}_p
 &\le C \sum_{k=1}^n \sum_{\lambda\in\Lambda(n,k)}\binom{n}{\lambda_1,\dots,\lambda_k}\prod_{j=1}^k \Norm{\:\Norm{D^{\lambda_j}u}_{\calH^{\otimes\lambda_j}}}_{kp}.
\end{align*}
Finally, an application of the inequality $(a_1+\dots+a_n)^2\le n(a_1^2+\dots+a_n^2)$ completes the proof of Lemma \ref{L:LeibnizMoment}.
\end{proof}

\begin{lemma}\label{L:DI}
Suppose $\rho\in C^\infty$ with all derivatives bounded and let $\{u(t,x),t>0,x\in\R\}$ be an adapted process such that for each $(t,x)$ fixed, $u(t,x)$ is a smooth and cylindrical random variable. Let 
 \[
 I(t,x)= \int_0^t\int_\R \rho(u(s,y)) G(t-s,x-y) W(\ud s\ud y).
 \]
Denote 
\begin{align}\label{E:NoT-Alpha}
\begin{aligned}
 \alpha &=((\theta_1,\xi_1),\dots, (\theta_n,\xi_n))
\quad\text{and}\\
\hat{\alpha}_k&=((\theta_1,\xi_1),\dots, (\theta_{k-1},\xi_{k-1}),(\theta_{k+1},\xi_{k+1}),\dots (\theta_n,\xi_n)).
\end{aligned}
\end{align}
Then 
\begin{align}\label{E:DI}
\begin{aligned}
D^{n}_{\alpha} I(t,x) =
& \sum_{k=1}^n D_{\hat{\alpha}_k}^{n-1}
\rho(u(\theta_k,\xi_k))G(t-\theta_k,x-\xi_k)\\
&+ \int_{\vee_{i=1}^n \theta_i}^t\int_{\R} D^{n}_\alpha \rho(u(s,y))\: 
G(t-s,x-y) W(\ud s \ud y).
\end{aligned}
\end{align}
\end{lemma}
\begin{proof}
It is clear that when $n=1$, 
\[
D_{(\theta_1,\xi_1)} I(t,x) = \rho(u(\theta_1,\xi_1)) G(t-\theta_1,x-\xi_1)
+\int_{\theta_1}^t\int_\R D_{(\theta_1,\xi_1)} \rho(u(s,y)) G(t-s,x-y) 
W(\ud s\ud y).
\]
Assume \eqref{E:DI} is true for all $k=1,\cdots n-1$. Now we consider the case when $k=n$.
Let $\alpha=((\theta_1,\xi_1),\dots, (\theta_{n-1},\xi_{n-1}))$
and $\beta=((\theta_1,\xi_1),\dots, (\theta_{n},\xi_{n}))$.
By the induction assumption,
\begin{align*}
D^{n}_\beta I(t,x)   = &
D_{(\theta_n,\xi_n)} D^{n-1}_\alpha  X(t,x)\\
=& D_{(\theta_n,\xi_n)}\sum_{k=1}^{n-1} D_{\hat{\alpha}_k}^{n-2}
\rho(u(\theta_k,\xi_k))G(t-\theta_k,x-\xi_k)\\
&+
D^{n-1}_\alpha \rho(u(\theta_n,\xi_n))\: 
G(t-\theta_n,x-\xi_n)
\\
&+ \int_{\vee_{i=1}^n \theta_i}^t\int_{\R} D_{(\theta_n,\xi_n)} D^{n-1}_\alpha \rho(u(s,y))\: 
G(t-s,x-y)W(\ud s \ud y)\\
=&\sum_{k=1}^{n} D_{\hat{\beta}_k}^{n-1}
\rho(u(\theta_k,\xi_k))G(t-\theta_k,x-\xi_k)\\
&+ \int_{\vee_{i=1}^n \theta_i}^t\int_{\R} D^n_\beta \rho(u(s,y))\: 
G(t-s,x-y)W(\ud s \ud y).
\end{align*}
Hence, \eqref{E:DI} is true for $k=n$. This proves Lemma \ref{L:DI}.
\end{proof}

\bigskip
Now we are ready to prove part (3) or Proposition \ref{P:D1}.

\begin{proof}[Proof of part (3) of  Proposition \ref{P:D1}]
Fix some $T>0$. Throughout the proof, we assume that $t\in [0,T]$.
Let $u_m(t,x)$ be the Picard approximations given in \eqref{E:Picard}.
We will prove by induction that
\begin{align}\label{E:InductDInfy}
\sup_{m\in\bbN} 
\Norm{ u_m(t,x)}_{n,p,S} <\infty,\quad\text{for all $n\ge 1$ and all $p\ge 2$.}
\end{align}
This result together with Lemma \ref{L:Dinfty} implies that $u(t,x)\in \D_S^\infty$. 
Set $\calH = L^2(S)$.
Notice that 
\[
\Norm{F}_{n,p,S}^p  =  \E(|F|^p) + \sum_{k=1}^n \E\left(\Norm{D^k F}_{\calH^{\otimes k}}^{p}\right).
\]
Hence, it suffices to prove the following property by induction:
\begin{align}\label{E:InductDInfy2}
\sup_{m\in\bbN} \sup_{(t,x)\in[0,T]\times \R}
\Norm{\: \Norm{D^n u_m(t,x)}_{\calH^{\otimes n}}}_p <\infty,\quad\text{for all $n\ge 1$ and all $p\ge 2$.}
\end{align}
Fix an arbitrary $p\ge 2$.
Let 
\[
\lambda:=4\sqrt{p}(\Norm{\rho'}_{L^\infty(\R)}\vee \Lip_\rho),\quad
\widetilde{\calK}(t,x)：=\calK_{2\sqrt{p}\lambda}(t,x),
\quad
K_T:= (1\star\calK_{2\sqrt{p}\lambda})(T,x)\vee 1.
\] 
It is clear that $K_T$ is a constant that does not depend on $x$.

{\bigskip\bf\noindent Step 1.~}
We first consider the case $n=1$. 
We will prove by induction that 
\begin{align}\label{E:IndInd1}
\begin{aligned}
  \sup_{(t,x)\in[0,T]\times \R}
\Norm{\: \Norm{D u_m(t,x)}_{\calH}}_p^2 &\le 
4\lambda^2 K_T^2 \sup_{(s,y)\in S} \left[(1+J_0^2(s,y)) + ((1+J_0^2)\star\widetilde{\calK})(s,y)\right]\\
&=:\Theta,
\end{aligned}
\end{align}
where the constant $\Theta$ does not depend on $m$.
Note that the above constant $\Theta$ is finite due to the assumption on the initial data \eqref{E:supSNormU} and \eqref{E:K_bound}.
It is clear that $D_{\theta_1,\xi_1}u_0(t,x)\equiv 0$ satisfies  \eqref{E:IndInd1}. Suppose that \eqref{E:IndInd1} is true for $k\le m$. Now we consider the case $k=m+1$.
Notice that
\[
\begin{aligned}
 D_{\theta_1,\xi_1}u_{m+1}(t,x)= &\rho(u_{m}(\theta_1,\xi_1))G(t-\theta_1,x-\xi_1)\\
&+\int_{\theta_1}^t\int_\R G(t-s,x-y) \rho'(u_{m}(s,y)) D_{\theta_1,\xi_1} u_{m}(s,y) W(\ud s\ud y).
\end{aligned}
\]
For the first term, by the moment bounds for $\Norm{u_m(s,y)}_p$ in \eqref{E:unBds}, we see that
\begin{align*}
\sup_{(t,x)\in [0,T]\times\R}\E&\left[\left(\iint_{S}\rho(u_m(\theta_1,\xi_1))^2 G^2(t-\theta_1,x-\xi)\ud \theta_1\ud \xi_1\right)^{p/2}\right]^{2/p}
\\
&\le 
\sup_{(t,x)\in [0,T]\times\R} \iint_{S}\Norm{\rho(u_m(\theta_1,\xi_1))}_p^2 G^2(t-\theta_1,x-\xi)\ud \theta_1\ud \xi_1\\
&\le
\lambda^2 K_T \sup_{(s,y)\in S} \left(1+\Norm{u_m(s,y)}_p^2\right)\\
&\le \lambda^2 K_T
\sup_{(s,y)\in S} \left[(1+J_0^2(s,y)) + ((1+J_0^2)\star \widetilde{\calK})(s,y)\right] = \frac{\Theta}{4 K_T}.
\end{align*}
Hence, applying the Burkholder-Davis-Gundy inequality to the second term, we obtain that
\begin{align}\label{E_:Ind1}
\Norm{\:\Norm{Du_{m+1}(t,x)}_{\calH}}_p^2
\le & \frac{\Theta}{2K_T} +
2\lambda^2 \int_0^t\int_\R
G(t-\theta_1,x-\xi_1)^2 
\Norm{\:\Norm{Du_{m}(\theta_1,\xi_1)}_{\calH}}_p^2\ud \theta_1\ud \xi_1.
\end{align}
Therefore, by \eqref{E:K_p} with $\Vip=0$,
\[
\Norm{\:\Norm{Du_{m+1}(t,x)}_{\calH}}_p^2\le 
\frac{\Theta}{2K_T} + 
\frac{ \Theta}{2K_T} (1\star\widetilde{\calK})(T,x)
=\Theta \left(\frac{1}{2K_T}+\frac{1}{2}\right)\le \Theta,
\]
which proves \eqref{E:IndInd1} for $m+1$. One can conclude that \eqref{E:InductDInfy2} holds true for $n=1$ .

{\bigskip\bf\noindent Step 2.~}
Assume that \eqref{E:InductDInfy2} holds for $n-1$.
Now we will prove that \eqref{E:InductDInfy} is still true for $n$.
Similar to Step 1, we will prove by induction that for some constant $\Theta$ that does not depend on $m$, it holds that
\begin{align}\label{E:IndInd2}
 \sup_{(t,x)\in[0,T]\times \R}
\Norm{\: \Norm{D^n u_m(t,x)}_{\calH^{\otimes n}}}_p^2 <\Theta.
\end{align}
It is clear that $D^nu_0(t,x)\equiv 0$ satisfies  \eqref{E:IndInd2}. Suppose that \eqref{E:IndInd2} is true for $k\le m$. Now we consider the case $k=m+1$.
By Lemma \ref{L:DI}, we see  that
\begin{align*}
 D^n_{\alpha} u_{m+1}(t,x) =
& \sum_{k=1}^n D_{\hat{\alpha}_k}^{n-1}
\rho(u_{m}(\theta_k,\xi_k))G(t-\theta_k,x-\xi_k)\\
&+ \int_{\vee_{i=1}^n \theta_i}^t\int_{\R} D^n_\alpha \rho(u_{m}(s,y))\: 
G(t-s,x-y) W(\ud s \ud y);
\end{align*}
see \eqref{E:NoT-Alpha} for the notation for $\alpha$ and $\hat{\alpha}_k$.
Hence, 
\begin{align*}
\Norm{D^n u_{m+1}(t,x)}_{\calH^{\otimes n}}
\le & \sum_{k=1}^n \left(
\iint_S \ud \theta_k\ud\xi_k \:  G^2(t-\theta_k,x-\xi_k)\Norm{D^{n-1}\rho(u_{m}(\theta_k,\xi_k))}_{\calH^{\otimes (n-1)}}^2
\right)^{1/2}\\
&+\Norm{\int_{\vee_{i=1}^n\theta_i}^t \int_{\R} D^n \rho(u_{m}(s,y))G(t-s,x-y)W(\ud s\ud y)}_{\calH^{\otimes n}}.
\end{align*}
Then by the Burkholder-Davis-Gundy inequality and $(a_1+\dots+a_n)^2\le n(a_1^2+\dots+a_n^2)$, 
\begin{align*}
\Norm{\:\Norm{D^n u_{m+1}(t,x)}_{\calH^{\otimes n}}}_p^2 
\le & C_n \iint_S G^2(t-s,x-y) \Norm{\:\Norm{D^{n-1} \rho(u_{m}(s,y))}_{\calH^{\otimes (n-1)}}}_p^2 \ud s\ud y\\
&+C_n
\int_0^t\int_\R G^2(t-s,x-y) 
\Norm{\:\Norm{D^n \rho(u_{m}(s,y))}_{\calH^{\otimes n}}}_p^2 \ud s\ud y\\
=:& A_1 + A_2,
\end{align*}
where $C_n=n+1$.
By Lemma \ref{L:LeibnizMoment} and the induction assumption, 
\[
A_1 \le C_n K_T  \sup_{m\in\bbN}\sup_{(s,y)\in [0,T]\times\R}\Norm{\:\Norm{D^{n-1} \rho(u_{m}(s,y))}_{\calH^{\otimes (n-1)}}}_p^2  <\infty.
\]
Following the notation in \cite{NualartQuer07}, let 
\[
\Delta_\alpha^n (\rho, u) := D_\alpha^n \rho(u) - \rho'(u)D^n_\alpha u
\]
be all terms in the summation of $D_\alpha^n \rho(u)$ that have Malliavin derivatives of order less than or equal to $n-1$.
Then 
\begin{align}
\notag
A_2\le & 2 \lambda^2 C_n
\int_0^t\int_\R G^2(t-s,x-y) 
\Norm{\:\Norm{D^n u_{m}(s,y)}_{\calH^{\otimes n}}}_p^2 \ud s\ud y\\
\notag
&+2C_n
\int_0^t\int_\R G^2(t-s,x-y) 
\Norm{\:\Norm{\Delta^n (\rho,u_{m}(s,y))}_{\calH^{\otimes n}}}_p^2 \ud s\ud y\\
=: &A_{2,1} + A_{2,2}.\label{E_:A2}
\end{align}
By the induction assumption, we see that 
\[
A_{2,2} \le 2 C_n K_T  \sup_{m\in\bbN}\sup_{(s,y)\in [0,T]\times\R}\Norm{\:\Norm{\Delta^{n}(\rho,u_{m}(s,y))}_{\calH^{\otimes n}}}_p^2  <\infty.
\]
Therefore, for some constant $C_n'>0$,
\[
\Norm{\:\Norm{D^n u_{m+1}(t,x)}_{\calH^{\otimes n}}}_p^2 
\le C_n' + 2C_n\lambda^2 \int_0^t\int_\R G^2(t-s,x-y) \Norm{\:\Norm{D^{n} u_{m}(s,y)}_{\calH^{\otimes n}}}_p^2 \ud s\ud y.
\]
Comparing the above inequality with \eqref{E_:Ind1}, we see that one can prove \eqref{E:IndInd2} by the same arguments as those in Step 1. Therefore, we have proved \eqref{E:InductDInfy2}, which completes the proof of part (3) of Proposition \ref{P:D1}.
\end{proof}

\bigskip

\begin{remark}\label{R:LocalAnalize}
In this remark, we point out why we need the above localized Malliavin analysis.
Actually, if $\rho(t,x,u)=\lambda u$, then there is no need to resort to the localized Malliavin analysis. 
One can prove
that $u(t,x)\in \D^\infty$ because $\rho^{(n)}(u)\equiv 0$ for $n\ge 2$.
However, if $\rho(t,x,u)$ is not linear in the third argument, we are not able to prove that $u(t,x)\in\D^\infty$ and we need to resort to a larger space $\D^\infty_S$.
The reason is explained as follows:
Let $\calH_T=L^2((0,\infty)\times \R)$ and $u_m(t,x)$ be the Picard approximation of $u(t,x)$ in \eqref{E:Picard}.
From \eqref{E:DunBds}, we see that 
\begin{align}\label{E:RDum}
\sup_{m\in\bbN}\Norm{\:\Norm{D u_m(t,x)}_{\calH_T}}_p^2 \le  ((1+J_0^2)\star\calG)(t,x).
\end{align}
Now for the Malliavin derivatives of order $2$, the $A_{2,2}$ term in \eqref{E_:A2} is bounded as 
\begin{align*}
A_{2,2} \le & 2 C_2 \Norm{\rho''}_{L^\infty(\R)}^2\int_0^t\int_\R G^2(t-s,x-y) \Norm{\:\Norm{D u_m(s,y)}_{\calH_T}^2}_p^2 \ud s\ud y \\
\le & C
\int_0^t\int_\R G^2(t-s,x-y) \Norm{\:\Norm{D u_m(s,y)}_{\calH_T}}_{2p}^4 \ud s\ud y.
\end{align*}
Then from \eqref{E:RDum}, we see that a sufficient condition for $A_{2,2}$ to be finite is 
\[
\left((J_0^2\star\calG)^2 \star G^2 \right)(t,x) <\infty.
\]
In general, for the Malliavin derivatives of all orders, one needs to impose the following condition
\begin{align}\label{E_:RJn}
\left((J_0^2\star\calG)^n \star G^2\right)(t,x) <\infty,\quad\text{for all $n\in\bbN$.}
\end{align}
If the initial data is bounded, then $\sup_{t>0,x\in\R}(J_0^2\star\calG)(t,x)<\infty$, which implies \eqref{E_:RJn}.
However, condition \eqref{E_:RJn} is too restrictive for measure-valued initial data. For example, if the initial data is the Dirac delta measure $\delta_0(x)$, then $J_0(t,x)= G(t,x)$ and  $(J_0^2\star\calG)(t,x) \sim  t^{1-2/\alpha} G(t,x)$. Thus,
\begin{align*}
\label{E_:simsim}
(J_0^2\star\calG)^n(t,x)\sim t^{n(1-2/\alpha)} G^n(t,x)
\sim t^{n(1-2/\alpha)-(n-1)/\alpha} G(t,x)=:f(t,x),
\end{align*}
and condition \eqref{E_:RJn} can be written as 
$(f\star G^2)(t,x)<\infty$. One can bound $G^2(t,x)$ by $C \calG(t,x)$ and the integral in the spatial variable can be evaluated using the semigroup property. As for the integral in the time variable in the convolution $f\star G^2$, it is easy to verify that it is finite if 
\[
n(1-2/\alpha) -(n-1)/\alpha > -1 \quad\Longleftrightarrow\quad
n<\frac{1+\alpha}{3-\alpha}.
\]
Therefore, condition \eqref{E_:RJn} is true only for $n<(1+\alpha)/(3-\alpha)$, where the upper bound is a number in $(1,3]$ since $\alpha\in (1,2]$.
\end{remark}

\section{Existence and smoothness of density at a single point}\label{S:Main}

In this section, we will first generalize the sufficient condition for the existence and smoothness of density by Mueller and Nualart \cite{MuellerNualart08} from function-valued initial data to measure-valued initial data; see Theorem \ref{T_:Single} in Section \ref{SS:MuellerNualart}. Then based on a stopping time argument, we establish the necessary and sufficient condition \eqref{E:lip} in Section \ref{SS:Single}.

\subsection{A sufficient condition}
\label{SS:MuellerNualart}

%

In this subsection, we will prove the following theorem:

\begin{theorem}\label{T_:Single}
Let $u(t,x)$ be the solution  to \eqref{E:FracHt} starting from an initial measure $\mu$
that satisfies \eqref{E:InitD}.
Assume that $\mu$ is proper at some point $x_0\in\R$ with a
density function $f$ over a neighborhood $(a,b)$ of $x_0$.
Suppose that  $\rho(0,x_0,f(x_0))\ne 0$ and that
$f$ is $\beta$-H\"older continuous on $(a,b)$ for some $\beta\in (0,1)$.
Then we have the following two statements:
\begin{enumerate}[parsep=0ex,topsep=1ex]
 \item[(a)] If $\rho$ is differentiable in the third argument with bounded Lipschitz continuous derivative,
 then  for all $t>0$ and $x\in\R$,  $u(t,x)$ has an absolutely continuous law with respect to the Lebesgue measure.
 \item[(b)] If $\rho$ is infinitely differentiable in the third argument with bounded derivatives,
 then  for all $t>0$ and $x\in\R$,  $u(t,x)$ has a smooth  density.
\end{enumerate}
\end{theorem}

We need one lemma regarding the SPDE \eqref{E:SPDE-Var}.
Recall that the constant $\Lambda$ is defined in \eqref{E:Lambda}.

\begin{lemma}\label{L:Moment}
Suppose that $u_0(x)$ is an $\calF_\theta$-measurable process indexed by $x\in\R$ with $\theta\ge 0$ such that
$u_0(x)\in L^p(\Omega)$ for all $p\ge 2$ and $x\in\R$, and
\[
\int_\R \Norm{u_0(y)}_p^2 G(t,x-y)\ud y<+\infty,\quad\text{for all $t>0$ and $x\in\R$}.
\]
Then there is a unique solution $u(t,x)$ to \eqref{E:SPDE-Var} starting from $u_0$ with $\W$ replaced by $\W_\theta$.
Moreover, for all $p\ge 2$ and $T>0$, there is a finite constant $C:=C_{p,T,\Lambda}>0$ such that
\begin{align}\label{E:Moment}
\Norm{u(t,x)}_p^2 \le C \int_\R \Norm{u_0(y)}_p^2 G(t,x-y)\ud y<+\infty,\quad\text{for all $(t,x)\in [0,T]\times\R$}.
\end{align}
\end{lemma}
\begin{proof}
Fix $(t,x)\in [0,T]\times\R$ and $p\ge 2$.
The proof of the existence and uniqueness of a solution follows from
the standard Picard iteration. Here we only show the bounds for $\Norm{u(t,x)}_p$.
Note that the moment formulas in \cite{ChenDalang13Heat} and \cite{ChenDalang15FracHeat}
are for deterministic initial conditions.
In the current case, similar formulas hold due to the Lipschitz continuity of $\sigma(t,x,z)$ uniformly in $(t,x)$.

Note that the moment bound in \eqref{E:BGD-U} is for the case when the initial data are deterministic.
When initial data are random, we need to replace $J_0^2(t,x)$ by $\Norm{\int_\R u_0(y) G(t,x-y)\ud y}_p^2$.
By Minkowski's inequality and H\"older's inequality,
\begin{align}\label{E:J*}
\Norm{\int_\R u_0(y) G(t,x-y)\ud y}_p^2
\le \int_\R \Norm{u_0(y)}_p^2 G(t,x-y)\ud y=:J_*(t,x).
\end{align}
Hence, by \eqref{E:BGD-U},
\[
\Norm{u(t,x)}_p^2
\le C J_*(t,x)
+ C \int_0^t\ud s\int_\R \ud y\: G(t-s,x-y)^2 \Norm{u(s,y)}_p^2\ud s\ud y.
\]
Then by Lemma \ref{L:Mom},
\begin{align*}
\Norm{u(t,x)}_p^2
\le & C J_*(t,x)+C \int_0^t \frac{\ud s}{(t-s)^{1/\alpha}}
\int_\R \ud y\: G(t-s,x-y)
\int_\R \ud z \: G(s,y-z)   \Norm{u_0(z)}_p^2 \\
\le&
CJ_*(t,x) +
C \int_0^t \frac{\ud s}{(t-s)^{1/\alpha}}
\int_\R \ud z \: G(t,x-z)   \Norm{u_0(z)}_p^2\\
=&C\left(1+\int_0^t \frac{\ud s}{(t-s)^{1/\alpha}}\right)J_*(t,x)
\end{align*}
which is finite because $\alpha>1$. This completes the proof of Lemma \ref{L:Moment}.
\end{proof}

\begin{proof}[Proof of Theorem \ref{T_:Single}]
Recall that $\rho (u(t,x))$ is a short-hand notation for $\rho (t,x,u(t,x))$.
By Proposition \ref{P:D1}, we know that 
\begin{equation}\label{E:SPDE-M}
\begin{aligned}
D_{\theta,\xi}u(t,x) =&
\quad\rho(u(\theta,\xi)) G(t-\theta,x-\xi)\\
&+
\int_\theta^t \int_\R G(t-s,x-y)\rho'(u(s,y))D_{\theta,\xi}u(s,y) W(\ud s, \ud y),
\end{aligned}
\end{equation}
for $0\le \theta\le t$ and $\xi\in\R$.
By the assumptions on $\mu$, for some constants $a<b$, we have that
$\one_{[a,b]}(x) \mu(\ud x) = f(x)\ud x$,  where $f$ is a $\beta$-H\"older
continuous (and hence bounded) function n $[a,b]$ for some $\beta\in(0,1)$.
Set $S=[0,t]\times [a,b]$.

For part (b), notice that Lemma \ref{L:LimPmom} and part (3) of Proposition \ref{P:D1}
imply that $u(t,x)\in \D^\infty_S$. Denote
\begin{align}\label{E:C}
C(t,x)=\int_0^t \int_a^b \left[D_{\theta,\xi}u(t,x)\right]^2\ud \xi\ud \theta.
\end{align}
By Theorem \ref{T:Density2}, both parts (a) and (b) are proved once we can show
$\E\left[C(t,x)^{-p}\right]<\infty$ for all $p\ge 2$.

Because $\rho(0,x_0,f(x_0))> 0$, by the  continuity we can find $a'$ and $b'$
such that $b'-a'\le 1$, $[a',b']\subseteq[a,b]$ and $\rho(0,x,f(x))\ge \delta>0$
for all $x\in [a',b']$. Let $\psi$ be a continuous function with support
$[a',b']$ and $0\le \psi(\xi)\le 1$. 
Set
\begin{align}\label{E:Y}
Y_\theta(t,x):=\int_\R D_{\theta,\xi}u(t,x) \psi(\xi)\ud \xi.
\end{align}
Choose $r\in(0,1)$ and $\epsilon$ such that $0<\epsilon^r<t$.
Then
\begin{align*}
C(t,x)&\ge \int_0^t \ud\theta\int_\R \ud \xi \: \psi(\xi) \left[D_{\theta,\xi}u(t,x)\right]^2\\
&\ge \int_0^t \left(\int_\R D_{\theta,\xi}u(t,x) \psi(\xi)\ud \xi\right)^2\ud \theta\\
&= \int_0^t Y_\theta^2(t,x) \ud \theta\ge \int_0^{\epsilon^r} Y_\theta^2(t,x) \ud \theta\\
&\ge \epsilon^r Y_0^2(t,x)-\int_0^{\epsilon^r}
\left|Y_\theta^2(t,x)-Y_0^2(t,x)\right|\ud \theta.
\end{align*}
Hence,
\begin{align*}
 \bbP\left(C(t,x)<\epsilon\right) &\le
 \bbP\left(\left|Y_0(t,x)\right|< \sqrt{2}\: \epsilon^{\frac{1-r}{2}}\right)+
 \bbP\left(\int_0^{\epsilon^r} \left|Y_\theta^2(t,x)-Y_0^2(t,x)\right|\ud \theta>\epsilon\right)
 \\
 &=: \bbP(A_1)+\bbP(A_2).
\end{align*}
In the following, we consider $\bbP(A_1)$ and $\bbP(A_2)$ separately in two steps.

{\bigskip\noindent\bf Step 1.~~} We first consider $\bbP(A_1)$.
By integrating both sides of \eqref{E:SPDE-M} against $\psi(\xi)\ud \xi$,
we see that $Y_\theta(t,x)$ solves the following integral equation
\begin{align}
\label{E:SI-Y}
\begin{aligned}
 Y_\theta(t,x)=&\quad \int_\R  \psi(\xi) \rho(u(\theta,\xi)) G(t-\theta,x-\xi)\ud \xi\\
 &+ \int_\theta^t \int_\R G(t-s,x-y)\rho'(u(s,y))Y_\theta(s,y)W(\ud s,\ud y),\quad \text{for $t\ge \theta$}.
\end{aligned}
\end{align}
In particular, $Y_0(t,x)$ is a mild solution to the following SPDE
\begin{align}\label{E:Y_0}
 \begin{cases}
  \left(\displaystyle\frac{\partial}{\partial t} - \Dxa \right) Y_0(t,x) =
\rho'(u(t,x)) Y_0(t,x) \dot{W}(t,x),& t>0\;,\: x\in\R,\cr
Y_0(0,x) = \psi(x)\rho(0,x,f(x)).
 \end{cases}
\end{align}
Because for some $\delta'>0$, $\psi(x)\rho(0,x,f(x))\ge \delta'$ for all $x\in[a',b']$,
the assumption in Theorem \ref{T:NegMom} is satisfied. Hence, \eqref{E:Rate} implies that
for all $p\ge 1$, $t>0$ and $x\in\R$,
\begin{align}\label{E:PA2}
\bbP(A_1)\le C_{t,x,p}\: \epsilon^p,\quad\text{for $\epsilon$ small enough.}
\end{align}

{\bigskip\noindent\bf Step 2.~~} Now we consider $P(A_2)$.
For all $t>0$, $x\in\R$ and for all $q\ge 1$, we see  by Chebyshev inequality that
\[
\bbP(A_2)\le\epsilon^{q(r-1)}
\sup_{(\theta,x)\in [0,\epsilon^r]\times\R}
\E\left[\left|Y_\theta(t,x)-Y_0(t,x)\right|^{2q}\right]^{1/2}
\E\left[\left|Y_\theta(t,x)+Y_0(t,x)\right|^{2q}\right]^{1/2}.
\]
We claim that
\begin{align} \label{E:NormY}
\sup_{(\theta, x)\in [0,t]\times\R} \E\left[|Y_\theta(t,x)|^{2q}\right]<+\infty.
\end{align}
From \eqref{E:SI-Y}, we see that $X_\theta(t,x):=Y_\theta(t+\theta,x)$, $t\ge 0$, solves the following SPDE
\begin{align}\label{E:X_theta}
 \begin{cases}
  \left(\displaystyle\frac{\partial}{\partial t} - \Dxa \right) X_\theta(t,x) =
\rho'(u(t,x)) X_\theta(t,x) \dot{W}_\theta(t,x),& t>0\;,\: x\in\R,\cr
X_\theta(0,x) = \psi(x)\rho(\theta,x,u(\theta,x)),
 \end{cases}
\end{align}
where $\W_\theta(t,x):=\W(t+\theta,x)$ is a time-shifted white noise.
By the linear growth condition \eqref{E:LinGrw} of $\rho$ and part (1) of Lemma \ref{L:LimPmom},
we see that for all $t>0$ and $x\in\R$,
\begin{align*}
\int_\R \Norm{\psi(y)\rho(u(\theta, y))}_{2q}^2 G(t,x-y)\ud y&\le
\Lip_\rho^2
\int_\R\psi(y)\left(\vv+\Norm{u(\theta, y)}_{2q}^2 \right)G(t,x-y)\ud y\\
&\le
\Lip_\rho^2 \left(\vv+\sup_{(s,y)\in[0,t]\times [a',b']}\Norm{u(s, y)}_{2q}^2\right)<+\infty.
\end{align*}
Hence, Lemma \ref{L:Moment} implies that for $\theta\in [0,t]$, there is a solution to \eqref{E:X_theta} with
\begin{align*}
\sup_{(\theta,x)\in [0,t]\times\R}\Norm{Y_\theta(t,x)}_{2q}^2
& = \sup_{(\theta,x)\in [0,t]\times\R} \Norm{X_\theta(t-\theta,x)}_{2q}^2<+\infty,
\end{align*}
which proves \eqref{E:NormY}.

\bigskip
Now we consider the other term:
\begin{align}\notag
 Y_\theta(t,x)-Y_0(t,x) =
 &\quad  \int_\R \psi(\xi)\left[\rho(u(\theta,\xi)) G(t-\theta,x-\xi) - \rho(u(0,\xi)) G(t,x-\xi)\right]\ud \xi\\
 \notag
 & -\int_0^\theta \int_\R G(t-s,x-y)  \rho'(u(s,y)) Y_0(s,y) W(\ud s,\ud y)\\
 \notag
 & + \int_\theta^t\int_\R G(t-s,x-y)\rho'(u(s,y)) \left(Y_\theta(s,y)-Y_0(s,y)\right)W(\ud s,\ud y)\\
 =:&\Psi_1-\Psi_2+\Psi_3.
 \label{E:Y-Y}
\end{align}
By the Lipschitz continuity and linear growth condition of $\rho$, for $q\ge 2$ and some constant $C_q>0$,
\begin{align*}
 \E\left[|\Psi_1|^{2q}\right]
 \le&
 \quad C_q\: \E\left[\left|\int_\R \psi(\xi) |u(\theta,\xi)-u(0,\xi)|G(t-\theta,x-\xi)\ud \xi\right|^{2q}\right]\\
 &+C_q\Lip_\rho^2 \left|\int_\R \psi(\xi) \left| G(t-\theta,x-\xi)-G(t,x-\xi)\right| (\vv+u(0,\xi)^2)\ud \xi\right|^{2q}\\
 =:& C_q \Psi_{11}+C_q\Lip_\rho^2 \Psi_{12}.
\end{align*}
By the Minkowski inequality and the H\"older continuity of $u(t,x)$ (part (2) of Lemma \ref{L:LimPmom}),
\[
\Psi_{11}\le \left(\int_\R \psi(\xi) G(t-\theta,x-\xi)\Norm{u(\theta,\xi)-u(0,\xi)}_{2q}\ud \xi\right)^{2q}
\le C_K \: \theta^{\frac{q\min(\alpha-1,\beta)}{\alpha}}.
\]
%
As for $\Psi_{12}$, applying H\"older's inequality twice, we see that
\begin{align*}
 \Psi_{12}\le &
 \left(\int_\R \psi(\xi)  [G(t-\theta,x-\xi)-G(t,x-\xi)]^2 \: (1+u(0,\xi)^2) \ud \xi\right)^{q}\\
 \le &\sup_{\xi\in[a',b']} (1+f(\xi)^{2})^q \: \left(\int_\R  [G(t-\theta,x-\xi)-G(t,x-\xi)]^2 \ud \xi\right)^{q}\\
 \le & C \sup_{\xi\in[a',b']} (1+f(\xi)^{2})^q\: \theta^{\frac{(\alpha-1)q}{\alpha}},
\end{align*}
where in the last step we have applied \cite[Proposition 4.4]{ChenDalang15FracHeat} and $C$ is a universal constant.
Therefore, for all $q\ge 1$, there exists some constant $C_{\mu,t,x,q}>0$,
\[
\E\left[|\Psi_1|^{2q}\right]\le  C_{\mu,t,x,q} \: \theta^{\frac{q\min(\alpha-1,\beta)}{\alpha}},\quad
\text{for all $\theta\in (0,t]$.}
\]
As for $\Psi_2$, set $\Lambda:=\sup_{(t,x,z)}|\rho'(t,x,z)|$.
Because the initial data for $Y_0(t,x)$ is a bounded function,
$\sup_{(s,y)\in [0,t]\times\R}\Norm{Y_0(s,y)}_{2q}<\infty$.
By \eqref{E:BGD} and \eqref{E:TildeG},
\begin{align*}
\Norm{\Psi_2}_{2q}^2&\le
\Lambda^2 C_q
\int_0^\theta\ud s \int_\R \ud y\: G(t-s,x-y)^2 \Norm{Y_0(s,y)}_{2q}^2\\
&\le
\Lambda^2 C_q\sup_{(s,y)\in [0,t]\times\R}\Norm{Y_0(s,y)}_{2q}^2
\int_0^\theta \ud s \int_\R \ud y\: G(t-s,x-y)^2\\
&\le
\Lambda^2 C_q C_{a,\delta}^2 \sup_{(s,y)\in [0,t]\times\R}\Norm{Y_0(s,y)}_{2q}^2
\int_0^\theta \widetilde{G}(2(t-s),0) \ud s
\\
&=
\Lambda^2 C_q C_{a,\delta}^2 \widetilde{G}(1,0) \sup_{(s,y)\in [0,t]\times\R}\Norm{Y_0(s,y)}_{2q}^2
\int_0^\theta  \frac{1}{(2(t-s))^{1/\alpha}}\ud s\\
&\le C \theta^{1-1/\alpha},\quad
\text{for all $\theta\in (0,t]$.}
\end{align*}
Thus, for some constant $C_{\mu,t,x,q}'>0$,
\[
\E\left[|\Psi_2|^{2q}\right]\le  C_{\mu,t,x,q}' \: \theta^{\frac{(\alpha-1)q}{\alpha}},\quad
\text{for all $\theta\in (0,t]$.}
\]
Applying   the Burkholder-Gundy-Davis inequality \eqref{E:BGD-U}
to the $\Psi_3$ in \eqref{E:Y-Y}, we can write
\begin{align}
 \Norm{Y_\theta(t,x)-Y_0(t,x)}_{2q}^2
 \le&\quad C_q \left(\Norm{\Psi_1}_{2q}^2 +\Norm{\Psi_2}_{2q}^2\right)
 \label{e.5.12} \\
 &+C_q\Lambda^2\int_\theta^t\ud s \int_\R \ud y\: G^2(t-s,x-y) \Norm{Y_\theta(s,y)-Y_0(s,y)}_{2q}^2. \nonumber
\end{align}
By \eqref{E:NormY}, we know that
$g_{\theta}(t):=\sup_{x\in\R}\Norm{Y_\theta(t,x)-Y_0(t,x)}_{2q}^2$ is well defined.
Thus, we can write inequality \eqref{e.5.12} as
\[
g_{\theta}(t)\le C_{\Lambda,\mu,q,x}  \left( \theta^{\frac{\min(\alpha-1,\beta)}{\alpha}}+\int_\theta^t \frac{1}{(t-s)^{1/\alpha}} g_\theta(s)\ud s\right).
\]
By Gronwall's lemma (see Lemma \ref{L:Mittag}),  we see that
\[
g_{\theta}(t)\le C_{\Lambda,\mu,q,t,x}^* \: \theta^{\frac{\min(\alpha-1,\beta)}{\alpha}}
\quad\text{for all $\theta\in (0,t]$.}
\]
Therefore,
\begin{align}
\sup_{0<\theta \le \epsilon^r ,\:x\in\R} \E\left[
\left|Y_\theta(t,x)-Y_0(t,x)\right|^{2q}
\right]
\le C_{\Lambda,\mu,q,t,x,r} \: \epsilon^{\frac{rq\min(\alpha-1,\beta)}{\alpha}},
\end{align}
and consequently, we have
\begin{align}\label{E:PA1}
\bbP(A_2)\le C_{\Lambda,\mu,q,t,x,r}'\: \epsilon^{(r-1)q+\frac{rq\min(\alpha-1,\beta)}{2\alpha}},\qquad\text{for all $t>0$ and $x\in \R$.}
\end{align}
Notice that
\begin{align*}
(r-1) +\frac{r \beta}{2\alpha}>0&\quad\Rightarrow \quad r>\frac{2\alpha}{2\alpha+\beta}\in [0,1),\quad\text{and}\\
(r-1) +\frac{r (\alpha-1)}{2\alpha}>0&\quad\Rightarrow \quad r>\frac{2\alpha}{3\alpha-1}\in [4/5,1) \quad\text{for $\alpha\in (1,2]$}.
\end{align*}
By choosing $r$ such that,
\[
\frac{2\alpha}{\min(3\alpha-1,2\alpha+\beta)} < r<1,
\]
we see that $(r-1)q+\frac{rq \min(\alpha-1,\beta)}{2\alpha}>0$ for all $q\ge 2$.

\bigskip
Finally, Theorem \ref{T_:Single} is proved by an application of Lemma \ref{L:NegMom} with \eqref{E:PA1} and \eqref{E:PA2}.
\end{proof}

\subsection{Proof of Theorem \ref{T:Single}}
\label{SS:Single}
Now we are ready to prove Theorem \ref{T:Single}.

\begin{proof}[Proof of Theorem \ref{T:Single}]
Recall that the solution $u(t,x)$ to \eqref{E:FracHt} is understood in the mild form \eqref{E:mild}.
Let $I(t,x)$ be the stochastic integral part in \eqref{E:mild}, i.e., $u(t,x)=J_0(t,x)+I(t,x)$.
Denote 
\[
t_0:=\inf\left\{s>0,\: \sup_{y\in\R}\left|\rho\left(s,y,(G(s,\cdot)*\mu)(y)\right)\right|\ne 0\right\}.
\]
If condition \eqref{E:IFF} is not satisfied, i.e., $t\le t_0$, then $I(t,x)\equiv 0$, which implies that 
$u(t,x)=J_0(t,x)$ is deterministic. Hence, $u(t,x)$ doesn't have a density. This proves one direction for both parts (a) and (b). 

Now we assume that condition \eqref{E:IFF} is satisfied, i.e., $t>t_0$.
By the continuity of the function
\[
(0,\infty)\times\R^2 \ni (t,x,z) \mapsto \rho\left(t,x, (G(t,\cdot)*\mu)(x)+z\right) \in\R,
\]
we know that for some $\epsilon_0\in (0,t-t_0)$ and some $x_0\in\R$, it holds that 
\begin{align}\label{E:rho=0}
\rho\left(t_0+\epsilon, y, (G(t_0+\epsilon,\cdot)*\mu)(y)+z\right)\ne 0
\end{align}
for all $(\epsilon,y,z)\in (0,\epsilon_0)\times [x_0-\epsilon_0,x_0+\epsilon_0]\times [-\epsilon_0,\epsilon_0]$.
Let $\tau$ be the stopping time defined as follows
\[
\tau:= \left(t_0+\epsilon_0\right)\wedge \inf\left\{t>t_0, \: \sup_{y\in [x_0-\epsilon_0,x_0+\epsilon_0]} |I(t,y)|\ge \epsilon_0\right\}.
\] 
Let $\W_*(t,x)=\W(t+\tau,x)$ be the time shifted space-time white noise (see \eqref{E:W*}) and similarly,
let $\rho_*(t,x,z) = \rho(t+\tau,x,z)$. Let $u_*(t,x)$ be the solution to the following stochastic heat equation 
\begin{align}\label{E:FracHt*}
 \begin{cases}
  \left(\displaystyle\frac{\partial}{\partial t} - \Dxa \right) u_*(t,x) =
\rho_*(t,x,u_*(t,x))  \dot{W}_*(t,x),& t>0\;,\: x\in\R,\cr
u_*(0,x) = u(\tau,x),\:& x\in\R.
 \end{cases}
\end{align}
By the construction, we see that 
\[
\sup_{y\in[x_0-\epsilon_0,x_0+\epsilon_0]}|I(\tau,y)|\le \epsilon_0.
\]
Hence, property \eqref{E:rho=0} implies that
\[
\rho_*\left(0,y,u(\tau,y)\right)=
\rho\left(\tau,y,J_0(\tau,y)+I(\tau,y)\right)\ne 0, \quad\text{for all $y\in [x_0-\epsilon_0,x_0+\epsilon_0]$.}
\]
Notice that $y\mapsto u(\tau,y)$ is $\beta$-H\"older continuous a.s. for any $\beta\in (0,1/2)$.
Therefore, we can apply Theorem \ref{T_:Single} to \eqref{E:FracHt*} to see that if $\rho$ is differentiable in the third argument with bounded Lipschitz continuous derivative, then $u_*(t,x)$ has a conditional density, denoted as  $f_t(x)$, that is absolutely continuous with respect to the Lebesgue measure. Moreover, if $\rho$ is infinitely differentiable in the third argument with bounded derivatives, then this conditional density $x\mapsto f_t(x)$ is smooth a.s.

Finally, for any nonnegative continuous function $g$ on $\R$ with compact support, it holds that 
\begin{align*}
\E\left[g(u(t,x))\right] &=\E\left[\E \left[g(u(t,x))| \calF_\tau\right]\right]\\ 
&=\E\left[\E \left[g(u_*(t-\tau,x))| \calF_\tau\right]\right]\\
&=\E\left[\int_\R g(x) f_{t-\tau}(x)\ud x\right] \\
&= \int_\R g(x)\: \E\left[f_{t-\tau}(x)\right]\ud x.
\end{align*}
Therefore, if $\rho$ is differentiable in the third argument with bounded Lipschitz continuous derivative, then $u(t,x)$ has a density, namely $\E\left[f_{t-\tau}(x)\right]$, which is absolutely continuous with respect to the Lebesgue measure. Moreover, if $\rho$ is infinitely differentiable in the third argument with bounded derivatives, then this density is smooth.
This completes the proof of Theorem \ref{T:Single}.
\end{proof}

\section{Smoothness of joint density at multiple points}\label{S:Mult}
In this section, we will establish both Theorem \ref{T:Mult} and Theorem \ref{T:LocDen}.

\subsection{Proof of Theorem \ref{T:Mult}} \label{SS:Mult}
Recall that $\rho (u(t,x))$ is a short-hand notation for $\rho (t,x,u(t,x))$.
By taking the Malliavin derivative on both sides of \eqref{E:mild}, we see that
\begin{align}
\label{E:SPDE-Du}
D_{r,z}u(t,x) =&
\rho(u(r,z)) G(t-r,x-z) + Q_{r,z}(t,x)
\end{align}
for $0\le r\le t$ and $z\in\R$, where
\begin{align}
Q_{r,z}(t,x):=\int_r^t \int_\R G(t-s,x-y)\rho'(u(s,y))D_{r,z}u(s,y) W(\ud s, \ud y).
\end{align}
Let $S_{r,z}(t,x)$ be the solution to the equation
\[
S_{r,z}(t,x)=G(t-r,x-z)+\int_r^t\int_\R G(t-s,x-y)\rho'(u(s,y)) S_{r,z}(s,y)W(\ud s,\ud y).
\]
Then
\[
\Psi_{r,z}(t,x):=D_{r,z}u(t,x) = S_{r,z}(t,x)\rho(u(r,z)).
\]

We first prove a lemma.
\begin{lemma}\label{L:S}
For $t\in (0,T]$, $x\in\R$ and $p\ge 2$, there exists a constant
$C>0$, which depends on $T$, $p$, and $\sup_{x\in\R}|\rho'(x)|$,  such that
\[
\Norm{S_{r,z}(t,x)}_p^2\le C\: (t-r)^{-1/\alpha}G(t-r,x-z).
\]
\end{lemma}
\begin{proof}
Let $\lambda:=\sup_{x\in\R}|\rho'(x)|$.
By the Burkholder-Gundy-Davis inequality (see \eqref{E:BGD-U}),
\[
\Norm{S_{r,z}(t,x)}_p^2
\le C_p  G(t-r,x-z)^2 + C_{p}\lambda^2 \int_r^t\ud s\int_\R \ud y\:
G(t-s,x-y)^2 \Norm{S_{r,z}(s,y)}_p^2.
\]
Denote $f(\theta,\eta)=\Norm{S_{r,z}(\theta+r,\eta+z)}_p^2$.
Then by setting $\theta=t-r$ and $\eta=x-z$, we see that $f$ satisfies the inequality
\[
f(\theta,\eta)
\le C_p  G(\theta,\eta)^2 + C_{p}\lambda^2 \int_0^\theta \ud s\int_\R \ud y\:
G(\theta-s,\eta-y)^2 f(s,y).
\]
Then the lemma is proved by an application of Proposition 2.2 of \cite{ChenDalang13Heat} in case of $\alpha=2$,
or by Proposition 3.2 of \cite{ChenDalang15FracHeat} in case of $\alpha\in (1,2)$.
\end{proof}

\bigskip
\begin{proof}[Proof of Theorem \ref{T:Mult}]
Fix $t>0$ and $x_1<\dots<x_d$.
Let $\epsilon_0>0$ such that
\begin{align}\label{E:epsilon}
\epsilon_0\le \min(t/2,1)\quad\text{and}\quad
2\: \epsilon_0^{\frac{\alpha-1}{\alpha(1+\alpha)}} < \min_{i,j} |x_i-x_j|.
\end{align}
Choose a compact set $K$ such that $\bigcup_{i=1}^d
\left[x_i-\epsilon_0^{1/\alpha},x_i+\epsilon_0^{1/\alpha}\right]\subseteq K$.
Set $S=[t/2,t]\times K$.
The localized Malliavin matrix of the random vector $(u(t,x_1),\dots,u(t,x_d))$
with respect to $S$ is equal to
\[
\sigma_{ij} := \sigma_{ij,S} = \int_{t/2}^t\ud r \int_K \ud z \:
\Psi_{r,z}(t,x_i)\Psi_{r,z}(t,x_j).
\]
Hence, for any $\xi\in\R^d$ and any $\epsilon\in (0,\epsilon_0)$,
\begin{align*}
\InPrd{\sigma \xi,\xi}&=\int_{t/2}^t\ud r \int_K \ud z \left(\sum_{i=1}^d
\Psi_{r,z}(t,x_i)\xi_i\right)^2\\
&\ge
\sum_{j=1}^d\int_{t-\epsilon}^t \ud r \int_{x_j-\epsilon^{1/\alpha}}^{x_j+\epsilon^{1/\alpha}}\ud z \left(\sum_{i=1}^d \Psi_{r,z}(t,x_i)\xi_i\right)^2\\
&\ge
\sum_{j=1}^d\int_{t-\epsilon}^t \ud r \int_{x_j-\epsilon^{1/\alpha}}^{x_j+\epsilon^{1/\alpha}}\ud z  \: \Psi_{r,z}^2(t,x_j)\xi_j^2\\
&\quad +2\sum_{j=1}^d\sum_{i\ne j}\int_{t-\epsilon}^t \ud r \int_{x_j-\epsilon^{1/\alpha}}^{x_j+\epsilon^{1/\alpha}}\ud z  \: \Psi_{r,z}(t,x_j)\xi_j
\Psi_{r,z}(t,x_i)\xi_i\\
&=:I_\epsilon^*(\xi)+2I_\epsilon^{(1)}(\xi),
\end{align*}
where we have applied the following inequality
\[
\left(\sum_{i=1}^d a_i\right)^2=
 a_j^2 +2\sum_{i\ne j}a_ia_j + \left(\sum_{i\ne j} a_i\right)^2
\ge a_j^2 +2\sum_{i\ne j}a_ia_j.
\]
Then by the inequality $(a+b)^2\ge \frac{2}{3}a^2-2b^2$, from \eqref{E:SPDE-Du}, we see that
\begin{align*}
I_\epsilon^*(\xi)\ge&
\quad \frac{2}{3}\sum_{j=1}^d\int_{t-\epsilon}^t \ud r \int_{x_j-\epsilon^{1/\alpha}}^{x_j+\epsilon^{1/\alpha}}\ud z  \: \rho(u(r,z))^2 G(t-r,x_j-z)^2\xi_j^2\\
&-2 \sum_{j=1}^d\int_{t-\epsilon}^t \ud r \int_{x_j-\epsilon^{1/\alpha}}^{x_j+\epsilon^{1/\alpha}}\ud z  \:
Q_{r,z}^2(t,x_j)\xi_j^2\\
\ge&
\quad \frac{2}{3}\sum_{j=1}^d\int_{t-\epsilon}^t \ud r \int_{x_j-\epsilon^{1/\alpha}}^{x_j+\epsilon^{1/\alpha}}\ud z  \:
\rho(u(t,x_j))^2 G(t-r,x_j-z)^2\xi_j^2\\
& -\frac{2}{3}\sum_{j=1}^d\int_{t-\epsilon}^t \ud r \int_{x_j-\epsilon^{1/\alpha}}^{x_j+\epsilon^{1/\alpha}}\ud z  \:
\left|\rho(u(r,z))^2-\rho(u(t,x_j))^2 \right| G(t-r,x_j-z)^2\xi_j^2\\
&-2 \sum_{j=1}^d\int_{t-\epsilon}^t \ud r \int_{x_j-\epsilon^{1/\alpha}}^{x_j+\epsilon^{1/\alpha}}\ud z  \:
Q_{r,z}^2(t,x_j)\xi_j^2\\
=:&
\: \frac{2}{3}I_\epsilon^{(0)}(\xi)-\frac{2}{3}I_\epsilon^{(2)}(\xi) -2I_\epsilon^{(3)}(\xi).
\end{align*}
Hence,
\[
(\det \sigma)^{1/d}
\ge\inf_{|\xi|=1}
\InPrd{\sigma \xi,\xi}\ge
\frac{2}{3}\inf_{|\xi|=1} I_\epsilon^{(0)}(\xi)-
2\sum_{i=1}^3 \sup_{|\xi|=1}|I_\epsilon^{(i)}(\xi)|.
\]
Now, we will consider each of the four terms separately.

{\bigskip\bf\noindent Step 1.~} We first consider $I_\epsilon^{(0)}(\xi)$. 
Let $\gamma$ and $\beta$ be the constants in condition \eqref{E:lip}.
Denote 
\[
f_{\gamma,\beta}(x):=\exp\left\{-2\beta
\left[\log\frac{1}{|x|\wedge 1}\right]^\gamma
\right\},\quad\text{for $x\in\R$.}
\]
Notice that for $|\xi|=1$, by \eqref{E:lip}, for some $\beta>0$,
\begin{align*}
I_\epsilon^{(0)}(\xi) & \ge
\lip_\rho^2 \sum_{j=1}^d\int_{t-\epsilon}^t \ud r \int_{x_j-\epsilon^{1/\alpha}}^{x_j+\epsilon^{1/\alpha}}\ud z\:
f_{\beta,\gamma}(u(t,x_j))
G(t-r,x_j-z)^2\xi_j^2\\
&\ge
\lip_\rho^2 
\left[\inf_{x\in K} f_{\beta,\gamma}(u(t,x))\right] \sum_{j=1}^d\int_{t-\epsilon}^t \ud r \int_{x_j-\epsilon^{1/\alpha}}^{x_j+\epsilon^{1/\alpha}}\ud z\:
 G(t-r,x_j-z)^2\xi_j^2\\
 &=
\lip_\rho^2 \left[\inf_{x\in K} f_{\beta,\gamma}(u(t,x))\right] \sum_{j=1}^d\int_{t-\epsilon}^t \ud r \int_{-\epsilon^{1/\alpha}}^{\epsilon^{1/\alpha}}\ud z\:
 G(t-r,z)^2 \xi_j^2\\
 &=
\lip_\rho^2 \left[\inf_{x\in K} f_{\beta,\gamma}(u(t,x))\right] \int_{t-\epsilon}^t \ud r \int_{-\epsilon^{1/\alpha}}^{\epsilon^{1/\alpha}}\ud z\:
 G(t-r,z)^2.
\end{align*}
By the scaling property of $G(t,x)$ (see \eqref{E:ScaleG}), we see that
\begin{align*}
\int_{t-\epsilon}^t \ud r \int_{-\epsilon^{1/\alpha}}^{\epsilon^{1/\alpha}}\ud z\:
 G(t-r,z)^2 &=
\int_{t-\epsilon}^{t} \ud r\: (t-r)^{-1/\alpha} \int_{|y|\le \frac{\epsilon^{1/\alpha}}{(t-r)^{1/\alpha}}}\ud y\:
 G(1,y)^2\\
 &\ge
 \int_{t-\epsilon}^{t} \ud r\: (t-r)^{-1/\alpha} \int_{|y|\le 1}\ud y\:
 G(1,y)^2\\
 &= C \epsilon^{1-1/\alpha}.
\end{align*}
Therefore,
\[
\inf_{|\xi|=1}\left|I_\epsilon^{(0)}(\xi)\right|\ge C_0\: \epsilon^{1-1/\alpha} 
\inf_{x\in K} f_{\beta,\gamma}(u(t,x)).
\]

{\bigskip\bf\noindent Step 2.~} Now we consider $I_\epsilon^{(1)}(\xi)$.
Because $|\xi_i|\le 1$, by Minkowski's inequality, for $p\ge 2$,
\begin{align*}
\Norm{\sup_{|\xi|=1}\left|I_\epsilon^{(1)}(\xi)\right|}_p
\le&
\Norm{\sum_{j=1}^d\sum_{i\ne j}\int_{t-\epsilon}^t\ud r\int_{x_j-\epsilon^{1/\alpha}}^{x_j+\epsilon^{1/\alpha}}\ud z
|\Psi_{r,z}(t,x_i)\Psi_{r,z}(t,x_j)| }_p\\
\le&
\sum_{j=1}^d\sum_{i\ne j}\int_{t-\epsilon}^t\ud r\int_{x_j-\epsilon^{1/\alpha}}^{x_j+\epsilon^{1/\alpha}}\ud z
\Norm{\Psi_{r,z}(t,x_i)\Psi_{r,z}(t,x_j)}_p.
\end{align*}
Set
\begin{align}\label{E:CtK}
 C_{t,K}:=\sup_{(s,y)\in [t/2,t]\times K}\Norm{\rho(u(s,y))^2}_{2p}.
\end{align}
By H\"older's inequality,
for $r\in(t-\epsilon,t)\subseteq (t/2,t)$ and $z\in K$,
\[
\Norm{\Psi_{r,z}(t,x_i)\Psi_{r,z}(t,x_j)}_p\le
C_{t,K} \Norm{S_{r,z}(t,x_i)}_{4p}\Norm{S_{r,z}(t,x_j)}_{4p}.
\]
Then, by Lemma \ref{L:S} and by H\"older's inequality,
\begin{align*}
\int_{t-\epsilon}^t\ud r\int_{x_j-\epsilon^{1/\alpha}}^{x_j+\epsilon^{1/\alpha}}& \ud z
\Norm{\Psi_{r,z}(t,x_i)\Psi_{r,z}(t,x_j)}_p\\
\le&
C \int_{t-\epsilon}^t\ud r\: (t-r)^{-1/\alpha}\int_{x_j-\epsilon^{1/\alpha}}^{x_j+\epsilon^{1/\alpha}} \sqrt{G(t-r,x_i-z)G(t-r,x_j-z)}\ud z\\
\le&
C \int_{0}^\epsilon \ud r\: r^{-1/\alpha}
\left(\int_{x_j-\epsilon^{1/\alpha}}^{x_j+\epsilon^{1/\alpha}} G(r,x_j-z)\ud z\right)^{1/2}
\left(\int_{x_j-\epsilon^{1/\alpha}}^{x_j+\epsilon^{1/\alpha}} G(r,x_i-z)\ud z\right)^{1/2}\\
\le &
C \int_{0}^\epsilon \ud r\: r^{-1/\alpha}
\left(\int_{x_j-\epsilon^{1/\alpha}}^{x_j+\epsilon^{1/\alpha}} G(r,x_i-z)\ud z\right)^{1/2}.
\end{align*}
Notice that for some constant $C>0$ (see (4.2) of \cite{ChenDalang15FracHeat}),
\begin{align*}
 \int_{x_j-\epsilon^{1/\alpha}}^{x_j+\epsilon^{1/\alpha}} G(r,x_i-z)\ud z
 &=
 r^{-1/\alpha}\int_{-\epsilon^{1/\alpha}}^{\epsilon^{1/\alpha}} G(1,r^{-1/\alpha}(x_i-x_j-z))\ud z\\
 &\le C r^{-1/\alpha} \int_{-\epsilon^{1/\alpha}}^{\epsilon^{1/\alpha}}
 \frac{1}{1+[r^{-1/\alpha}|x_i-x_j-z|]^{1+\alpha}} \ud z\\
& =C r \int_{-\epsilon^{1/\alpha}}^{\epsilon^{1/\alpha}}
 \frac{1}{r^{1+1/\alpha}+|x_i-x_j-z|^{1+\alpha}} \ud z.
\end{align*}
Thanks to \eqref{E:epsilon}, for $|z|\le \epsilon^{1/\alpha}$,
\[
|x_i-x_j-z|\ge |x_i-x_j|-|z| \ge 2 \epsilon^{\frac{\alpha-1}{\alpha(1+\alpha)}} - \epsilon^{1/\alpha}
=\epsilon^{\frac{\alpha-1}{\alpha(1+\alpha)}}\left(2-\epsilon^{\frac{2}{\alpha(1+\alpha)}}\right)
\ge
\epsilon^{\frac{\alpha-1}{\alpha(1+\alpha)}}.
\]
Hence,
\begin{align*}
 \int_{x_j-\epsilon^{1/\alpha}}^{x_j+\epsilon^{1/\alpha}} G(r,x_i-z)\ud z
 &\le C r
 \int_{-\epsilon^{1/\alpha}}^{\epsilon^{1/\alpha}}
 \frac{1}{r^{1+1/\alpha}+\epsilon^{1-1/\alpha}} \ud z\\
 &=
  \frac{2 C r\epsilon^{1/\alpha}}{r^{1+1/\alpha}+\epsilon^{1-1/\alpha}} \\
 &\le  \frac{2 C r \epsilon^{1/\alpha}}{2\sqrt{r^{1+1/\alpha}\epsilon^{1-1/\alpha}}}
 = C r^{\frac{1}{2}-\frac{1}{2\alpha}} \: \epsilon^{-\frac{1}{2} + \frac{3}{2\alpha}},
\end{align*}
and
\[
\int_{t-\epsilon}^t\ud r\int_{x_j-\epsilon^{1/\alpha}}^{x_j+\epsilon^{1/\alpha}} \ud z
\Norm{\Psi_{r,z}(t,x_i)\Psi_{r,z}(t,x_j)}_p\le
C' \epsilon^{-\frac{1}{4} + \frac{3}{4\alpha}}
\int_0^\epsilon r^{\frac{1}{4}-\frac{1}{4\alpha}}r^{-\frac{1}{\alpha}}\ud r
= C'' \epsilon^{1-\frac{1}{2\alpha}}.
\]
Therefore,
\[
\Norm{\sup_{|\xi|=1}\left|I_\epsilon^{(1)}(\xi)\right|}_p\le C\: \epsilon^{1-\frac{1}{2\alpha}}.
\]

{\bigskip\bf\noindent Step 3.~} Now we  consider $I_\epsilon^{(2)}(\xi)$.
Similarly,
\[
\Norm{\sup_{|\xi|=1}\left|I_\epsilon^{(2)}(\xi)\right|}_p
\le
\sum_{j=1}^d\int_{t-\epsilon}^t \ud r \int_{x_j-\epsilon^{1/\alpha}}^{x_j+\epsilon^{1/\alpha}}\ud z  \:
\Norm{\rho(u(r,z))^2-\rho(u(t,x_j))^2}_p G(t-r,x_j-z)^2.
\]
By the H\"older regularities of the solution
(see \cite{ChenKim14Comparison} for the case $\alpha\in (1,2)$ and \cite{ChenDalang14Holder} for the case $\alpha=2$),
\begin{align}\label{E:Holder}
\Norm{\rho(u(r,z))^2-\rho(u(t,x_j))^2}_p\le C_{t,K} \left((t-r)^{1-1/\alpha}+(z-x_j)^{\alpha-1}\right).
\end{align}
Hence,
\begin{align*}
\int_{t-\epsilon}^t \ud r &\int_{x_j-\epsilon^{1/\alpha}}^{x_j+\epsilon^{1/\alpha}}\ud z  \:
\Norm{\rho(u(r,z))^2-\rho(u(t,x_j))^2}_p G(t-r,x_j-z)^2\\
&\le C
\int_{t-\epsilon}^t \ud r \int_{x_j-\epsilon^{1/\alpha}}^{x_j+\epsilon^{1/\alpha}}\ud z  \:
 \left((t-r)^{1-1/\alpha}+\epsilon^{1-1/\alpha}\right) G(t-r,x_j-z)^2\\
&\le C
\int_{t-\epsilon}^t \ud r \left((t-r)^{1-1/\alpha}+\epsilon^{1-1/\alpha}\right)
\int_{\R}\ud z  \: G(t-r,x_j-z)^2\\
&\le C
\int_{t-\epsilon}^t \ud r \left((t-r)^{1-1/\alpha}+\epsilon^{1-1/\alpha}\right)
\widetilde{G}(2(t-r),0)\\
&\le C
\int_{t-\epsilon}^t \ud r \left((t-r)^{1-1/\alpha}+\epsilon^{1-1/\alpha}\right)(t-r)^{-1/\alpha}\\
&=C \: \epsilon^{2(1-1/\alpha)},
\end{align*}
where we have used the fact \eqref{E:TildeG}.
Therefore,
\[
\Norm{\sup_{|\xi|=1}\left|I_\epsilon^{(2)}(\xi)\right|}_p\le C\: \epsilon^{2(1-1/\alpha)}.
\]

{\bigskip\bf\noindent Step 4.~} Now we  consider $I_\epsilon^{(3)}(\xi)$.
Similarly,
\[
\Norm{\sup_{|\xi|=1}\left|I_\epsilon^{(3)}(\xi)\right|}_p
\le
\sum_{j=1}^d\int_{t-\epsilon}^t \ud r \int_{\R}\ud z  \:
\Norm{Q_{r,z}^2(t,x_j)}_p.
\]
Let $\lambda:=\sup_{x\in\R}|\rho'(x)|$.
By the Burkholder-Gundy-Davis inequality (see \eqref{E:BGD}),
\begin{align*}
\Norm{Q_{r,z}^2(t,x_j)}_p&=\Norm{Q_{r,z}(t,x_j)}_{2p}^2\\
&\le C_p \lambda^2 \int_r^t\ud s\int_\R\ud y \: G(t-s,x_j-y)^2 \Norm{\rho(u(r,z)) S_{r,z}(s,y)}_{2p}^2\,.
\end{align*}
By H\"older's inequality,
for $r\in(t-\epsilon,t)$ and $z\in K$,
\[
\Norm{\rho(u(r,z)) S_{r,z}(s,y)}_{2p}\le
C_{t,K}^{1/2} \Norm{S_{r,z}(s,y)}_{4p},
\]
where the constant $C_{t,K}$ is defined in \eqref{E:CtK}. Hence, by Lemma \ref{L:S},
\begin{align*}
\Norm{Q_{r,z}^2(t,x_j)}_p&\le C
\int_r^t\ud s\: (s-r)^{-1/\alpha}\int_\R\ud y \: G(t-s,x_j-y)^2 G(s-r,y-z)\\
&\le C
\int_r^t\ud s\: (s-r)^{-1/\alpha}(t-s)^{-1/\alpha}\int_\R\ud y \: G(t-s,x_j-y) G(s-r,y-z)\\
&\le C
\int_r^t\ud s\: (s-r)^{-1/\alpha}(t-s)^{-1/\alpha}G(t-r,x_j-z).
\end{align*}
Then integrating over $\ud z$ gives that
\begin{align*}
\int_{t-\epsilon}^t \ud r \int_{\R}\ud z  \:
\Norm{Q_{r,z}^2(t,x_j)}_p & \le
C\int_{t-\epsilon}^t\ud r\int_r^t\ud s\: (s-r)^{-1/\alpha}(t-s)^{-1/\alpha}\\
&=
C\int_{t-\epsilon}^t\ud r\: (t-r)^{1-2/\alpha}\\
&=
C\epsilon^{2(1-1/\alpha)}.
\end{align*}
Therefore,
\[
\Norm{\sup_{|\xi|=1}\left|I_\epsilon^{(3)}(\xi)\right|}_p\le C\: \epsilon^{2(1-1/\alpha)}.
\]

{\bigskip\bf\noindent Step 5.~} Finally, by choosing
\begin{align}\label{E:eta}
1<\eta <\min\left(2,\frac{2\alpha-1}{2(\alpha-1)}\right),
\end{align}
we see that for all $p>1$, 
\begin{align}\notag
\bbP\left((\det \sigma)^{1/d}< \epsilon^{(1-1/\alpha)\eta} \right)\le&
\quad \bbP\left(\frac{2}{3}\inf_{|\xi|=1}\left|I_\epsilon^{(0)}(\xi)\right|< 2 \epsilon^{(1-1/\alpha)\eta} \right)\\
\notag
&+
\sum_{i=1}^3 \bbP\left(2\sup_{|\xi|=1}\left|I_\epsilon^{(i)}(\xi)\right|> \frac{1}{3} \epsilon^{(1-1/\alpha)\eta} \right)\\
\notag
\le&
\quad \bbP\left(\inf_{x\in K} f_{\beta,\gamma}(u(t,x)) < C_0' \: \epsilon^{(1-1/\alpha)(\eta-1)} \right)\\
\label{E:4Prob}
&+
\sum_{i=1}^3 C_i' \epsilon^{p\left(\min\left(1-\frac{1}{2\alpha},2-\frac{2}{\alpha}\right) - (1-1/\alpha)\eta\right)}.
\end{align}
Notice that for any $\theta$ and $x\in (0,1)$,
\[
\exp\left\{-2\beta\left[\log \frac{1}{x}\right]^\gamma\right\}
< \theta \quad\Longleftrightarrow\quad
x <
\exp\left\{-(2\beta)^{-1/\gamma}\left[\log \frac{1}{\theta}\right]^{1/\gamma}\right\}.
\]
Hence, as $\epsilon$ is small enough, for some constant $C_0''>0$,
\begin{align*}
\bbP&\left(\inf_{x\in K} f_{\beta,\gamma}(u(t,x))< C_0' \: \epsilon^{(1-1/\alpha)(\eta-1)} \right)\\
&= \bbP\left(\inf_{x\in K} u(t,x) \wedge 1< 
\exp\left\{-C_0''\left(\frac{(1-1/\alpha)(\eta-1)}{2\beta}\right)^{\frac{1}{\gamma}}\left[\log\frac{1}{\epsilon}\right]^{\frac{1}{\gamma}}\right\}\right)\\
&= \bbP\left(\inf_{x\in K} u(t,x) < 
\exp\left\{-C_0''\left(\frac{(1-1/\alpha)(\eta-1)}{2\beta}\right)^{\frac{1}{\gamma}}\left[\log\frac{1}{\epsilon}\right]^{\frac{1}{\gamma}}\right\}\right).
\end{align*}
Then Theorem \ref{T:NegMom} implies  that for some $B>0$ which depends on $\alpha$, $\beta$, $\eta$, $\gamma$, $K$ and $t$, such that
\begin{align}\label{E:gepsilon}
 \bbP&\left(\inf_{x\in K} f_{\beta,\gamma}(u(t,x))< C_0' \: \epsilon^{(1-1/\alpha)(\eta-1)} \right)\le 
 \exp\left(-B \left[\log(1/\epsilon)\right]^{\frac{2-1/\alpha}{\gamma}}\right)=:g(\epsilon).
\end{align}
Because $\gamma\in (0,2-1/\alpha)$, $\lim_{\epsilon\downarrow 0}g(\epsilon)=0$.
The above choice of $\eta$ guarantees the following two inequalities
\[
\min\left(1-\frac{1}{2\alpha},2-\frac{2}{\alpha}\right) - (1-1/\alpha)\eta>0\quad
\text{and}\quad
(1-1/\alpha)(\eta-1)>0.
\]
Hence, 
\[
\lim_{\epsilon' \downarrow 0}\bbP\left((\det \sigma)^{1/d}< \epsilon' \right)=0,
\]
that is, $\det \sigma>0$ a.s. By part (1) of Proposition \ref{P:D1} we know that $u(t,x)\in \D^{1,2}$. Therefore, we can conclude from part (1) of Theorem \ref{T:Density2} that part (a) of Theorem \ref{T:Mult} is true.

As for part (b), since our choice of $S=[t/2,t]\times K$ is a compact set away from $t=0$, condition \eqref{E:supSNormU} is satisfied. 
Hence, part (3) of Proposition \ref{P:D1} implies that $u(t,x)\in\D^\infty_S$. In order to apply part (2) of Theorem \ref{T:Density2}, we still need to establish the existence of  nonnegative moments of $\det \sigma$. 
Because $\gamma\in (0,2-1/\alpha)$, the function $g(\epsilon)$ defined in \eqref{E:gepsilon} goes to zero as $\epsilon$ tends to zero faster than any $\epsilon^p$ with $p\ge 1$.
Thanks to \eqref{E:4Prob}, we can apply Lemma \ref{L:NegMom} to see that 
$\E\left[(\det \sigma)^{-p}\right]<\infty$ for all $p>0$.
Part (b) of Theorem \ref{T:Mult} is then proved by an application of
Theorem \ref{T:Density} (b).
This completes the  proof of Theorem \ref{T:Mult}.
\end{proof}

\subsection{Proof of Theorem \ref{T:LocDen}}\label{SS:LocDen}
\begin{proof}[Proof of Theorem \ref{T:LocDen}]
The proof of Theorem \ref{T:LocDen} is similar to that of Theorem \ref{T:Mult} with some minor changes.
Denote $F=(u(t,x_1),\dots,u(t,x_d))$ and
\[
\Omega_c := \{\omega: \: \rho(u(t,x_i,\omega))^2>c, i=1,\dots, d\},\quad\text{for $c>0$.}
\]
Following the notation in the proof of Theorem \ref{T:Mult},
we need only to prove that
\[
\lim_{\epsilon\rightarrow 0_+}\epsilon^{-k} \bbP\left(\left(\det\sigma\right)^{1/d}< \epsilon,
\: \Omega_c
\right) =0,\quad\text{for all $c>0$ and $k\in\bbN$}.
\]
Now fix an arbitrary $c>0$.
In Step 1 in the proof of Theorem \ref{T:Mult}, we bound $I_\epsilon^{(0)}(\xi)$ simply as follows
\[
I_\epsilon^{(0)}(\xi)\ge c \lip_\rho^2 \sum_{j=1}^d\int_{t-\epsilon}^t \ud r \int_{x_j-\epsilon^{1/\alpha}}^{x_j+\epsilon^{1/\alpha}}\ud z\:
G(t-r,x_j-z)^2\xi_j^2 \ge C \epsilon^{1-1/\alpha} \quad\text{a.s. on $\Omega_c$}
\]
The upper bounds for $I_\epsilon^{(i)}(\xi)$, $i=1,2,3$, are the same as those in Steps 2-4 of the proof of
Theorem \ref{T:Mult}. Then by the same choice of $\eta$ as that in Step 5 of the proof of Theorem \ref{T:Mult},
\begin{align*}
\bbP\left((\det \sigma)^{1/d}< \epsilon^{(1-1/\alpha)\eta},\:\Omega_c \right)\le&
\quad \bbP\left(\frac{2}{3}\inf_{|\xi|=1}\left|I_\epsilon^{(0)}(\xi)\right|< 2 \epsilon^{(1-1/\alpha)\eta},\: \Omega_c \right)\\
&+
\sum_{i=1}^3 \bbP\left(2\sup_{|\xi|=1}\left|I_\epsilon^{(i)}(\xi)\right|> \frac{1}{3} \epsilon^{(1-1/\alpha)\eta} \right).
\end{align*}
Then the remaining part of proof is the same as that in Step 5 of the proof of Theorem \ref{T:Mult}.
This completes the proof of Theorem \ref{T:LocDen}.
\end{proof}

\begin{remark}\label{R:H1-H4}
In Hu {\it et al} \cite{HHNS14}, \textcolor{\myred}{the smoothness of density is studied for a general class of the} second order stochastic partial differential equations with a centered Gaussian noise that is white in time and homogeneous in space;
see Section \ref{S:HHNS} for a brief account of these results.
Our equation here falls into this class.
The main contribution of Theorem \ref{T:LocDen} is the measure-valued initial condition. Actually, if initial data is bounded, one can prove Theorem \ref{T:LocDen} through verifying conditions (H1)--(H4) of Theorem \ref{T:HHNS} and Remark \ref{R:HHNS}.
In our current case, the correlation function $f$ in \eqref{E:f} is the Dirac delta function $\delta_0$. (H1) is true because 
\[
\int_0^T \int_\R |\calF G(t,\cdot)(\xi)|^2 \ud \xi \ud t
= C \int_0^T \int_\R G(t,x)^2 \ud x\ud t= C T^{1-1/\alpha}<\infty,
\]
and  $\sup_{t\in[0,T]}\int_\R G(t,x)\ud x \equiv 1$. (H2) is true thanks to Theorem 1.6 of \cite{ChenKim14Comparison} with $\kappa_1:=(1-\alpha)/(2\alpha)$ and $\kappa_2:=(\alpha-1)/2$.
Because $\int_0^\epsilon\ud r \int_\R \ud x\: G(t,x)^2 =C \epsilon^{1-1/\alpha}$, (H3) is satisfied with any $\epsilon_0>0$ and $\eta:=1-1/\alpha>0$. Finally, for (H4),
for any $\epsilon>0$, 
\[
\int_0^\epsilon \ud r \: r^{\kappa_1} \int_\R \ud x \: G(r,x)^2
= C 
\int_0^\epsilon r^{\kappa_1 -1/\alpha} \ud r = C \epsilon^{1+\kappa_1 -1/\alpha} = C \epsilon^{3(\alpha-1)/(2\alpha)},
\]
from where we can choose any $\epsilon_1>0$ and $\eta_1:=3(\alpha-1)/(2\alpha)>\eta$. Since by \eqref{E:TildeG} and \eqref{E:BddG}, for $w\ne 0$,
\begin{align*}
\int_0^\epsilon \ud r \int_\R \ud x \: G(r,x)G(r,w+x)
&\le C \int_0^\epsilon \ud r \int_\R \ud x \: \widetilde{G}(r,x)\widetilde{G}(r,w+x)\\
&= C \int_0^\epsilon \widetilde{G}(r,w) \ud r =C \int_0^\epsilon r^{-1/\alpha}\widetilde{G}\left(1,\frac{w}{r^{1/\alpha}}\right) \ud r \\
&= C \int_0^\epsilon r^{-1/\alpha} \frac{1}{1+|w|^{1+\alpha} r^{-(1+\alpha)/\alpha}}\ud r \\
& = C \int_0^\epsilon \frac{r}{r^{1+1/\alpha}+|w|^{1+\alpha}}\ud r \le \frac{C}{|w|^{1+\alpha}} \epsilon^{2},
\end{align*}
we can choose any $\epsilon_2>0$ and $\eta_2:=2>\eta$. Finally, because
\[
x^{\kappa_2} G(t,x) \le C t^{-1/\alpha} \frac{x^{(\alpha-1)/2}}{1+|x|^{1+\alpha} t^{-(1+1/\alpha)}}
=C t^{(\alpha-3)/(2\alpha)} \sup_{y>0}\frac{y^{(\alpha-1)/2}}{1+y^{1+\alpha}}\le C t^{(\alpha-3)/(2\alpha)},
\]
where we have used the fact that $0<(\alpha-1)/2<1+\alpha$, we see that
\begin{align*}
\int_0^\epsilon \ud r \int_\R \ud x \: x^{\kappa_2}G(r,x)G(r,w+x)
&\le C \int_0^\epsilon \ud r \: r^{(\alpha-3)/(2\alpha)}\int_\R \ud x \: G(r,w+x)\\
&= C \int_0^\epsilon \ud r \: r^{(\alpha-3)/(2\alpha)}
= C \epsilon^{3(\alpha-1)/(2\alpha)}.
\end{align*}
Therefore, one can choose any $\epsilon_3>0$ and $\eta_3:=\eta_1= 3(\alpha-1)/(2\alpha) > \eta$. With this, we have verified all conditions (H1)--(H4) of Theorem 3.1 and Remark 3.2 in \cite{HHNS14}.
\end{remark}

\section{Strict positivity of density}\label{S:Pos}

We will first establish two general  criteria for the strict positivity of the density in Section \ref{SS:TwoCri}.
Then we will prove Theorem \ref{T:Pos} in Section \ref{SS:Pos} by verifying these two criteria.
Some technical lemmas are gathered in Section \ref{SS:Tech}.
\bigskip

We first introduce some notation.
We will use bold letters to denote vectors.
Denote
\[
|\mathbf{x}| := \sqrt{x_1^2+\dots+x_d^2},
\]
for $\mathbf{x}\in\R^d$ and
\[
\Norm{A}:=\left(\sum_{i_1,\dots,i_m = 1}^d A^2_{i_1,\dots,i_m}\right)^{1/2}
\]
for $A\in (\R^d)^{\otimes m}$ and $m\ge 2$ (i.e., the Frobenius norm when $m=2$).

\bigskip
Recall that $W=\{W_t,t\ge 0\}$ can be viewed as a cylindrical Wiener process in the Hilbert space $L^2(\R)$ with
the covariance given by \eqref{E:Cov}.
Let $\mathbf{h}=(h^1,\dots,h^d) \in L^2(\R_+\times\R)^d$ and $\mathbf{z}=(z_1,\dots,z_d)\in\R^d$.
Define a translation of $W_t$, denoted by  $\widehat{W}_t$, as follows:
\begin{equation}
\widehat{W}_t(g):=
W_t(g) + \sum_{i=1}^d z_i\int_0^t \int_\R h^i(s,y) g(y) \ud s \ud y, \quad\text{for any $g\in L^2(\R)$.}\label{e.def-hat-W}
\end{equation}
Then $\big\{\: \widehat{W}_t, \: t\ge 0\big\}$ is a cylindrical Wiener process in $L^2(\R)$ on the
probability space $(\Omega,\calF,\widehat{\mathbb{P}})$, where
\[
\frac{\ud \widehat{\mathbb{P}}}{\ud \mathbb{P}}=\exp\left(
-\sum_{i=1}^d z_i \int_0^\infty\int_\R h^i(s,y)W(\ud s, \ud y)-\frac{1}{2}\sum_{i=1}^dz_i^2\int_0^\infty\int_\R
|h^i(s,y)|^2\ud s\ud y
\right).
\]
For any predictable process $Z\in L^2(\Omega\times\R_+; L^2(\R))$, we have that
\[
\int_0^\infty \int_\R Z(s,y)\widehat{W}(\ud s,\ud y) =
\int_0^\infty \int_\R Z(s,y)W(\ud s,\ud y) +\sum_{i=1}^d z_i\int_0^\infty\int_\R Z(s,y)h^i(s,y)\ud s\ud y.
\]
In the following, we write $\sum_{i=1}^d z_i h^i(s,y)=:\InPrd{\mathbf{z},\mathbf{h}(s,y)}$.
Let $\widehat{u}_{\mathbf{z}}(t,x)$ be the solution to \eqref{E:FracHt} with respect to $\widehat{W}$, that is,
\begin{align}
\begin{aligned}
 \widehat{u}_{\mathbf{z}}(t,x) = J_0(t,x)&+ \int_0^t\int_\R G(t-s,x-y) \rho(\widehat{u}_{\mathbf{z}}(s,y))W(\ud s,\ud y)\\
 &+\int_{0}^t\int_\R G(t-s,x-y)\rho(\widehat{u}_{\mathbf{z}}(s,y)) \InPrd{\mathbf{z},\mathbf{h}(s,y)}\ud s\ud y.
\end{aligned}
\end{align}
Then, the law of $u(t,x)$ under $\mathbb{P}$ coincides with the law of $\widehat{u}_{\mathbf{z}}(t,x)$ under $\widehat{\mathbb{P}}$.

\subsection{Two criteria for strict positivity of densities}\label{SS:TwoCri}

We first state a technical lemma, which can be found in \cite[Lemma 3.2]{BP98} or \cite[Lemma 4.2.1]{Nualart95}.

\begin{lemma}\label{L:diff}
 For any $\beta_1>0$, $\beta_2>0$ and $\kappa>0$, there exist nonnegative constants $R$ and $\alpha$ such that
 any mapping $g:\R^d\mapsto\R^d$ that satisfies
 \begin{enumerate}[parsep=0ex,topsep=1ex]
  \item[(1)] $|\det g'(0)|\ge 1/\beta_1$
  \item[(2)] $\sup_{|\mathbf{z}|\le\kappa}\left(|g(\mathbf{z})|+\Norm{g'(\mathbf{z})}+\Norm{g''(\mathbf{z})}\right)\le \beta_2$
 \end{enumerate}
is a diffeomorphism from a neighborhood of zero contained in the ball $B(0,R)$ onto the ball $B(g(0),\alpha)$.
\end{lemma}

To each $\mathbf{z}\in\R^d$ and $\mathbf{h}\in L^2(\R_+\times\R)^d$,
let $\widehat{W}_t^{\mathbf{z}}$ be the translation of $W_t$ with respect to $\mathbf{z}$ and $\mathbf{h}$ defined by \eqref{e.def-hat-W}.
As a consequence of Lemma 2.1.4 in \cite{Nualart06},
for any $d$-dimensional random vector $F$, measurable with respect to $W$,
with each component of $F$ belonging to $\D^{3,2}$,
both $\partial_{\mathbf{z}}F(\widehat{W}^{\mathbf{z}})$ and
$\partial_{\mathbf{z}}^2F(\widehat{W}^{\mathbf{z}})$ exist
and are continuous  with respect to $z$.
We shall explain the continuity.
 Fix $i\in \{1,\dots,d\}$ and set $g=\sum_{j\ne i} z_j h_j$.
By  Lemma 2.1.4 of \cite{Nualart06}, we have that $F(\widehat{W}^{\mathbf{z}})= F^{\InPrd{\mathbf{z},\mathbf{h}}}$ and
$\left\{F^{\InPrd{\mathbf{z},\mathbf{h}}} = F^{z_i h_i + g}, z_i\in\R \right\}$ has a version that is absolutely continuous with respect to the Lebesgue measure on $\R$ and $\partial_{z_i} F^{\InPrd{\mathbf{z},\mathbf{h}}}=\InPrd{DF,h_i}^{\InPrd{\mathbf{z},\mathbf{h}}}_{H}$,  where $H=L^2(\R)$.
Now fix another $i'\in \{1,\dots,d\}$ and set
$g'=\sum_{j\ne i'}z_j h_j$. Another application of Lemma 2.1.4 of \cite{Nualart06}
 shows that
\[
\left\{\partial_{z_i}F^{\InPrd{\mathbf{z},\mathbf{h}}}=\InPrd{DF,h_i}^{z_{i'}h_{i'}+g'}_{H},z_{i'}\in\R \right\}
\]
has a version that is absolutely continuous with respect to the Lebesgue measure on $\R$ and
\[
\partial_{z_i,z_{i'}}^2 F^{\InPrd{\mathbf{z},\mathbf{h}}} = \InPrd{D^2 F,h_i\otimes h_{i'}}_{H^{\otimes 2}}^{(z_ih_i+g)\otimes (z_{i'}h_{i'}+g')}.
\]  
Since $F\in \mathbb{D}^{3,2}$,  one can apply  Lemma 2.1.4 of \cite{Nualart06} for a third time to show that for any $i''\in \{1,\dots,d\}$
there is a continuous version of $z_{i''}\mapsto \partial_{z_i,z_{i'}}^2 F^{\InPrd{\mathbf{z},\mathbf{h}}}$. Therefore, both $\partial_{\mathbf{z}}F(\widehat{W}^{\mathbf{z}})$ and
$\partial_{\mathbf{z}}^2F(\widehat{W}^{\mathbf{z}})$ exist and are continuous.

%
%

The following theorem is an extension of Theorem 3.3 by Bally and Pardoux \cite{BP98}.
The difference is that we include a conditional probability in \eqref{E:my(ii)} below,
which is necessary in order to deal with the parabolic Anderson model.

\begin{theorem}\label{T:Criteria}
Let $F$ be a $d$-dimensional random vector measurable with respect to $W$,
such that each component of $F$ is in $\D^{3,2}$.
Assume that for some $f\in C(\R^d)$  and for some  open subset
 $\Gamma$ of $\R^d$, it holds that
\[
\one_\Gamma(\mathbf{y}) \: \left[\bbP\circ F^{-1}\right] (\ud \mathbf{y})
=
\one_\Gamma(\mathbf{y}) f(\mathbf{y}) \ud \mathbf{y}.
\]
Fix a point  $\mathbf{y}_*\in\Gamma$. Suppose that there exists a sequence $\{\mathbf{h}_n\}_{n\in\bbN} \subseteq  L^2(\R_+\times\R)^d$
such that the associated random field
$\phi_n(\mathbf{z})= F(\widehat{W}^{\mathbf{z},n})$
satisfies the following two conditions.
\begin{enumerate}[parsep=0ex,topsep=1ex]
\item[(i)] There are constants $c_0>0$ and $r_0>0$ such that for all $r\in (0,r_0]$,
the following limit holds true:
\begin{align}\label{E:my(i)}
\liminf_{n\rightarrow\infty}
\bbP\left(
|F-\mathbf{y}_*|\le r
\:\:\text{and}\:\:
|\det\partial_{\mathbf{z}} \phi_n(0)|\ge \frac{1}{c_0}
\right)>0.
\end{align}
\item[(ii)]  There are some constants $\kappa>0$  and $K>0$ such that
\begin{align}\label{E:my(ii)}
 \lim_{n\rightarrow\infty}
 \bbP\left(
 \sup_{|\mathbf{z}|\le \kappa}\Norm{\phi_n(\mathbf{z})}_{C^2} \le K
 \Bigg| \: |F-\mathbf{y}_*|\le r_0\right)=1,
\end{align}
where
\[
\Norm{\phi_n(\mathbf{z})}_{C^2}:=
|\phi_n(\mathbf{z})|+\Norm{\partial_{\mathbf{z}}\phi_n(\mathbf{z})}+\Norm{\partial_{\mathbf{z}}^2\phi_n(\mathbf{z})}.
\]
\end{enumerate}
Then $f(\mathbf{y_*})>0$.
\end{theorem}
\begin{proof}
Fix $\mathbf{y}_*\in\Gamma$, $\kappa$, $r_0$, $c_0$ and $K$.
Denote
\[
\Omega_{r,n}:=\left\{\omega:\: |F-\mathbf{y}_*|\le r,\;\;
|\det\partial_{\mathbf{z}} \phi_{n}(0)|\ge \frac{1}{c_0},\;\;
\sup_{|\mathbf{z}|\le\kappa}\Norm{\phi_{n}(\mathbf{z})}_{C^2}\le K
\right\}.
\]
By applying Lemma \ref{L:diff} with
$\beta_1=c_0$ and $\beta_2=K$,  we see that
there exist $\alpha>0$ and $R>0$ such that
for all $\omega\in\Omega_{r,n}$ and $n\in\bbN$,
the mapping $\mathbf{z}\mapsto\phi_n(\mathbf{z},\omega)$ is
a diffeomorphism between an open neighborhood $V_n(\omega)$ of $0$ in $\R^d$,
contained in some ball $B(0,R)$,
and the ball $B(F(\omega),\alpha)$.
When applying Lemma \ref{L:diff},
by restricting to  smaller neighborhoods (reducing values of $R$ and $\alpha$) when necessary,
we can and do assume that
\begin{align}\label{E_:D>bd}
\left|\det \partial_{\mathbf{z}}\phi_n(\mathbf{z},\omega)\right| \ge \frac{1}{2c_0},\quad\text{for all $\mathbf{z}\in V_n(\omega)$ and $\omega\in\Omega_{r_0, n}$,
$n\in\bbN$.}
\end{align}
Denote $r=r_0\wedge (\alpha/2)$.
\bigskip

We claim that
\begin{align}\label{E:OmegaRN}
 \liminf_{n\rightarrow\infty}P(\Omega_{r,n}) >0.
\end{align}
To see this, denote
\[
A=\bigg\{
\left|F-\mathbf{y}_*\right|\le r
\bigg\},\quad
B_n=\left\{|\det\partial_{\mathbf{z}} \phi_{n}(0)|\ge \frac{1}{c_0}
\right\},\quad
C_n=\bigg\{
\sup_{|\mathbf{z}|\le\kappa}\Norm{\phi_{n}(\mathbf{z})}_{C^2}\le K
\bigg\}.
\]
Notice that $P(A)>0$ and
\begin{align*}
P(B_n\cap C_n\big| A)  &=
1- P(B_n^c \cup C_n^c\big| A)\\
&\ge
1- P(B_n^c \big| A)
- P( C_n^c\big| A)\\
&=P(B_n \big| A)+ P( C_n\big| A)-1\,.
\end{align*}
From conditions \eqref{E:my(i)} and \eqref{E:my(ii)}, we see that
\begin{align*}
\liminf_{n\rightarrow\infty} P(B_n\cap C_n\big| A) &
\ge
\liminf_{n\rightarrow\infty} P(B_n \big| A)+ \liminf_{n\rightarrow\infty} P( C_n\big| A)-1\\
&=\liminf_{n\rightarrow\infty} P(B_n \big| A)>0.
\end{align*}
This proves \eqref{E:OmegaRN}.
Hence, there exists an $m \  (=m_r)\in\bbN$ such that
\[
\bbP\left(
\Omega_{r,m}\right)
>0.
\]
By \eqref{E_:D>bd}, we have that
\begin{align}\label{E2_:D>bd}
\Omega_{r,m}\subseteq
\Omega_{r_0,m}\subseteq
\left\{
\left|\det \partial_{\mathbf{z}}\phi_{m}(\mathbf{z},\omega)\right| \ge \frac{1}{2c_0}
\right\}.
\end{align}
From Girsanov's theorem,
\[
\left(\bbP\circ F^{-1}\right)(\ud \mathbf{z}) =
\rho_m(\mathbf{z})
\left(\bbP\circ \phi_{m}^{-1}\right)(\ud \mathbf{z}),
\]
where
\[
\rho_m(\mathbf{z}) = \exp\left(
\InPrd{\mathbf{z},W(\mathbf{h}_m)}-\frac{1}{2}\sum_{i=1}^d z_i^2 \Norm{h^i_m}_{L^2(\R_+\times\R)}^2
\right).
\]
Let $G(\mathbf{x})=(2\pi)^{-d/2}\exp\left(-|\mathbf{x}|^2/2\right)$ for $\mathbf{x}\in\R^d$. Then
for all $g \in C_b^{\infty}(\R^d;\R_+)$,
\begin{align*}
 \E\left[g(F)\right]&= \int_{\R^d} \E\left[g(F)\right] G(\mathbf{x}) \ud \mathbf{x}\\
 &=\int_{\R^d}  \E\left[\rho_m(\mathbf{x}) g(\phi_m(\mathbf{x}))\right] G(\mathbf{x}) \ud \mathbf{x}\\
 &\ge
 \E\left[ \int_{\R^d}  \rho_m(\mathbf{x}) g(\phi_m(\mathbf{x}))G(\mathbf{x}) \ud \mathbf{x};\: \Omega_{r,m}\right] \\
 &\ge
 \E\left[ \int_{V_m}  \rho_m(\mathbf{x}) g(\phi_m(\mathbf{x}))G(\mathbf{x}) \ud \mathbf{x};\: \Omega_{r,m} \right].
 \end{align*}
Thanks to \eqref{E2_:D>bd}, we see that
\[
\E\left[g(F)\right]
 \ge \E\left[
 \int_{B(F,\alpha)} g(\mathbf{y})
 \left(\frac{G\: \rho_m }{|\det \partial_{\mathbf{z}}\phi_m|}\right)\left(\phi_m^{-1}(\mathbf{y})\right) \ud \mathbf{y};\: \Omega_{r,m}\right]\\
 \ge
 \int_{\R^d}g(\mathbf{y}) \theta(\mathbf{y})\ud \mathbf{y},
\]
where
\[
\theta(\mathbf{y})=\E\left[\varphi(|F-\mathbf{y}|)\: \min\left(1,\left(\frac{G\: \rho_m }{|\det \partial_{\mathbf{z}}\phi_m|}\right)\left(\phi_m^{-1}(\mathbf{y})\right)\right);\Omega_{r,m}\right],
\]
and $\varphi:\R_+\mapsto[0,1]$ is a continuous function such that
\[
\one_{[0,r]}(t) \le \varphi(t) \le \one_{[0,\alpha]}(t).
\]
By the construction of $\theta$, we see that $\theta(\mathbf{y}_*)>0$.
Because the function
\[
\R^d \ni \mathbf{y}\mapsto
\varphi(|F-\mathbf{y}|)\: \min\left(1,\left(\frac{G\: \rho_m }{|\det \partial_{\mathbf{z}}\phi_m|}\right)\left(\phi_m^{-1}(\mathbf{y})\right)\right)
\]
is continuous a.s. on $\Omega_{r,m}$ and is bounded by $1$,
the dominated convergence theorem implies that $\theta$ is a continuous function.
Finally, we see that for all $g\in C_b^\infty(\R^d;\R_+)$ with $\mathbf{y}_*\in \spt{g}\subset\Gamma$,
\[
\int_{\R^d}g(\mathbf{y})f(\mathbf{y})\ud \mathbf{y}=\E[g(F)] \ge \int_{\R^d} g(\mathbf{y})\theta(\mathbf{y})\ud \mathbf{y}>0,
\]
which implies that $f(\mathbf{y}_*)>0$.
This completes the proof of Theorem \ref{T:Criteria}.
\end{proof}

\subsection{Proof of Theorem \ref{T:Pos}}\label{SS:Pos}

\begin{proof}[Proof of Theorem \ref{T:Pos}]
Choose and fix an arbitrary final time $T$. We will prove Theorem \ref{T:Pos}
for $t=T$.
Throughout the proof, we fix $\kappa>0$ and assume that $|\mathbf{z}|\le \kappa$ and $t\in (0,T]$.
Without loss of generality, one may assume that $T>1$ in order that $T-2^{-n}>0$ for all $n\ge 1$.
Otherwise we simply replace all ``$n\ge 1$" in the proof below by ``$n\ge N$" for some large $N>0$.
The proof consists of three steps.

{\bigskip\bf\noindent Step 1.~}
For $n\ge 1$, define $\mathbf{h}_n$ as follows:
\begin{align}\label{E:hn}
h_{n}^i(s,y):= c_n \one_{[T-2^{-n},T]}(s) \one_{[x_i-2^{-n},x_i+2^{-n}]}(y), \qquad\text{for $1\le i\le d$,}
\end{align}
where
\begin{align}\label{E:cni-1}
c_n^{-1} = \int_{T-2^{-n}}^T \ud s \int_{x_i-2^{-n}}^{x_i+2^{-n}} \ud y\: G(T-s,x_i-y)
= \int_0^{2^{-n}} \ud s \int_{-2^{-n}}^{2^{-n}} \ud y\: G(s,y).
\end{align}
Lemma \ref{L:LowG} below implies that for some universal constant $C>0$,
\begin{align}\label{E:cn-bd}
c_n \le C 2^{n (2-1/\alpha)},\quad\text{for all $n\ge 1$.}
\end{align}
See \eqref{E2:c2n} below for an explicit formula for $c_n$ when $\alpha=2$.

Let $\widehat{W}^n_{\mathbf{z}}$ be the cylindrical Wiener process translated by $\mathbf{h}_{n}$
and $\mathbf{z}=(z_1,\dots,z_d)$.
Let $\{\widehat{u}_{\mathbf{z}}^n(t,x), \:(t,x)\in (0,T]\times\R\}$ be
the random field shifted with respect to $\widehat{W}^n_{\mathbf{z}}$,
i.e., $ \widehat{u}_{\mathbf{z}}^n(t,x) $   satisfies the following equation:
\begin{align}\label{E:hatUn}
\begin{aligned}
 \widehat{u}_{\mathbf{z}}^n(t,x) = J_0(t,x)&+ \int_0^t\int_\R G(t-s,x-y) \rho(\widehat{u}_{\mathbf{z}}^n(s,y))W(\ud s,\ud y)\\
 &+\int_{0}^t\int_\R G(t-s,x-y)\rho(\widehat{u}_{\mathbf{z}}^n(s,y)) \InPrd{\mathbf{z},\mathbf{h}_n(s,y)}\ud s\ud y.
\end{aligned}
\end{align}

For $x\in\R$, denote the gradient vector and the Hessian matrix of $\widehat{u}_{\mathbf{z}}^n(t,x)$ by
\begin{align}\label{E:uij}
\widehat{u}^{n,i}_{\mathbf{z}}(t,x) := \partial_{z_i} \:\widehat{u}_{\mathbf{z}}^n (t,x)
\quad\text{and}\quad
\widehat{u}^{n,i,k}_{\mathbf{z}}(t,x) := \partial^2_{z_i z_k} \:\widehat{u}_{\mathbf{z}}^n (t,x),
\end{align}
respectively.
From \eqref{E:hatUn}, we see that
\[
\widehat{u}^n_{\mathbf{z}}(s,y)=u(s,y) \quad\text{for $s\le T-2^{-n}$ and $y\in\R$.}
\]
Hence, $\{\widehat{u}^{n,i}_{\mathbf{z}}(t,x),(t,x)\in(0,T]\times \R\}$ satisfies
\begin{align}
\label{E:hatUni}
\begin{aligned}
\widehat{u}^{n,i}_{\mathbf{z}}(t,x) =&
\theta_{\mathbf{z}}^{n,i}(t,x) \\
&+
\one_{\{t>T-2^{-n}\}}\int_{T-2^{-n}}^t \int_\R G(t-s,x-y)\rho'(\widehat{u}_{\mathbf{z}}^n(s,y))
\widehat{u}^{n,i}_{\mathbf{z}}(s,y) W(\ud s,\ud y)\\
 &+
\int_{0}^t \int_\R G(t-s,x-y)\rho'(\widehat{u}_{\mathbf{z}}^n(s,y))
\widehat{u}^{n,i}_{\mathbf{z}}(s,y)\InPrd{\mathbf{z},\mathbf{h}_n(s,y)} \ud s\ud y,
\end{aligned}
\end{align}
where
\begin{align}\label{E:thetazni}
\theta_{\mathbf{z}}^{n,i}(t,x) = \int_{0}^t \int_\R G(t-s,x-y)\rho(\widehat{u}_{\mathbf{z}}^n(s,y)) h_n^i(s,y)\ud s\ud y.
\end{align}
Similarly, $\{\widehat{u}^{n,i,k}_{\mathbf{z}}(t,x),(t,x)\in(0,T]\times \R\}$ satisfies
\begin{align}
\notag
 \widehat{u}^{n,i,k}_{\mathbf{z}}(t,x) =&\theta^{n,i,k}_{\mathbf{z}}(t,x)
 +\int_{0}^t\int_\R G(t-s,x-y)\rho'(\widehat{u}^{n}_{\mathbf{z}}(s,y)) \widehat{u}^{n,i}_{\mathbf{z}}(s,y) h_n^k(s,y)\ud s\ud y\\
\notag
 &+\one_{\{t>T-2^{-n}\}}\int_{T-2^{-n}}^t\int_\R G(t-s,x-y)\rho''(\widehat{u}^{n}_{\mathbf{z}}(s,y)) \widehat{u}^{n,i}_{\mathbf{z}}(s,y)\widehat{u}^{n,k}_{\mathbf{z}}(s,y) W(\ud s,\ud y)\\
\notag
 &+\int_{0}^t\int_\R G(t-s,x-y)\rho''(\widehat{u}^{n}_{\mathbf{z}}(s,y)) \widehat{u}^{n,i}_{\mathbf{z}}(s,y)\widehat{u}^{n,k}_{\mathbf{z}}(s,y) \InPrd{\mathbf{z},\mathbf{h}_n(s,y)}\ud s\ud y\\
\notag
 &+\one_{\{t>T-2^{-n}\}}\int_{T-2^{-n}}^t\int_\R G(t-s,x-y)\rho'(\widehat{u}^{n}_{\mathbf{z}}(s,y)) \widehat{u}^{n,i,k}_{\mathbf{z}}(s,y) W(\ud s,\ud y)\\
\label{E:hatUnik}
 &+\int_{0}^t\int_\R G(t-s,x-y)\rho'(\widehat{u}^{n}_{\mathbf{z}}(s,y)) \widehat{u}^{n,i,k}_{\mathbf{z}}(s,y) \InPrd{\mathbf{z},\mathbf{h}_n(s,y)}\ud s\ud y,
\end{align}
where
\begin{align}\label{E:thetaznik}
\theta^{n,i,k}_{\mathbf{z}}(t,x):=\partial_{z_k}\theta^{n,i}_{\mathbf{z}}(t,x)=
\int_{0}^t \int_\R G(t-s,x-y)\rho'(\widehat{u}_{\mathbf{z}}^n(s,y)) \widehat{u}_{\mathbf{z}}^{n,k}(s,y) h_n^i(s,y)\ud s\ud y.
\end{align}
Note that the second term on the right-hand side of \eqref{E:hatUnik} is equal to
$\theta^{n,k,i}_{\mathbf{z}}(t,x)$.

Denote
\begin{align}
\label{E:C2Norm}
\begin{aligned}
 &\hspace{-3em}\Norm{\left\{\widehat{u}^{n}_{\mathbf{z}}(t,x_i)\right\}_{1\le i\le d}}_{C^2}\\
&=
\left|\left\{\widehat{u}^{n}_{\mathbf{z}}(t,x_i)\right\}_{1\le i\le d}\right|
+
\Norm{\left\{\widehat{u}^{n,i}_{\mathbf{z}}(t,x_j)\right\}_{1\le i,j\le d}}
+
\Norm{\left\{\widehat{u}^{n,i,k}_{\mathbf{z}}(t,x_j)\right\}_{1\le i,j,k\le d}}.
\end{aligned}
\end{align}

Suppose that
$\mathbf{y}\in\R^d$ belongs to the interior
of the support of the joint law of
\[
(u(T,x_1),\dots, u(T,x_d))
\]
and $\rho(y_i)\ne 0$ for all $i=1,\dots,d$.
By Theorem \ref{T:Criteria},
Theorem \ref{T:Pos}  is proved once
we show that
there exist some positive constants $c_1$, $c_2$, $r_0$, and $\kappa$ such that
the following two conditions are satisfied:
\begin{align}\label{E:(i)}
 \liminf_{n\rightarrow\infty}\bbP\Bigg(\left| \left\{u(T,x_i)-y_i\right\}_{1\le i\le d}\right| \le r\:\:\text{and}\:\:
 \left| \det\left[\left\{\widehat{u}_0^{n,i}(T,x_j)\right\}_{1\le i,j\le d}\right]\right|\ge c_1\Bigg)>0,
\end{align}
for all $r\in (0,r_0]$, and
\begin{align}\label{E:(ii)}
\lim_{n\rightarrow\infty}\bbP\left(
\sup_{|\mathbf{z}|\le \kappa}
\Norm{\left\{\widehat{u}^{n}_{\mathbf{z}}(T,x_i)\right\}_{1\le i\le d}}_{C^2}
\le c_2
\:\:\Bigg|\:\:
\left|\left\{u(T,x_i)-y_i\right\}_{1\le i\le d}\right|\le r_0
\right) =1.
\end{align}

These two conditions are verified in the following two steps.
\bigskip


{\bigskip\noindent\bf Step 2.~} Let $\mathbf{y}$ be a point in the intersection of $\{\rho\ne 0\}^d$  and
the interior of the support of the joint law of $(u(T,x_1),\dots,u(T,x_d))$.
Then there exists $r_0\in (0,1)$ such that for all $0<r\le r_0$,
\[
\mathbb{P}\left(
\Big\{(u(T,x_1),\dots, u(T,x_d))\in B(\mathbf{y},r)\Big\}
\cap
\Big\{\prod_{i=1}^d \left|\rho(u(T,x_i))\right|\ge 2c_0\Big\}\right)>0,
\]
where
\[
c_0 = \frac{1}{2}\inf_{(z_1,\dots,z_d) \in B(\mathbf{y},r_0)}\prod_{i=1}^d \left|\rho(z_i)\right|.
\]
Due to \eqref{E2:UniU-Close} below, it holds that
\[
\lim_{n\rightarrow\infty}
\det\left[ \left\{\widehat{u}_{0}^{n,i}(T,x_j)\right\}_{1\le i,j\le d}\right]
= \prod_{i=1}^d \rho\left(u(T,x_i)\right)
\qquad\text{a.s.}
\]
Hence, by denoting
\begin{align*}
A&:=
\Big\{(u(T,x_1),\dots, u(T,x_d))\in B(\mathbf{y},r)\Big\},\\
D &:=\Big\{\prod_{i=1}^d \left|\rho(u(T,x_i))\right|\ge 2c_0\Big\},\\
E_n&:=\left\{
\left|\det\left[ \left\{\widehat{u}_{0}^{n,i}(T,x_j)\right\}_{1\le i,j\le d}\right]
- \prod_{i=1}^d \rho\left(u(T,x_i)\right)\right| <c_0\right\},\quad\text{and}\\
G_n & :=\Big\{\left|\det\left[\{\widehat{u}_{0}^{n,i}(T,x_j)\}_{1\le i,j\le d}\right]\right|
\ge c_0\Big\},
\end{align*}
we see that
\begin{align*}
 \bbP\left(A\cap G_n\right)
 &  \ge
 \bbP\left(A\cap D \cap E_n\right) \rightarrow \bbP(A\cap D)>0,\quad\text{as $n\rightarrow\infty$.}
\end{align*}
Therefore,
\begin{align*}
& \liminf_{n\rightarrow\infty}\mathbb{P}\left(\Big\{(u(T,x_1),\dots, u(T,x_d))\in B(\mathbf{y},r)\Big\}
\cap \Big\{\left|\det\left[\{\widehat{u}_{0}^{n,i}(T,x_j)\}_{1\le i,j\le d}\right]\right|
\ge c_0\Big\}\right)>0,
\end{align*}
which proves condition \eqref{E:(i)}.

{\bigskip\noindent\bf Step 3.~}
From \eqref{E:C2Norm}, we see that
\begin{align*}
\Norm{\left\{\widehat{u}^{n}_{\mathbf{z}}(T,x_i)\right\}_{1\le i\le d}}_{C^2}
&\le
\sum_{i=1}^d |\widehat{u}^{n}_{\mathbf{z}}(T,x_i)|
+
\sum_{i,j=1}^d |\widehat{u}^{n,i}_{\mathbf{z}}(T,x_j)|
+
\sum_{i,j,k=1}^d |\widehat{u}^{n,i,k}_{\mathbf{z}}(T,x_j)|.
\end{align*}
By Proposition \ref{P:thetaBdd} below,
there exists some constant $K_{r_0}$ independent of $n$ such that
\[
\lim_{n\rightarrow\infty}\bbP\left(
\sup_{|\mathbf{z}|\le\kappa}\Norm{\left\{\widehat{u}^{n}_{\mathbf{z}}(T,x_i)\right\}_{1\le i\le d}}_{C^2} \le K_{r_0} \:\Bigg|\:
|\{u(T,x_i)-y_i\}_{1\le i\le d}|\le r_0
\right) =1,
\]
where $\kappa$ is fixed as at the beginning of the proof.
Therefore, condition \eqref{E:(ii)} is also satisfied.
This completes the proof of Theorem \ref{T:Pos}.
\end{proof}

\subsection{Technical propositions}\label{SS:Tech}
In this part, we will prove three Propositions \ref{P:Psi}, \ref{P:SupU} and \ref{P:thetaBdd} in three Subsections \ref{sss:Psi}, \ref{sss:moment} and \ref{sss:bound}, respectively.
Throughout this section, we fix a final time $T$ and we assume that $T\ge 1$ for our convenience.

\subsubsection{Properties of the function \texorpdfstring{$\Psi_n^i(t,x)$}{Lg}}
\label{sss:Psi}
For $t\in (0,T]$ and $x\in\R$, denote
\begin{align}
\label{E:PsiN}
\begin{aligned}
\Psi_n^i(t,x)&:=
\int_0^t\int_\R G(t-s,x-y)h_n^i(s,y)\ud s\ud y,\\
\Psi_n(t,x)&:=
\int_0^t\int_\R G(t-s,x-y)\InPrd{\mathbf{1},\mathbf{h}_n(s,y)}\ud s\ud y,\\
\end{aligned}
\end{align}
where $\mathbf{1}=(1,\dots,1)\in\R^d$.
The aim of this subsection is to prove the following proposition.

\begin{proposition}\label{P:Psi}
For $i\in\{1,\dots,d\}$, the following properties for the functions \texorpdfstring{$\Psi_n^i(t,x)$}{Lg}   hold:
\begin{enumerate}[parsep=0ex,topsep=1ex]
\item[(1)] The function $t\mapsto\Psi_n^i(t,x)$ is nondecreasing and
\begin{align}\label{E:hG0}
0\le \Psi_n^i(t,x)\le C \left(1-\min((T-t)2^n, 1)\right)^{1-1/\alpha},\quad\text{for all $(t,x)\in[0,T]\times\R$.}
\end{align}
 \item[(2)] For all $(t,x)\in (0,T]\times\R$ and $n\in\bbN$, it holds that
 \begin{align}\label{E:hG1}
\Psi_n^i(t,x)\le \Psi_n^i(T,x)\le C \Psi_n^i(T,x_i)
\quad\text{and}\quad \Psi_n^i(T,x_i)=1.
\end{align}
 \item[(3)] If $x\ne x_i$, then
 \begin{align}\label{E:hG2}
\Psi_n^i(T,x) \le C_x \: 2^{-n(1+1/\alpha)}\rightarrow 0,\quad\text{as $n\rightarrow\infty$.}
\end{align}
 \item[(4)] For all $(t,x)\in (0,T]\times\R$ and $n\in\bbN$,
 \begin{align}\label{E:hG3}
  \int_0^t\int_\R (t-s)^{-1/\alpha}G(t-s,x-y) \Psi_n^i(s,y)\ud s\ud y\le C 2^{-n(1-1/\alpha)},
 \end{align}
 where the constant $C$ does not depend on $t$, $x$, $n$ and $i$.
\end{enumerate}
\end{proposition}

See Figure \ref{fig:Psi} for some graphs of this function $\Psi_n(t,x)$.

\begin{figure}[h!tbp]
 \centering
 \subfloat[$n=1$]{\includegraphics[scale=0.4]{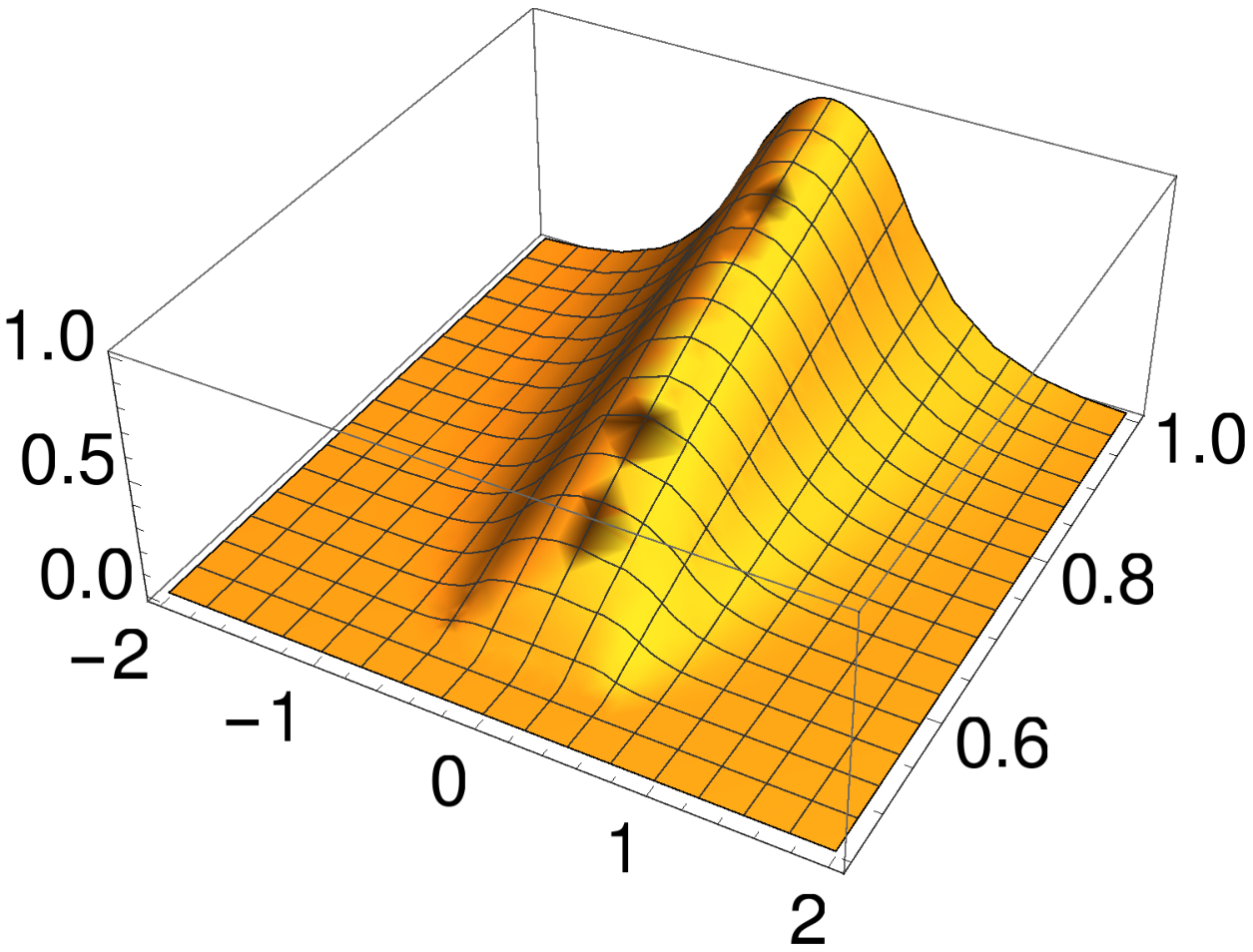}}%
 \hfill
 \subfloat[$n=2$]{\includegraphics[scale=0.4]{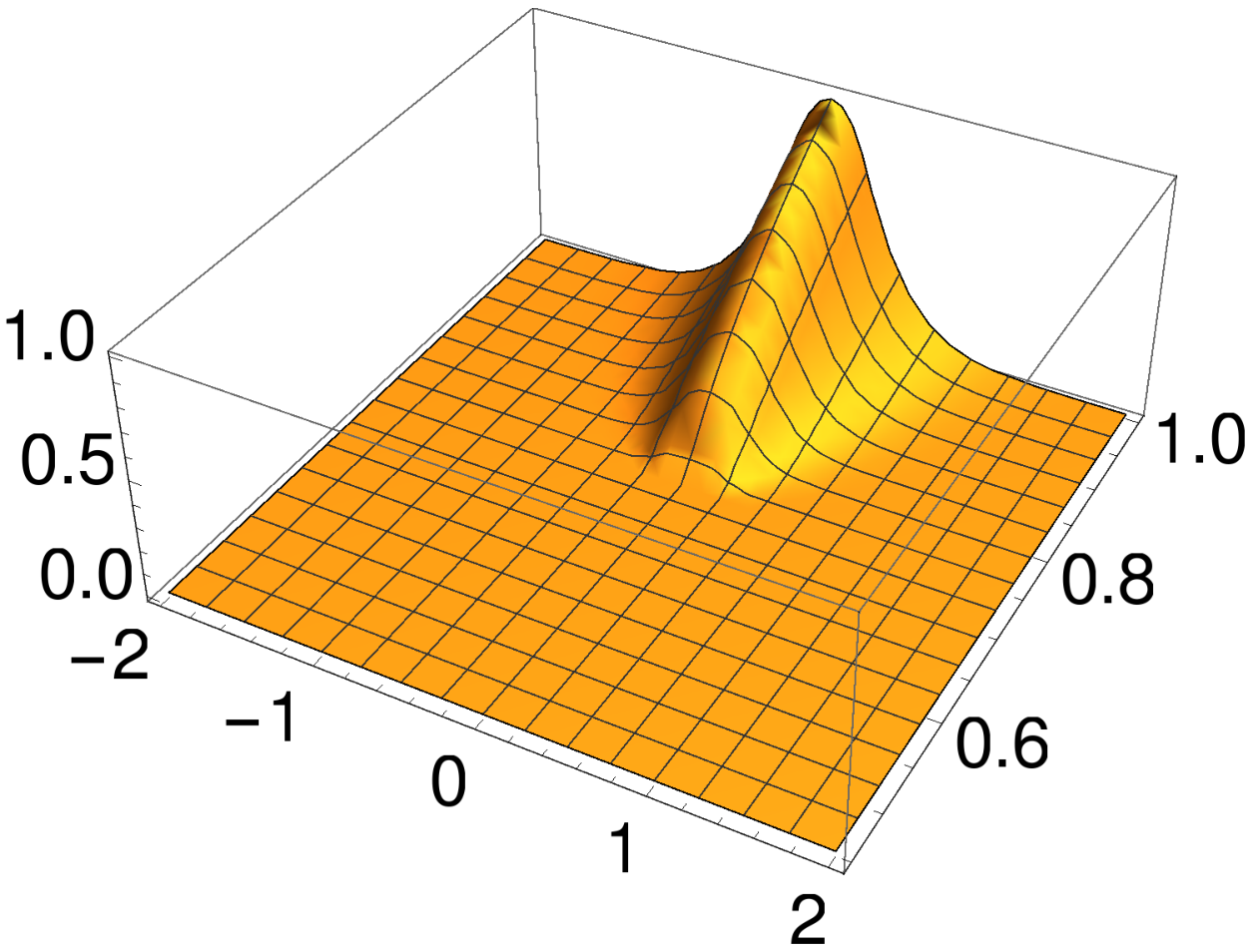}}
 \hfill
 \subfloat[$n=3$]{\includegraphics[scale=0.4]{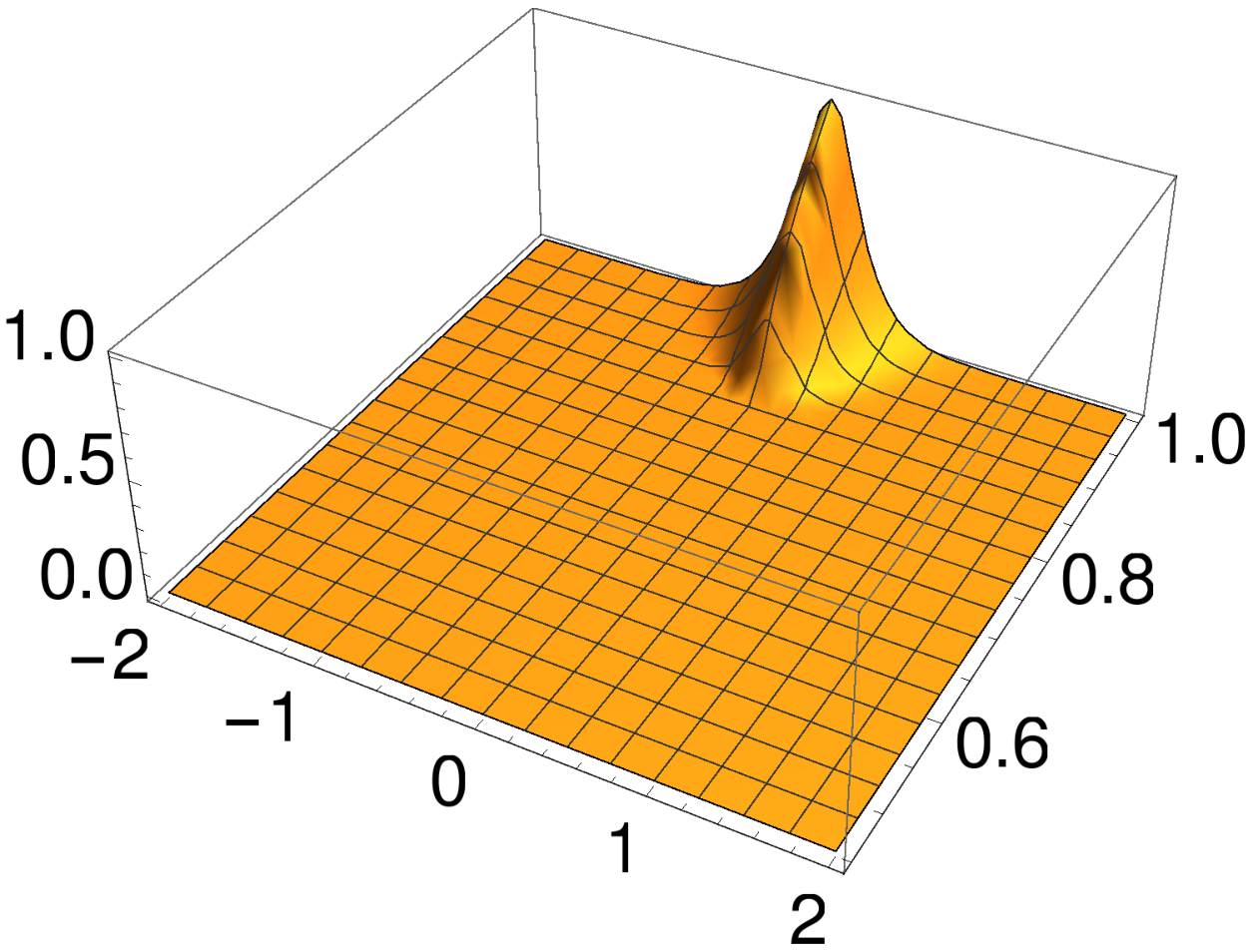}}
 \caption{$\Psi_n(t,x)$ with $\alpha=2$, $d=1$, $T=1$ and $x_1=0$
 for $t\in [0.45,1]$ and $x\in [-2,2]$.}
 \label{fig:Psi}
\end{figure}

We need one lemma.

\begin{lemma}\label{L:LowG}
If $\alpha\in [1,2]$, then there exists some constant $C=C(\alpha,\delta)>0$ such that
 \[
 \int_0^t\ud s\int_{-t}^t \ud y \: G(s,y)\ge C t^{2-1/\alpha}\quad\text{for all $t\in (0,1]$.}
 \]
\end{lemma}
\begin{proof}
Denote the integral by $I_t$. When $\alpha\ne 2$,
{by \eqref{E:ScaleG} and \eqref{E:BddG},}
\begin{align*}
 I_t &= \int_0^t\ud s\int_{-t}^t \ud y \: s^{-1/\alpha} G(1,s^{-1/\alpha}y)\\
 &\ge C
 \int_0^t\ud s\int_{-t}^t \ud y \: s^{-1/\alpha} \frac{1}{1+[s^{-1/\alpha}|y|]^{1+\alpha}}\\
 &\ge
 2C\int_0^t\ud s\int_{0}^t \ud y \: s^{-1/\alpha} \frac{1}{1+[s^{-1/\alpha} t]^{1+\alpha}}\\
 &= 2 C t\int_0^t \frac{ s }{s^{1+1/\alpha}+t^{1+\alpha}}\ud s\\
 &\ge 2 C t\int_0^t \frac{ s }{t^{1+1/\alpha}+t^{1+\alpha}}\ud s\\
 &=\frac{ C t^2 }{t^{1/\alpha}+t^{\alpha}}\ge \frac{C}{2} t^{2-1/\alpha},
\end{align*}
where in the last step we have used the fact that $t\in(0,1]$ and $\alpha\ge 1$.
When $\alpha=2$,
\begin{align*}
  I_t &= \int_0^t\ud s\int_{-t}^t \ud y \: \frac{1}{\sqrt{4\pi s}} e^{-\frac{y^2}{4s}}
  \ge 2t \int_{t/2}^t \frac{e^{-\frac{t^2}{4s}}}{\sqrt{4\pi s}} \ud s
  \ge 2t \int_{t/2}^t \frac{e^{-\frac{t^2}{4(t/2)}}}{\sqrt{4\pi t}} \ud s
  \ge C t^{3/2}.
\end{align*}
This completes the proof of Lemma \ref{L:LowG}.
\end{proof}

\begin{remark}\label{R:c2n}
Even though the explicit expression
for the double integral in Lemma \ref{L:LowG} as a function of $t$
is not needed in our later proof, it is interesting to
evaluate this double integral in some special case.
Actually, when $\alpha=2$, it is proved in Lemma \ref{L:c2n} with $\nu=2$ that
\begin{align}
\label{E:c2n}
\int_0^t\ud s\int_{-t}^t \ud y \: G(s,y)
&=t \left((t+2) \Phi\left(\sqrt{t/2}\right)-
   t+\sqrt{\frac{2}{\pi }} e^{-t/4}
   \sqrt{t/2}-1\right),
\end{align}
where $\Phi(x)=\int_{-\infty}^x(2\pi)^{-1/2}e^{-y^2/2}\ud y$.
By setting $t=2^{-n}$, we obtain an
explicit expression for $c_n$ defined in \eqref{E:cni-1}
\begin{align}\notag
c_n&=
2^{-n} \left( (2^{-n}+2) \Phi\left(2^{-\frac{1+n}{2}}\right)-
   2^{-n}+\sqrt{\frac{2}{\pi }} e^{-2^{-n-2}}
   2^{-\frac{1+n}{2}}-1\right)\\
&= \frac{2}{\sqrt{\pi}} \: 2^{-\frac{3n}{2}} -\frac{1}{2}2^{-2n}+ o(2^{-2n}),
\label{E2:c2n}
\end{align}
where
the last equality is due to the fact that
$\Phi(x)=1/2+(2\pi)^{-1/2}x+o(x)$ for small $x$.
\end{remark}
\bigskip

\begin{proof}[Proof of Proposition \ref{P:Psi}]
{\bigskip\noindent (1)~} Since the fundamental solutions are nonnegative, we see that $\Psi_n^i(t,x)\ge 0$.
Because
\begin{align}
\label{E:Psi-Rs}
 \Psi_n^i(t,x)&=c_n \one_{\{t>T-2^{-n}\}}\int_0^{2^{-n}-(T-t)}\ud s \int_{-2^{-n}}^{2^{-n}} \ud y \: G(s, x-x_i-y),
\end{align}
we see that the function $t\mapsto \Psi_n^i(t,x)$ is nondecreasing.
By the scaling property \eqref{E:ScaleG} of $G$ and \eqref{E:cn-bd},
\[
 \Psi_n^i(t,x) \le C
 2^{n(1-1/\alpha)}\one_{\{t>T-2^{-n}\}}\int_0^{2^{-n}-(T-t)}s^{-1/\alpha}\ud s
 =C \left(1-\min((T-t)2^n, 1)\right)^{1-1/\alpha}.
\]
\medskip

{\bigskip\noindent(2)~} From Step 1 we see $\Psi_n^i(t,x)\le \Psi_n^i(T,x)$.
By \eqref{E:TildeG},
$\Psi_n^i(T,x)\le C_{\alpha,\delta} (C_{\alpha,\delta}')^{-1} \Psi_n^i(T,x_i)$.
Finally, \eqref{E:cni-1} implies $\Psi_n^i(T,x_i)=1$. This proves \eqref{E:hG1}.
\medskip

{\bigskip\noindent(3)~}  As for \eqref{E:hG2}, notice that
\[
\Psi_n^i(t,x)\le c_n \int_0^{2^{-n}} \ud s\: s^{-1/\alpha} \int_{-2^{-n}}^{2^{-n}}\ud y \: G(1, s^{-1/\alpha}(x-x_i - y)).
\]
Suppose $n$ is large enough such that
$\Delta:=|x-x_i|>2^{1-n}$. Then for $|y|\le 2^{-n}<\Delta/2$,
\begin{align}\label{E_:Delta}
|x-x_i-y| \ge |x-x_i|-|y|\ge \Delta/2.
\end{align}
When $\alpha=2$, for $|y|\le 2^{-n}<\Delta/2$, we can write
\[
G(1, s^{-1/2}(x-x_i - y)) = \frac{1}{\sqrt{4\pi}}\exp\left(-\frac{(x-x_i-y)^2}{4s}\right)
\le \frac{1}{\sqrt{4\pi}}e^{-\frac{\Delta^2}{16s}}.
\]
Hence, for any $k>1$, by \eqref{E:cn-bd},
\begin{align*}
\Psi_n^i(t,x)&\le c_n 2^{1-n}
\int_0^{2^{-n}}  \frac{e^{-\frac{\Delta^2}{16s}}}{\sqrt{4\pi s} }\ud s\\
& \le c_n 2^{1-n} \left(\sup_{x\in[0,1]} x^{-k}e^{-\frac{\Delta^2}{16x}}\right) \int_0^{2^{-n}} \frac{u^k}{\sqrt{4\pi u}}\ud u\\
&= C\: c_n 2^{-n(k+3/2) }
\le C' 2^{-n k } \rightarrow 0,\quad\text{as $n\rightarrow\infty$.}
\end{align*}
In particular,
for the upper bound in \eqref{E:hG2} we may choose $k=3/2$.
When $\alpha\ne 2$, by \eqref{E_:Delta} and \eqref{E:cn-bd},
\begin{align*}
 \Psi_n^i(t,x)& \le C\: c_n \int_0^{2^{-n}}\ud s \: s^{-1/\alpha}\int_0^{2^{-n}}
 \frac{\ud y}{1+[s^{-1/\alpha}|x-x_i-y|]^{1+\alpha}}\\
 &=
 C\: c_n \int_0^{2^{-n}}\ud s \: s \int_0^{2^{-n}}
 \frac{\ud y}{s^{1+1/\alpha}+(\Delta/2)^{1+\alpha}}\\
 &\le C\: c_n \frac{2^{-n}}{(\Delta/2)^{1+\alpha}}\int_0^{2^{-n}}s \ud s \\
 & =
 C' \: c_n 2^{-3n} \le C'' 2^{-n(1+1/\alpha)} \rightarrow 0,\quad\text{as $n\rightarrow\infty$.}
\end{align*}
\medskip

{\bigskip\noindent(4)~} Denote the double integral in \eqref{E:hG3} by $I(t,x)$. By the semigroup property of $G$,
and by \eqref{E:ScaleG}, \eqref{E:BddG} and \eqref{E:Psi-Rs}, we see that
\begin{align*}
I(t,x)&= \int_0^t\ud s \: (t-s)^{-1/\alpha}\int_\R \ud y \: G(t-s,x-y) \int_0^s\ud r\int_\R \ud z\:
G(s-r,y-z)h_n^i(r,z)\\
&= \int_0^t\ud s \: (t-s)^{-1/\alpha}\int_0^s\ud r \int_\R \ud z \: G(t-r,x-z) h_n^i(r,z)\\
&=C \int_0^t\ud r\: (t-r)^{1-1/\alpha} \int_\R \ud z \: G(t-r,x-z) h_n^i(r,z) \\
&\le
C c_n \int_0^{2^{-n}}\ud r\: r^{1-1/\alpha} \int_{-2^{-n}}^{2^{-n}} \ud z \: r^{-1/\alpha}\\
&= C c_n 2^{-n(3-2/\alpha)} \le C 2^{-n(1-1/\alpha)}.
\end{align*}
This completes the proof of Proposition \ref{P:Psi}.
\end{proof}


A fact that we are going to use several times is that for any $x\in\R$ and $t\in [ T-2^{-n},T]$,
\begin{align}\notag
\int_{T-2^{-n}}^t\int_\R G^2(t-s,x-y)\ud s\ud y
&\le C_{\alpha,\delta}^2\int_0^{2^{-n}} \widetilde{G}(2s,0)\ud s
\\
\notag
&=C_{\alpha,\delta}^2 \widetilde{G}(1,0)\int_0^{2^{-n}} (2s)^{-1/\alpha}\ud s\\
&=C_{\alpha,\delta}^2 \widetilde{G}(1,0)2^{\frac{n(1-\alpha)-1}{\alpha}}
{= C 2^{-n(1-1/\alpha)},}
\label{E:G10}
\end{align}
{which goes to zero as $n\rightarrow\infty$.}

\subsubsection{Moments of \texorpdfstring{$\widehat{u}_{\mathbf{z}}^n(t,x)$}{Lg} and its first two derivatives}
\label{sss:moment}

The aim of this subsection is to prove the following proposition.

\begin{proposition}\label{P:SupU}
For all $\kappa>0$, $1\le i, k\le d$, $n\in\bbN$, 
{$p\ge 2$}, $t\in [0,T]$ and $x\in\R$, we have that
\begin{gather}\label{E:SupU}
\Norm{\sup_{|\mathbf{z}|\le \kappa}\left|\widehat{u}^n_{\mathbf{z}}(t,x)\right|}_p^2
\le C +C J_0^2(t,x) + C \left(J_0^2 \star \calG \right)(t,x),\\
\label{E2:SupU}
\Norm{\sup_{|\mathbf{z}|\le \kappa} \left|\widehat{u}^{n,i}_{\mathbf{z}}(t,x)\right|
}_p^2
\le C \left[2^{-n(1-1/\alpha)}+\Psi_n(t,x)^2\right],\\
\label{E3:SupU}
\Norm{\sup_{|\mathbf{z}|\le \kappa} \left|\widehat{u}^{n,i,k}_{\mathbf{z}}(t,x)\right|
}_p^2
\le C \left[2^{-n(1-1/\alpha)} + \Psi_n(t,x)^2\right],\\
\label{E4:SupU}
\Norm{\sup_{|\mathbf{z}|\le \kappa}\left| \theta^{n,i}_{\mathbf{z}}(t,x)\right|
}_p^2
\le C \Psi_n^i(t,x)^2,\\
\label{E5:SupU}
\Norm{\sup_{|\mathbf{z}|\le \kappa}\left| \theta^{n,i,k}_{\mathbf{z}}(t,x)\right|
}_p^2
\le C \:\Psi_n^i(t,x)^2,
\end{gather}
where the function $\calG(t,x)$ is defined in \eqref{E:calG}.
\end{proposition}
\begin{proof}
Because
\[
\Norm{\sup_{|\mathbf{z}|\le \kappa}
\left|\widehat{u}^n_{\mathbf{z}}(t,x)\right|}_p^2
\le
2\Norm{\sup_{|\mathbf{z}|\le \kappa}\left|\widehat{u}^n_{\mathbf{z}}(t,x)-\widehat{u}^n_{0}(t,x)\right|}_p^2
+2\Norm{\widehat{u}^n_{0}(t,x)}_p^2,
\]
one can apply the Kolmogorov continuity theorem (see Theorem \ref{T:KolCont}) and \eqref{E:IncHatU} to the first term
and apply \eqref{E:uPmnt} to the second term to see that
\[
\Norm{\sup_{|\mathbf{z}|\le \kappa}
\left|\widehat{u}^n_{\mathbf{z}}(t,x)\right|
}_p^2
\le
2 C\left[2^{-n(1-1/\alpha)}+\Psi_n(t,x)\right]
+2\left[C +C J_0^2(t,x) + C \left(J_0^2 \star \calG \right)(t,x)\right],
\]
which proves \eqref{E:SupU}.
In the same way, \eqref{E:UniLp0} and \eqref{E2:IncHatU} imply \eqref{E2:SupU};
\eqref{E:UnikLp0} and \eqref{E3:IncHatU} imply \eqref{E3:SupU}.
Now we show  \eqref{E4:SupU} and \eqref{E5:SupU}.
From \eqref{E:thetazni}, we see that
\[
\Norm{\sup_{|\mathbf{z}|\le \kappa}
\left|
\theta^{n,i}_{\mathbf{z}}(t,x)
\right|
}_p
\le C \int_0^t\int_\R G(t-s,x-y)
\Norm{\sup_{|\mathbf{z}|\le \kappa}
\left|
\widehat{u}^{n}_{\mathbf{z}}(s,y)
\right|
}_p h_n^i(s,y)\ud s\ud y.
\]
Then application of  \eqref{E:SupU} yields  \eqref{E4:SupU}.
Similarly, from \eqref{E:thetaznik}, since $\rho'$ is bounded,
\[
\Norm{\sup_{|\mathbf{z}|\le \kappa} \left|\theta^{n,i,k}_{\mathbf{z}}(t,x)
\right|
}_p
\le C \int_0^t\int_\R G(t-s,x-y)
\Norm{\sup_{|\mathbf{z}|\le \kappa}
\left|
\widehat{u}^{n,i}_{\mathbf{z}}(s,y)
\right|
}_p h_n^i(s,y)\ud s\ud y.
\]
Then we can apply \eqref{E2:SupU} to conclude \eqref{E5:SupU}.
This completes the proof of Proposition \ref{P:SupU}.
\end{proof}

In the next lemma, we study the moments of $\widehat{u}^n_{\mathbf{z}}(t,x)$.

\begin{lemma}\label{L:Moment-U}
For any $p\ge 2$, $(t,x)\in [0,T]\times\R$, $n\in\bbN$ and $\kappa>0$,
there exists some constant $C$ independent of $n$ such that
\begin{align}\label{E:uPmnt}
\sup_{n\in\bbN}\sup_{|\mathbf{z}|\le \kappa}\Norm{\widehat{u}_{\mathbf{z}}^n(t,x)}_p^2 \le
 C +C J_0^2(t,x) + C \left(J_0^2 \star \calG \right)(t,x),
\end{align}
where the function $\calG(t,x)$ is defined in \eqref{E:calG}.
As a consequence,
\begin{align}\label{E:theta-Lp0}
&\sup_{|\mathbf{z}|\le \kappa}\Norm{
\theta_{\mathbf{z}}^{n,i}(t,x)
}_p^2 \le C \Psi_n^i(t,x)^2,
\quad\text{and}\\
\label{E:theta-bd}
&\max_{1\le i\le d}\sup_{n\in\bbN} \sup_{(t,x)\in[0,T]\times \R}\sup_{|\mathbf{z}|\le \kappa}\Norm{\theta_{\mathbf{z}}^{n,i}(t,x)}_p < \infty,
\end{align}
and if $x\ne x_i$, then
\begin{align}\label{E:theta-AS0}
\lim_{n\rightarrow\infty}
\theta_{\mathbf{z}}^{n,i}(t,x)=0 \:\: \text{a.s. for all $t\in[0,T]$ and $|\mathbf{z}|\le\kappa$}.
\end{align}
\end{lemma}

\begin{proof}[Proof of Lemma \ref{L:Moment-U}]
We first  show that for each $n$,
$\widehat{u}_{\mathbf{z}}^n(t,x)$ is in $L^p(\Omega)$ for $p\ge 2$.
Actually, following the step 1 in the proof of Lemma 2.1.4 of \cite{Nualart06},  because for $|\mathbf{z}|\le \kappa$,
\begin{gather*}
\Norm{W(\InPrd{\mathbf{z},\mathbf{h}_n})}_2^2=
\int_{0}^t\int_\R |\InPrd{\mathbf{z},\mathbf{h}_n(s,y)}|^2 \ud s\ud y\quad\text{and}\\
\int_{0}^t\int_\R |\InPrd{\mathbf{z},\mathbf{h}_n(s,y)}|^2 \ud s\ud y
\le
 \sum_{i=1}^d \kappa^2\int_{0}^t\int_\R |h_n^i(s,y)|^2 \ud s\ud y \notag \le \kappa^2 d c_n^2 \:  2^{1-2n}
\le C  2^{2n(1-1/\alpha)},
\end{gather*}
where the last inequality is due to \eqref{E:cn-bd},
one can apply the H\"older inequality to obtain that
\begin{align*}
\Norm{\widehat{u}^n_{\mathbf{z}}(t,x)}_p
&=\Norm{u(t,x)\exp\left(\frac{1}{p}W(\InPrd{\mathbf{z},\mathbf{h}_n})-\frac{1}{2p}\int_{0}^t\int_\R |\InPrd{\mathbf{z},\mathbf{h}_n(s,y)}|^2 \ud s\ud y\right)}_p\\
&\le \Norm{u_{\mathbf{z}}(t,x)}_{2p} \:
\Norm{\exp\left(\frac{1}{p}W(\InPrd{\mathbf{z},\mathbf{h}_n})-\frac{1}{2p}\int_{0}^t\int_\R |\InPrd{\mathbf{z},\mathbf{h}_n(s,y)}|^2 \ud s\ud y\right)}_{2p}\\
&= \Norm{u(t,x)}_{2p}
\E\left[\exp\left(2W(\InPrd{\mathbf{z},\mathbf{h}_n})\right)\right]^{\frac{1}{2p}}
\exp\left(-\frac{1}{2p}\int_{0}^t\int_\R |\InPrd{\mathbf{z},\mathbf{h}_n(s,y)}|^2 \ud s\ud y\right)
\\ & =
\Norm{u(t,x)}_{2p}
\exp\left(\frac{1}{2p}\int_{0}^t\int_\R |\InPrd{\mathbf{z},\mathbf{h}_n(s,y)}|^2 \ud s\ud y\right)\\
&\le
\Norm{u(t,x)}_{2p} \exp\left(C 2^{2n(1-1/\alpha)}\right)<\infty.
\end{align*}
To obtain  a moment bound that is uniform in $n$,
it requires much more efforts.
Recall that $\widehat{u}^n_{\mathbf{z}}(t,x)$ satisfies the integral equation \eqref{E:hatUn}.
By the Minkowski and the Cauchy-Schwarz inequalities,
\begin{align}
\notag
&\hspace{-4em}
\sup_{|\mathbf{z}|\le \kappa}\Norm{\int_{0}^t\int_\R G(t-s,x-y)\rho(\widehat{u}_{\mathbf{z}}^n(s,y)) \InPrd{\mathbf{z},\mathbf{h}_n(s,y)}\ud s\ud y
}_p^2\\
\notag
\le&
\kappa^2\left(\int_{0}^t\int_\R G(t-s,x-y)
\sup_{|\mathbf{z}|\le \kappa}
\Norm{\rho(\widehat{u}_{\mathbf{z}}^n(s,y))}_p \:\InPrd{\mathbf{1}, \mathbf{h}_n(s,y)}\ud s\ud y\right)^2\\
\notag
\le &
\kappa^2\left(\int_{0}^t\int_\R G(t-s,x-y)
\sup_{|\mathbf{z}|\le \kappa}
\Norm{
\rho(\widehat{u}_{\mathbf{z}}^n(s,y))}_p^2 \:\InPrd{\mathbf{1}, \mathbf{h}_n(s,y)}\ud s\ud y\right)\\
\notag
&\qquad\times
\left(\int_{0}^t\int_\R G(t-s,x-y) \InPrd{\mathbf{1}, \mathbf{h}_n(s,y)}\ud s\ud y\right)\\
\label{E:DriftBd}
\le &
C d \kappa^2 \int_{0}^t\int_\R G(t-s,x-y)
\sup_{|\mathbf{z}|\le \kappa}
\Norm{
\rho(\widehat{u}_{\mathbf{z}}^n(s,y))
}_p^2 \: \InPrd{\mathbf{1},\mathbf{h}_n(s,y)}\ud s\ud y,
\end{align}
where the last inequality is due to \eqref{E:hG1}.
Denote
\begin{align}\label{E:Fkappa}
F_\kappa(t,x):=
\int_{0}^t\int_\R G(t-s,x-y)
\sup_{|\mathbf{z}|\le \kappa}
\Norm{\rho(\widehat{u}_{\mathbf{z}}^n(s,y))}_p^2 \: \InPrd{\mathbf{1},\mathbf{h}_n(s,y)}\ud s\ud y.
\end{align}
Hence, for some constant $C>0$ independent of $n$,
\begin{align*}
 \sup_{|\mathbf{z}|\le \kappa}
 \Norm{\widehat{u}^n_{\mathbf{z}}(t,x)
 }_p^2
 \le & \quad C J_0^2(t,x) + C \kappa^2 F_\kappa(t,x) \\
 &+C \int_0^t\ud s\int_\R \ud y\:
 G^2(t-s,x-y) \left(1+
 \sup_{|\mathbf{z}|\le \kappa}
 \Norm{\widehat{u}^n_{\mathbf{z}}(s,y)}_p^2\right).
\end{align*}
Then by applying Lemma \ref{L:Mom} with $\vv=1$, we see that
\begin{align}\label{E_:JFG}
\sup_{|\mathbf{z}|\le \kappa}
 \Norm{\widehat{u}^n_{\mathbf{z}}(t,x)}_p^2
 \le &\quad  C \left[ J_0^2(t,x) +\kappa^2 F_\kappa(t,x)+1+ \left((J_0^2+\kappa^2 F_\kappa)\star\calG\right)(t,x)\right],
\end{align}
for all $t\in [0,T]$, where $\calG$ is defined in \eqref{E:calG}. Notice that
\begin{align*}
 \left(F_\kappa\star\calG\right)(t,x) \le&
 C  \int_0^t \ud s\int_\R \ud y\:
 G(t-s,x-y) \frac{1}{(t-s)^{1/\alpha}}\int_0^r \ud r\\
 &\times \int_\R \ud w\:
 G(s-r,y-w) \left(1+\sup_{|\mathbf{z}|\le\kappa}\Norm{\widehat{u}^n_{\mathbf{z}}(r,w)}_p^2 \right) \InPrd{\mathbf{1},\mathbf{h}_n(r,w)}\\
 =&C \int_0^t \ud r \int_r^t\ud s\: \frac{1}{(t-s)^{1/\alpha}}\\
 &\times \int_\R \ud w\:
 G(t-r,x-w) \left(1+\sup_{|\mathbf{z}|\le\kappa}\Norm{\widehat{u}^n_{\mathbf{z}}(r,w)}_p^2\right) \InPrd{\mathbf{1},\mathbf{h}_n(r,w)}\\
 =&C+C \int_0^t \ud r \:  (t-r)^{1-1/\alpha}\\
 &\times \int_\R \ud w\: G(t-r,x-w) \sup_{|\mathbf{z}|\le\kappa}\Norm{\widehat{u}^n_{\mathbf{z}}(r,w)}_p^2 \InPrd{\mathbf{1},\mathbf{h}_n(r,w)}.
\end{align*}
Since $(t-r)^{1-1/\alpha}\le T^{1-1/\alpha}$, we see that for some constant $C>0$ (independent of $n$),
\begin{align}\label{E:FkBd2}
\left(F_\kappa\star\calG\right)(t,x)\le
C+
C\int_0^t \ud r \int_\R \ud w\:
 G(t-r,x-w) \sup_{|\mathbf{z}|\le\kappa}\Norm{\widehat{u}^n_{\mathbf{z}}(r,w)}_p^2 \InPrd{\mathbf{1},\mathbf{h}_n(r,w)}.
\end{align}
From \eqref{E:Fkappa}, we see that $F_\kappa(t,x)$ has an upper bound that is the same as that of $\left(F_\kappa\star\calG\right)(t,x)$
\begin{align}\label{E:FkBd}
F_\kappa(t,x)\le
C+
C\int_0^t \ud r \int_\R \ud w\:
G(t-r,x-w) \sup_{|\mathbf{z}|\le\kappa}\Norm{\widehat{u}^n_{\mathbf{z}}(r,w)}_p^2 \InPrd{\mathbf{1},\mathbf{h}_n(r,w)}.
\end{align}
Therefore, plugging these two upper bounds into \eqref{E_:JFG}, we see that
\begin{align}
\label{E:uIntLocal}
\begin{aligned}
\sup_{|\mathbf{z}|\le\kappa}
\Norm{\widehat{u}^n_{\mathbf{z}}(t,x)}_p^2
 \le& \quad C+C J_0^2(t,x) +  C\left(J_0^2\star\calG\right)(t,x)
    \\
 & + C \kappa^2\int_0^t \ud r \int_\R \ud w\:
 G(t-r,x-w) \sup_{|\mathbf{z}|\le\kappa}\Norm{\widehat{u}^n_{\mathbf{z}}(r,w)}_p^2 \InPrd{\mathbf{1},\mathbf{h}_n(r,w)}.
\end{aligned}
\end{align}
In order to solve this integral inequality, we first claim that
\begin{align}\label{E:lalBdd-Un}
\sup_{(t,x)\in [0,T]\times K} \left[ J_0^2(t,x) +  \left(J_0^2\star\calG\right)(t,x)\right] <\infty.
\end{align}
Actually, since $x_i$ are proper points with respect to the initial data $\mu$,
there is a compact set $K\subseteq \R$ such that $x_i\in K$ for $1\le i\le d$,
and $\mu$ restricted on $K$ has a bounded density. This fact together with
\eqref{E:BoundJ0} and \eqref{E:BoundR} yields \eqref{E:lalBdd-Un} easily
(note that $R$ in \eqref{E:BoundR} is defined by \eqref{E:Rtx} and
$\calG$ is defined by \eqref{E:calG}).
Hence, by denoting
\[
U^n_{\kappa,K,p}(t):= \sup_{x\in K} \sup_{|\mathbf{z}|\le\kappa}\Norm{\widehat{u}^n_{\mathbf{z}}(t,x)}_p^2,
\]
the inequality  \eqref{E:uIntLocal} can be rewritten as
\begin{align}
\begin{aligned}
 U^n_{\kappa,K,p}(t)
 & \le C+ C \kappa^2
\int_{0}^t\ud s \: U^n_{\kappa,K,p}(s) \int_\R \ud y\: G(t-s,x-y)
\InPrd{\mathbf{1},\mathbf{h}_n(s,y)}\\
&\le C+
  C  \kappa^2 c_n\: 2^{1-n} d\: \sup_{y\in\R} G(1,y)
\one_{\{t>T-2^{-n}\}}\int_{T-2^{-n}}^{t} (t-s)^{-1/\alpha} U^n_{\kappa,K,p}(s) \ud s,
\end{aligned}
\end{align}
where the constant $C>0$ independent of $n$ and $\kappa$.
Therefore, by \eqref{E:cn-bd},
there exists some constant $C>0$, independent of $n$, such that
\begin{align}
 \label{E:U}
 U^n_{\kappa,K,p}(t) \le C + C 2^{n(1-1/\alpha)} \one_{\{t>T-2^{-n}\}}\int_{T-2^{-n}}^{t} (t-s)^{-1/\alpha}  U^n_{\kappa,K,p}(s)\ud s.
\end{align}
Then by applying Lemma \ref{L:Contraction} to $U^n_{\kappa,K,p}(t)$ with $\epsilon=2^{-n}$,
we see that
\begin{align}\label{E:supxKu}
\sup_{n\in\bbN}\sup_{(t,x)\in [0,T]\times K}
\sup_{|\mathbf{z}|\le \kappa}
\Norm{\widehat{u}^n_{\mathbf{z}}(t,x)}_p^2
<\infty.
\end{align}
Hence, \eqref{E:DriftBd} and \eqref{E:supxKu} imply that
\[
\sup_{n\in\bbN}\sup_{(t,x)\in [0,T]\times K}
\sup_{|\mathbf{z}|\le \kappa}
\Norm{
\int_{0}^t\int_\R G(t-s,x-y)\rho(\widehat{u}_{\mathbf{z}}^n(s,y)) \InPrd{\mathbf{z},\mathbf{h}_n(s,y)}\ud s\ud y
}_p^2
<\infty.
\]
Then by taking $L^p(\Omega)$-norm on both sides of \eqref{E:hatUn}, we see that
\begin{align*}
\sup_{|\mathbf{z}|\le \kappa}
\Norm{\widehat{u}_{\mathbf{z}}^n(t,x)}_p^2
\le& C J_0^2(t,x) +C  +C \int_0^t\ud s\int_\R\ud y \: G^2(t-s,x-y)
\sup_{|\mathbf{z}|\le \kappa}
\Norm{  \widehat{u}_{\mathbf{z}}^n(s,y) }_p^2,
\end{align*}
for all $t\le T$, where the constant $C$ does not depend on $n$.
Then an application of Lemma \ref{L:Mom} proves \eqref{E:uPmnt}.

\bigskip
Now we study the moments of $\theta_{\mathbf{z}}^{n,i}(t,x)$ defined by  \eqref{E:thetazni}.
By \eqref{E:supxKu}, we see that
for all $t\le T$, $x\in\R$, and $n\in\bbN$,
\begin{align}\notag
 \sup_{|\mathbf{z}|\le \kappa}
 \Norm{\theta_{\mathbf{z}}^{n,i}(t,x)}_p^2 & \le
 \left(\int_{0}^t \int_\R G(t-s,x-y) \sup_{|\mathbf{z}|\le \kappa}\Norm{\rho(\widehat{u}_{\mathbf{z}}^n(s,y))}_p h_n^i(s,y)\ud s \ud y\right)^2\\
 &\le C
 \left(\int_{0}^t \int_\R G(t-s,x-y) h_n^i(s,y)\ud s \ud y\right)^2 = C \Psi_n^i(t,x)^2.
\label{E:thetaNorm}
\end{align}
Note that the constant $C$ in the above inequalities does not depend on $n$.
By Proposition \ref{P:Psi}, we see that both \eqref{E:theta-bd} and \eqref{E:theta-Lp0} hold.
When $x\ne x_i$, an application of \eqref{E:hG2} and the Borel-Cantelli lemma implies \eqref{E:theta-AS0}.
This completes the proof of Lemma \ref{L:Moment-U}
\end{proof}

\bigskip
In the next lemma, we study the moments of $\widehat{u}^{n,i}_{\mathbf{z}}(t,x)$.

\begin{lemma}\label{L:Moment-DU}
For any $p\ge 2$, $n\in\bbN$, $i=1,\dots, d$, and $\kappa>0$,
we have that
\begin{align}\label{E:UniLp0}
\sup_{|\mathbf{z}|\le\kappa}
\Norm{\widehat{u}^{n,i}_{\mathbf{z}}(t,x)}_p^2
\le C \left[2^{-n(1-1/\alpha)}
+ \Psi_n(t,x)^2\right],
\end{align}
and
\begin{align}\label{E:supxKui}
\max_{1\le i\le d} \sup_{n\in\bbN}\sup_{(t,x)\in [0,T]\times \R}
 \sup_{|\mathbf{z}|\le \kappa}\Norm{\widehat{u}^{n,i}_{\mathbf{z}}(t,x)}_p<\infty.
\end{align}
As a consequence,
\begin{align}\label{E:thetaIK-bd}
 &\max_{1\le i,k\le d} \sup_{|\mathbf{z}|\le \kappa}
 \Norm{
 \theta_{\mathbf{z}}^{n,i,k}(t,x)}_p^2 \le
 C \Psi_n^i(t,x)^2.
\end{align}
\end{lemma}

\begin{proof}
The proof consists four steps.

{\bigskip\noindent\bf Step 1.~}
First, we  show that $\widehat{u}_{\mathbf{z}}^{n,i}(t,x)\in L^p(\Omega)$ for $i=1,\dots,d$ and $p\ge 2$.
Notice that
\begin{multline}\label{E:rho'p2}
\sup_{|\mathbf{z}|\le \kappa}
\Norm{
\int_{0}^t\int_\R G(t-s,x-y)\rho'(\widehat{u}_{\mathbf{z}}^n(s,y)) \widehat{u}_{\mathbf{z}}^{n,i}(s,y) \InPrd{\mathbf{z},\mathbf{h}_n(s,y)}\ud s\ud y
}_p^2\\
\le
C c_n^2
\int_0^t\int_\R G^2(t-s,x-y)
\sup_{|\mathbf{z}|\le \kappa}
\Norm{
\widehat{u}_{\mathbf{z}}^{n,i}(s,y)
}_p^2 \ud s\ud y,
\end{multline}
where $c_n$ is defined in \eqref{E:cni-1}.
Hence, by \eqref{E:theta-bd} and \eqref{E:BGD},
\begin{align*}
\sup_{|\mathbf{z}|\le \kappa}
\Norm{\widehat{u}_{\mathbf{z}}^{n,i}(t,x)
 }_p^2 \le &  C
 + C_n'
 \int_{0}^t \int_\R G^2(t-s,x-y)
 \sup_{|\mathbf{z}|\le \kappa}
 \Norm{
 \widehat{u}_{\mathbf{z}}^{n,i}(s,y)}_p^2 \ud s\ud y.
\end{align*}
Therefore, Lemma \ref{L:Mom} implies that
\begin{align}\label{E:BndUni}
\sup_{1\le i\le d}
\sup_{|\mathbf{z}|\le \kappa}
\Norm{\widehat{u}_{\mathbf{z}}^{n,i}(t,x)}_p^2\le
C + C_n \left(C \star \calG \right)(t,x).
\end{align}

Next,  we shall obtain a bound that is uniform in $n$.
This requires much more efforts.
We shall  prove the three statements in the lemma in the remaining three steps:

{\bigskip\noindent\bf Step 2.~}
In this step, we will prove that the $p$-th moment of
$\sup_{|\mathbf{z}|\le \kappa}|\widehat{u}_{\mathbf{z}}^{n,i}(t,x)|$
is bounded uniformly in $n$, $i$, and $(t,x)\in [0,T]\times\R$.
Denote
\[
I_{\mathbf{z}}(t,x):=
\left|
\int_0^t\int_\R G(t-s,x-y)\rho'(\widehat{u}^n_{\mathbf{z}}(s,y))\widehat{u}^{n,i}_{\mathbf{z}}(s,y)
 \InPrd{\mathbf{z},\mathbf{h}_n(s,y)}\ud s\ud y
 \right|.
\]
Clearly, for $|\mathbf{z}|\le\kappa$,
\[
I_{\mathbf{z}}(t,x)\le \kappa \Norm{\rho'}_{L^\infty(\R)}\int_0^t\int_\R G(t-s,x-y)
 |\widehat{u}^{n,i}_{\mathbf{z}}(s,y)|
 \InPrd{\mathbf{1},\mathbf{h}_n(s,y)}\ud s\ud y.
\]
By the same arguments as those led to  \eqref{E:DriftBd},
\begin{align}
\begin{aligned}
\sup_{|\mathbf{z}|\le\kappa}
\Norm{I_{\mathbf{z}}(t,x)}_p^2
 \label{E:Step5}
 &\le C \kappa^2
 \int_0^t\int_\R G(t-s,x-y)
 \sup_{|\mathbf{z}|\le\kappa}
 \Norm{
 \widehat{u}^{n,i}_{\mathbf{z}}(s,y)}_p^2
 \InPrd{\mathbf{1},\mathbf{h}_n(s,y)}\ud s\ud y\\
 &=:  C \kappa^2 F_\kappa(t,x).
\end{aligned}
\end{align}
%
Then, taking the $L^p(\Omega)$-norm on both sides of \eqref{E:hatUni}
and applying \eqref{E:theta-bd}, \eqref{E:Step5} and \eqref{E:BGD} on the three parts, we see that for some constant $C>0$ that does not depend on $n$,
\begin{align}
\label{E:UniAs}
\begin{aligned}
\sup_{|\mathbf{z}|\le\kappa}
\Norm{
\widehat{u}^{n,i}_{\mathbf{z}}(t,x)}_p^2\le &
C [1+ \kappa^2 F_\kappa(t,x) ] \\
&\hspace{-2em}+ C
\one_{\{t>T-2^{-n}\}}
\int_{T-2^{-n}}^t\ud s\int_\R \ud y \:
G^2(t-s,x-y)
\sup_{|\mathbf{z}|\le\kappa}
\Norm{\widehat{u}^{n,i}_{\mathbf{z}}(s,y)}_p^2.
\end{aligned}
\end{align}
Hence, by replacing the bound $T-2^{-n}$ by $0$ in \eqref{E:UniAs} and applying Lemma \ref{L:Mom} with $\vv=0$,
\begin{align*}
\sup_{|\mathbf{z}|\le\kappa}
\Norm{\widehat{u}^{n,i}_{\mathbf{z}}(t,x)}_p^2& \le
C [1+ \kappa^2 F_\kappa(t,x) ]
+ C \kappa^2 \left(  F_\kappa \star\calG\right)(t,x)\\
&\le
C + C\kappa^2\int_0^t\ud s\int_\R \ud y \:
G(t-s,x-y)
\sup_{|\mathbf{z}|\le\kappa}
\Norm{\widehat{u}^{n,i}_{\mathbf{z}}(s,y)}_p^2 \InPrd{\mathbf{1},\mathbf{h}_n(s,y)},
\end{align*}
where the second inequality is due to \eqref{E:FkBd2} and \eqref{E:FkBd}.
Note that the constant $C$ does not depend on $n$.
Denote
\begin{align}\label{E:DD-AuxU}
U^{n,i}_{\kappa,p}(t):= \sup_{x\in \R} \sup_{|\mathbf{z}|\le\kappa}\Norm{\widehat{u}^{n,i}_{\mathbf{z}}(t,x)}_p^2.
\end{align}
Then one can show in the same way as those in the proof of Lemma \ref{L:Moment-U} that there exists some constant $C>0$, independent of both $n$ and $i$, such that
\[
U^{n,i}_{\kappa,p}(t) \le C + C 2^{n(1-1/\alpha)}\one_{\{t>T-2^{-n}\}}
\int_{T-2^{-n}}^{t} (t-s)^{-1/\alpha} U^{n,i}_{\kappa,p}(s)\ud s,\quad t\le T.
\]
Therefore, applying Lemma \ref{L:Contraction} to $\max_{1\le i\le d} U^{n,i}_{\kappa,p}(t)$ with $\epsilon=2^{-n}$,
we derive \eqref{E:supxKui}.

{\bigskip\noindent\bf Step 3.~}
Now we prove \eqref{E:UniLp0}.
By \eqref{E:theta-Lp0}, \eqref{E:G10} and \eqref{E:supxKui},
\begin{align*}
\sup_{|\mathbf{z}|\le\kappa}
\Norm{\widehat{u}^{n,i}_{\mathbf{z}}(t,x)}_p^2
\le &
\quad C  \sup_{|\mathbf{z}|\le\kappa}
\Norm{\theta^{n,i}_{\mathbf{z}}(t,x)}_p^2\\
&+C\one_{\{t>T-2^{-n}\}} \int_{T-2^{-n}}^t \int_\R G(t-s,x-y)^2\sup_{|\mathbf{z}|\le\kappa}
\Norm{\widehat{u}^{n,i}_{\mathbf{z}}(s,y)}_p^2 \ud s\ud y\\
&+C \kappa^2 \left(
\int_0^t\int_\R G(t-s,x-y)\sup_{|\mathbf{z}|\le\kappa}\Norm{\widehat{u}_{\mathbf{z}}^{n,i}(s,y)}_p \InPrd{\mathbf{1},
\mathbf{h}_n(s,y)}\ud s\ud y
\right)^2\\
\le& C \Psi_n^i(t,x)^2+ C 2^{-n(1-1/\alpha)} +C \kappa^2 \left(
\int_0^t\int_\R G(t-s,x-y)\InPrd{\mathbf{1},
\mathbf{h}_n(s,y)}\ud s\ud y
\right)^2\\
\le & C 2^{-n(1-1/\alpha)} + C \Psi_n(t,x)^2.
\end{align*}

{\bigskip\noindent\bf Step 4.~}
As for $\theta^{n,i,k}_{\mathbf{z}}(t,x)$ defined in \eqref{E:thetaznik}, by \eqref{E:supxKui}, we have that
\begin{align}
\notag
\sup_{|\mathbf{z}|\le\kappa}\Norm{\theta^{n,i,k}_{\mathbf{z}}(t,x)}_p^2
&\le C \left(\int_0^t\int_\R G(t-s,x-y) \sup_{|\mathbf{z}|\le\kappa}
\Norm{\widehat{u}^{n,k}_{\mathbf{z}}(s,y)}_p h_n^i(s,y)\ud s\ud y\right)^2\\
\notag
&\le
C \left(\int_0^t\int_\R G(t-s,x-y)  h_n^i(s,y)\ud s\ud y\right)^2\\
&\le
C \Psi_n^i(t,x)^2.
\label{E2:thetaNorm}
\end{align}
This completes the proof of Lemma \ref{L:Moment-DU}.
\end{proof}

\bigskip
In the next lemma, we study the moments of $\widehat{u}^{n,i,k}_{\mathbf{z}}(t,x)$.

\begin{lemma}\label{L:Moment-DDU}
For any $p\ge 2$, $n\in\bbN$, $1\le i,k\le d$, and $\kappa>0$,
we have that
\begin{align}\label{E:UnikLp0}
\sup_{|\mathbf{z}|\le\kappa}
\Norm{\widehat{u}^{n,i,k}_{\mathbf{z}}(t,x)}_p^2
\le C \left[2^{-n(1-1/\alpha)} + \Psi_n(t,x)^2 \right],
\end{align}
and
\begin{align}\label{E:supxKuik}
\max_{1\le i,k\le d} \sup_{n\in\bbN}\sup_{(t,x)\in [0,T]\times \R}
 \sup_{|\mathbf{z}|\le \kappa}\Norm{\widehat{u}^{n,i,k}_{\mathbf{z}}(t,x)}_p<\infty.
\end{align}
\end{lemma}
\begin{proof}
 Write the six parts of $\widehat{u}^{n,i,k}_{\mathbf{z}}(t,x)$
in \eqref{E:hatUnik} as
\begin{align}\label{E:Unik6}
\widehat{u}^{n,i,k}_{\mathbf{z}}(t,x) =
\theta^{n,i,k}_{\mathbf{z}}(t,x) +\theta^{n,k,i}_{\mathbf{z}}(t,x)+ \sum_{\ell=1}^4 U_\ell^n(t,x).
\end{align}
The moment bounds for $\theta^{n,i,k}_{\mathbf{z}}(t,x)$ are given by \eqref{E:thetaIK-bd}.
Since $\rho''$ is bounded, by the moments bound for $\widehat{u}^{n,i}_{\mathbf{z}}(t,x)$ as in \eqref{E:supxKui}
and by \eqref{E:G10},
\begin{align}\label{E:U13}
\sup_{|\mathbf{z}|\le\kappa}\Norm{U_1^n(t,x)}_p^2 \le C 2^{-n(1-1/\alpha)}.
\end{align}
Similarly, by \eqref{E:supxKui},
\[
\sup_{|\mathbf{z}|\le\kappa}\Norm{U_2^n(t,x)}_p^2 \le
C \Psi_n(t,x)^2.
\]
Then use the moment bounds for $U_3$ and $U_4$ to form an integral inequality similar to that in the proof of
Lemma \ref{L:Moment-U}.
By the same arguments as those in the proof of Lemma \ref{L:Moment-U},
one can show \eqref{E:supxKuik}. We leave the details for the interested readers.

As for \eqref{E:UnikLp0},
since $\rho'$ is bounded, by \eqref{E:supxKuik}, we see that
\[
\sup_{|\mathbf{z}|\le\kappa}\Norm{U_3^n(t,x)}_p^2 \le C 2^{-n(1-1/\alpha)}
\quad\text{and}\quad\sup_{|\mathbf{z}|\le\kappa}\Norm{U_4^n(t,x)}_p^2 \le C \Psi_n(t,x)^2.
\]
Combining these six bounds proves \eqref{E:UnikLp0}.
\end{proof}

\bigskip
The next lemma is  on  the H\"older continuity (in norm) of the
random fields $\mathbf{z}\mapsto \widehat{u}_{\mathbf{z}}^n(t,x)$ and its first two derivatives.

\begin{lemma}\label{L:HolderInZ}
For all $\kappa>0$, $1\le i, k\le d$, $n\in\bbN$, $t\in [0,T]$ and $x\in\R$, we have that
\begin{align}\label{E:IncHatU}
\sup_{|\mathbf{z}|\vee |\mathbf{z}'|\le \kappa}
\Norm{\widehat{u}^n_{\mathbf{z}}(t,x)-\widehat{u}^n_{\mathbf{z}'}(t,x)}_p^2
&\le C\left[2^{-n(1-1/\alpha)}+\Psi_n(t,x)\right]\:
|\mathbf{z}-\mathbf{z}'|^2,\\
\label{E2:IncHatU}
\sup_{|\mathbf{z}|\vee |\mathbf{z}'|\le \kappa}
\Norm{\widehat{u}^{n,i}_{\mathbf{z}}(t,x)-\widehat{u}^{n,i}_{\mathbf{z}'}(t,x)}_p^2
&\le C \left[2^{-n(1-1/\alpha)}+\Psi_n(t,x)^2\right] |\mathbf{z}-\mathbf{z}'|^2,\\
\label{E3:IncHatU}
\sup_{|\mathbf{z}|\vee |\mathbf{z}'|\le \kappa}
\Norm{\widehat{u}^{n,i,k}_{\mathbf{z}}(t,x)-\widehat{u}^{n,i,k}_{\mathbf{z}'}(t,x)}_p^2
&\le C \left[2^{-n(1-1/\alpha)} + \Psi_n(t,x)^2\right]|\mathbf{z}-\mathbf{z}'|^2.
\end{align}
\end{lemma}
\begin{proof}
We will prove these three properties in three steps.
We need only to prove the case when $t>T-2^{-n}$.
Hence, in the following proof, we assume that $t\in[T-2^{-n},T]$.

{\bigskip\noindent\bf Step 1.~}
In this step, we will prove \eqref{E:IncHatU}.
Assume that $|\mathbf{z}|\le \kappa$ and $|\mathbf{z}'|\le \kappa$.
Notice that
\begin{align*}
&  \Norm{\widehat{u}^n_{\mathbf{z}}(t,x)-\widehat{u}^n_{\mathbf{z}'}(t,x)}_p^2\\
 \le & \hspace{1.3em} 3 \LIP_\rho^2 4p \int_0^t \int_\R G^2(t-s,x-y)
\Norm{\widehat{u}^n_{\mathbf{z}}(s,y)-\widehat{u}^n_{\mathbf{z}'}(s,y)}_p^2 \ud s\ud y\\
&+3\Norm{\int_{0}^t\int_\R G(t-s,x-y)\rho(\widehat{u}^n_{\mathbf{z}}(s,y))\InPrd{\mathbf{z}-\mathbf{z}',\mathbf{h}_n(s,y)}\ud s \ud y}_p^2\\
&+3\Norm{\int_{0}^t\int_\R G(t-s,x-y)\left[\rho(\widehat{u}^n_{\mathbf{z}}(s,y))-\rho(\widehat{u}^n_{\mathbf{z}'}(s,y))\right]
\InPrd{\mathbf{z}',\mathbf{h}_n(s,y)}\ud s \ud y}_p^2\\
=: & \: 3 I_1 + 3 I_2 +3 I_3,
\end{align*}
where we have applied \eqref{E:BGD} for the bound in $I_1$.
By the same arguments as those led  to \eqref{E:DriftBd}, we see that
\begin{align*}
I_2\le & C |\mathbf{z}-\mathbf{z}'|^2
\left(\int_{0}^t \ud s \int_\R \ud y\: G(t-s,x-y) \left(1+ \Norm{\widehat{u}^n_{\mathbf{z}}(s,y)}_p\right)\InPrd{\mathbf{1},\mathbf{h}_n(s,y)}\right)^2,
\end{align*}
and
\begin{align}\label{E_:I3}
I_3 \le &
C \kappa^2\int_{0}^t \ud s \int_\R \ud y\: G(t-s,x-y) \Norm{\widehat{u}^n_{\mathbf{z}}(s,y)-\widehat{u}^n_{\mathbf{z}'}(s,y)}_p^2\InPrd{\mathbf{1},\mathbf{h}_n(s,y)}\\
=:& C \kappa^2 F_n(t,x),
\notag
\end{align}
where the constants $C$ do not depend on $n$, $t$ and $x$.
We claim that there exists a nonnegative constant $C$ independent of $n$, $t$ and $x$ such that
\begin{align}\label{E:boundI23}
\sup_{|\mathbf{z}|\vee |\mathbf{z}'|\le \kappa} I_i \le C \Psi_n(t,x)\: |\mathbf{z}-\mathbf{z}'|^2,\quad i=2,3.
\end{align}
The case $I_2$ is clear from \eqref{E:PsiN} and \eqref{E:uPmnt}.
As for $I_3$, since \eqref{E:boundI23} is true for $I_2$,
for some nonnegative constant $C$ independent of $n$, $t$ and $x$, it holds that
\begin{align}
\label{E:DuInt}
\begin{aligned}
 \Norm{\widehat{u}^n_{\mathbf{z}}(t,x)-\widehat{u}^n_{\mathbf{z}'}(t,x)}_p^2
\le & \: C\Psi_n(t,x)^2 |\mathbf{z}-\mathbf{z}'|^2 + C \kappa^2 F_n(t,x)\\
&+ C
\int_{0}^t \ud s \int_\R \ud y\: G^2(t-s,x-y) \Norm{\widehat{u}^n_{\mathbf{z}}(s,y)-\widehat{u}^n_{\mathbf{z}'}(s,y)}_p^2.
\end{aligned}
\end{align}
Therefore, by Lemma \ref{L:Mom}, for $t\in [0,T]$, it holds that
\[
\Norm{\widehat{u}^n_{\mathbf{z}}(t,x)-\widehat{u}^n_{\mathbf{z}'}(t,x)}_p^2
\le
 C\Psi_n(t,x)^2 |\mathbf{z}-\mathbf{z}'|^2 + C \kappa^2 F_n(t,x) +
 C \left(\left(|\mathbf{z}-\mathbf{z}'|^2 +  \kappa^2 F_n\right)\star\calG\right)(t,x),
\]
where $\calG$ is defined in \eqref{E:calG}.
By the same arguments as those led  to \eqref{E:uIntLocal}, we see that
\begin{align*}
 \Norm{\widehat{u}^n_{\mathbf{z}}(t,x)-\widehat{u}^n_{\mathbf{z}'}(t,x)}_p^2
 \le C\Psi_n(t,x)^2|\mathbf{z}-\mathbf{z}'|^2
 + &C\kappa^2
 \int_{0}^t  \ud s \int_\R \ud y\:  G(t-s,x-y) \\
 &\times \Norm{\widehat{u}^n_{\mathbf{z}}(s,y)-\widehat{u}^n_{\mathbf{z}'}(s,y)}_p^2\InPrd{\mathbf{1},\mathbf{h}_n(s,y)}.
\end{align*}
Now set
\[
U^n_{\mathbf{z},\mathbf{z}',K,p}(t):=\sup_{x\in K}\Norm{\widehat{u}^n_{\mathbf{z}}(t,x)-\widehat{u}^n_{\mathbf{z}'}(t,x)}_p^2.
\]
Hence, for some constants $C>0$,
\[
U^n_{\mathbf{z},\mathbf{z}',K,p}(t)\le C |\mathbf{z}-\mathbf{z}'|^2  + 2^{n(1-1/\alpha)}\lambda
\int_{T-2^{-n}}^{t}
(t-s)^{-1/\alpha} U^n_{\mathbf{z},\mathbf{z}',K,p}(s)\ud s.
\]
Then by applying Lemma \ref{L:Contraction},
\begin{align*}
\sup_{|\mathbf{z}|\vee |\mathbf{z}'|\le \kappa}\sup_{x\in K}
\Norm{\widehat{u}^n_{\mathbf{z}}(t,x)-\widehat{u}^n_{\mathbf{z}'}(t,x)}_p^2
\le C |\mathbf{z}-\mathbf{z}'|^2.
\end{align*}
Putting this upper bound into the right-hand of \eqref{E_:I3} proves that \eqref{E:boundI23} holds for $I_3$.
Therefore, \eqref{E:DuInt} becomes
\begin{align*}
\Norm{\widehat{u}^n_{\mathbf{z}}(t,x)-\widehat{u}^n_{\mathbf{z}'}(t,x)}_p^2
\le& \quad C\Psi_n(t,x)\:|\mathbf{z}-\mathbf{z}'|^2\\
& + C
\int_{0}^t \ud s \int_\R \ud y\: G^2(t-s,x-y) \Norm{\widehat{u}^n_{\mathbf{z}}(s,y)-\widehat{u}^n_{\mathbf{z}'}(s,y)}_p^2,
\end{align*}
where $C$ does not depend on $n$, $\mathbf{z}$, $\mathbf{z}'$, $t$ and $x$.
Finally, an application of Lemma \ref{L:Mom} and \eqref{E:hG3} implies that
\begin{align}
\label{E_:IncHatU2}
\sup_{|\mathbf{z}|\vee |\mathbf{z}'|\le \kappa}
\Norm{\widehat{u}^n_{\mathbf{z}}(t,x)-\widehat{u}^n_{\mathbf{z}'}(t,x)}_p^2
\le C \left(2^{-n(1-1/\alpha)}+\Psi_n(t,x)\right)|\mathbf{z}-\mathbf{z}'|^2.
\end{align}
This proves \eqref{E:IncHatU}.

{\bigskip\noindent\bf Step 2.~}
Now we prove \eqref{E2:IncHatU}.
Notice that for $t>T-2^{-n}$,
\begin{align*}
& \Norm{\widehat{u}^{n,i}_{\mathbf{z}}(t,x)-\widehat{u}^{n,i}_{\mathbf{z}'}(t,x)}_p^2\\
\le & \hspace{1.28em} 6 \Norm{\theta^{n,i}_{\mathbf{z}}(t,x)-\theta^{n,i}_{\mathbf{z}'}(t,x)}_p^2 \\
&+ 6\Norm{\int_0^t \ud s\int_\R \ud y\: G(t-s,x-y) \left[\rho'(\widehat{u}^n_{\mathbf{z}}(s,y))-\rho'(\widehat{u}^n_{\mathbf{z}'}(s,y))\right] \widehat{u}_{\mathbf{z}}^{n,i}(s,y)\InPrd{\mathbf{z},\mathbf{h}_n(s,y)}}_p^2\\
&+ 6\Norm{\int_0^t \ud s\int_\R \ud y\: G(t-s,x-y) \rho'(\widehat{u}^n_{\mathbf{z}'}(s,y))\left[\widehat{u}_{\mathbf{z}}^{n,i}(s,y)-\widehat{u}_{\mathbf{z}'}^{n,i}(s,y)\right] \InPrd{\mathbf{z},\mathbf{h}_n(s,y)}}_p^2\\
&+ 6\Norm{\int_0^t \ud s\int_\R \ud y\: G(t-s,x-y) \rho'(\widehat{u}^n_{\mathbf{z}'}(s,y))\widehat{u}_{\mathbf{z}'}^{n,i}(s,y)\InPrd{\mathbf{z}-\mathbf{z}',\mathbf{h}_n(s,y)}}_p^2\\
&+ 24p\int_{T-2^{-n}}^t\int_\R G^2(t-s,x-y)\Norm{\left[\rho'(\widehat{u}^n_{\mathbf{z}}(s,y))-\rho'(\widehat{u}^n_{\mathbf{z}'}(s,y))\right] \widehat{u}_{\mathbf{z}}^{n,i}(s,y)}_p^2 \ud s\ud y\\
&+ 24p\int_{T-2^{-n}}^t\int_\R G^2(t-s,x-y)\Norm{\rho'(\widehat{u}^n_{\mathbf{z}'}(s,y))\left[\widehat{u}_{\mathbf{z}}^{n,i}(s,y)-\widehat{u}_{\mathbf{z}'}^{n,i}(s,y)\right]}_p^2 \ud s\ud y\\
=:& 6\left(I_1+I_2+I_3+I_4+4p I_5+ 4p I_6\right).
\end{align*}
We claim that there exists some constant $C>0$ independent of $n$, $t$ and $x$ such that
\begin{align}\label{E:uI124}
 \sup_{|\mathbf{z}|\vee |\mathbf{z}'|\le \kappa} I_i \le
 C \left[2^{-n(1-1/\alpha)}+\Psi_n(t,x)^2\right]|\mathbf{z}-\mathbf{z}'|^2,\quad i=1,2,4.
\end{align}
Because
\[
\Norm{\theta_{\mathbf{z}}^{n,i}(t,x)-\theta_{\mathbf{z}'}^{n,i}(t,x)}_p
\le
C \int_{0}^t\int_\R G(t-s,x-y)\Norm{\widehat{u}_{\mathbf{z}}^{n}(s,y)-\widehat{u}_{\mathbf{z}'}^{n}(s,y)}_p h_n^i(s,y)
\ud s \ud y,
\]
as a direct consequence of \eqref{E_:IncHatU2}, we see that
\begin{align}
\sup_{|\mathbf{z}|\vee |\mathbf{z}'|\le \kappa}
\Norm{\theta_{\mathbf{z}}^{n,i}(t,x)-\theta_{\mathbf{z}'}^{n,i}(t,x)}_p^2
&\le   C \Psi_n^i(t,x)^2 |\mathbf{z}-\mathbf{z}'|^2,
\label{E:ThetaZZ}
\end{align}
which proves \eqref{E:uI124} for $I_1$.
The case for $I_4$ is a direct consequence of \eqref{E:supxKui} and \eqref{E:PsiN}.
As for $I_2$, by \eqref{E:supxKui} and \eqref{E_:IncHatU2},
\begin{align*}
I_2&\le
\left(\int_0^t\ud s\int_\R \ud y \: G(t-s,x-y) \Norm{\left[\rho'(\widehat{u}^n_{\mathbf{z}}(s,y))-\rho'(\widehat{u}^n_{\mathbf{z}'}(s,y))\right] \widehat{u}_{\mathbf{z}}^{n,i}(s,y)}_p \InPrd{\mathbf{z},\mathbf{h}_n(s,y)}\right)^2\\
&\le
C \left(\int_0^t\ud s\int_\R \ud y \: G(t-s,x-y) \Norm{\widehat{u}^n_{\mathbf{z}}(s,y)-\widehat{u}^n_{\mathbf{z}'}(s,y)}_{2p}\Norm{\widehat{u}_{\mathbf{z}}^{n,i}(s,y)}_{2p} \InPrd{\mathbf{z},\mathbf{h}_n(s,y)}\right)^2\\
&\le
\textcolor{\myred}{C \Psi_n(t,x)^2} \:
|\mathbf{z}-\mathbf{z}'|^2,
\end{align*}
where the constants $C$ do not depend on $n$, $\mathbf{z}$, $\mathbf{z}'$, $t$ and $x$.
Therefore, \eqref{E:uI124} holds.

For $I_3$ we have that
\[
I_3 \le  C
\int_0^t\int_\R
G(t-s,x-y)\Norm{\widehat{u}_{\mathbf{z}}^{n,i}(s,y)-\widehat{u}_{\mathbf{z}'}^{n,i}(s,y)}_p^2
\InPrd{\mathbf{1},\mathbf{h}_n(s,y)}
\ud s\ud y=: C F_n(t,x).
\]
As for $I_6$, we see that
\[
 I_6\le C \int_{T-2^{-n}}^t\int_\R G^2(t-s,x-y)
\Norm{\widehat{u}^{n,i}_{\mathbf{z}}(s,y)-\widehat{u}^{n,i}_{\mathbf{z}'}(s,y)}_{p}^2
 \ud s\ud y.
\]
As for $I_5$, by \eqref{E:supxKui} and \eqref{E:IncHatU}, we have that
\begin{align*}
 I_5&\le C
 \int_{T-2^{-n}}^t\int_\R G^2(t-s,x-y)\Norm{\widehat{u}^n_{\mathbf{z}}(s,y)-\widehat{u}^n_{\mathbf{z}'}(s,y)}_{2p}^2
 \Norm{\widehat{u}_{\mathbf{z}}^{n,i}(s,y)}_{2p}^2 \ud s\ud y\\
 &\le C|\mathbf{z}-\mathbf{z}'|^2
 \int_{T-2^{-n}}^t\int_\R G^2(t-s,x-y)
 \ud s\ud y\\
 &\le C
 2^{-n(1-1/\alpha)}\:|\mathbf{z}-\mathbf{z}'|^2.
\end{align*}
Therefore, for some constant $C$ independent of $n$, $t$ and $x$, it holds that
\begin{align*}
 \Norm{\widehat{u}^{n,i}_{\mathbf{z}}(t,x)-\widehat{u}^{n,i}_{\mathbf{z}'}(t,x)}_p^2
 \le & \: C\left[2^{-n(1-1/\alpha)}+\Psi_n(t,x)^2\right] |\mathbf{z}-\mathbf{z}'|^2 + C F_n(t,x)
\\
& + C \int_0^t\ud s \int_\R \ud y\: G^2(t-s,x-y)\Norm{\widehat{u}^{n,i}_{\mathbf{z}}(s,y)-\widehat{u}^{n,i}_{\mathbf{z}'}(s,y)}_{p}^2.
\end{align*}
Then, by Lemma \ref{L:Mom} with $\vv=0$,
\[
 \Norm{\widehat{u}^{n,i}_{\mathbf{z}}(t,x)-\widehat{u}^{n,i}_{\mathbf{z}'}(t,x)}_p^2
 \le
 C\left[2^{-n(1-1/\alpha)} +\Psi_n(t,x)^2\right] |\mathbf{z}-\mathbf{z}'|^2 + C F_n(t,x) + C (F_n\star\calG)(t,x),\quad t\le T.
\]
Therefore, through a similar argument, the above inequality reduces to
\begin{multline*}
\Norm{\widehat{u}^{n,i}_{\mathbf{z}}(t,x)-\widehat{u}^{n,i}_{\mathbf{z}'}(t,x)}_p^2
\le
C \left[2^{-n(1-1/\alpha)} +\Psi_n(t,x)^2\right]|\mathbf{z}-\mathbf{z}'|^2 \\
+  C
\int_0^t\int_\R
G(t-s,x-y)\Norm{\widehat{u}_{\mathbf{z}}^{n,i}(s,y)-\widehat{u}_{\mathbf{z}'}^{n,i}(s,y)}_p^2
\InPrd{\mathbf{1},\mathbf{h}_n(s,y)}
\ud s\ud y.\quad
\end{multline*}
Then by denoting
\[
U_{\mathbf{z},\mathbf{z}',p}^{n,i}(t):=
\sup_{x\in\R}\Norm{\widehat{u}_{\mathbf{z}}^{n,i}(t,x)-\widehat{u}_{\mathbf{z}'}^{n,i}(t,x)}_p^2
\]
and applying the same arguments as those led  to \eqref{E:supxKui}, we see that
\[
\sup_{x\in\R}\Norm{\widehat{u}_{\mathbf{z}}^{n,i}(t,x)-\widehat{u}_{\mathbf{z}'}^{n,i}(t,x)}_p^2\le
C|\mathbf{z}-\mathbf{z}'|^2.
\]
Then putting this upper bound into the upper bound for $I_6$, we see that
\[
I_6\le
C2^{-n(1-1/\alpha)}|\mathbf{z}-\mathbf{z}'|^2.
\]
Combining these six upper bounds proves \eqref{E2:IncHatU}.

{\bigskip\noindent\bf Step 3.~}
 The proof of  \eqref{E3:IncHatU} is similar and
 we only give an outline of the proof and leave
the details for the  readers.
Write $\widehat{u}_{\mathbf{z}}^{n,i,k}(t,x)$
in six parts as in \eqref{E:Unik6}.
Using the results obtained in Steps 1 and 2,
one first proves that
\[
\sup_{|\mathbf{z}|\vee |\mathbf{z}'|\le \kappa}
\Norm{V_{\mathbf{z}}(t,x)-V_{\mathbf{z}'}(t,x)}_p^2
\le C \left[2^{-n(1-1/\alpha)} + \Psi_n(t,x)^2\right]|\mathbf{z}-\mathbf{z}'|^2,
\]
with $V_{\mathbf{z}}$ being any one of
$\theta^{n,i,k}_{\mathbf{z}}$,
$\theta^{n,k,i}_{\mathbf{z}}$,
$U_1^n$, or $U_2^n$.
The terms $U_3^n$ and $U_4^n$ are used to form the recursion.
Then in the same way as above,
one can prove \eqref{E3:IncHatU}.
This completes the proof of Lemma \ref{L:HolderInZ}.
\end{proof}

\subsubsection{Conditional boundedness}
\label{sss:bound}

The aim of this subsection  is to prove the following proposition.

\begin{proposition}\label{P:thetaBdd}
For all $\kappa>0$ and $r>0$,
there exists constant  $K>0$ such that
for all $1\le i,k\le d$,
\begin{align}
\label{E:thetabdd}
\lim_{n\rightarrow\infty}\bbP\Bigg(
\Big[&\:
\sup_{|\mathbf{z}|\le\kappa} \left|\widehat{u}_{\mathbf{z}}^{n}(T,x_i)\right|
\vee
\sup_{|\mathbf{z}|\le\kappa} \left|\widehat{u}_{\mathbf{z}}^{n,i}(T,x_i)\right|
\vee
\sup_{|\mathbf{z}|\le\kappa} \left|\widehat{u}_{\mathbf{z}}^{n,i,k}(T,x_i)\right|
\Big]
\le K\:\: \Big| \:\: Q_r\Bigg) =1,
\end{align}
where
\[
 Q_r:=
\Big\{|\{u(T,x_i)-y_i\}_{1\le i\le d}|\le r\Big\}.
\]
\end{proposition}
\begin{remark}It  is natural and  interesting to express the following limits:
\begin{align*}
\lim_{n\rightarrow\infty}
\sup_{|\mathbf{z}|\le\kappa}|\widehat{u}_{\mathbf{z}}^{n}(T,x_i)|,
\quad
\lim_{n\rightarrow\infty}
\sup_{|\mathbf{z}|\le\kappa}|\widehat{u}_{\mathbf{z}}^{n,i}(T,x_i)|
\quad\text{and}\quad
\lim_{n\rightarrow\infty}
\sup_{|\mathbf{z}|\le\kappa}|\widehat{u}_{\mathbf{z}}^{n,i,k}(T,x_i)|
\end{align*}
 as functionals of $u(T,x_i)$.  However,
finding these exact limits seems to be  a very hard problem.
Here fortunately for our current purpose,
we only need the above  weaker result -- the conditional boundedness in \eqref{E:thetabdd}.
\end{remark}

We need some lemmas.
\begin{lemma}\label{L:Unz}
For any $\kappa>0$, $p\ge 2$, $n\in\bbN$, $t\in [0,T]$ and $x\in\R$, we have
\begin{align}\label{E:UnU-Close}
\Norm{
\sup_{|\mathbf{z}|\le\kappa} \left|
\widehat{u}_{\mathbf{z}}^{n}(t,x)-u(t,x)
\right|}_p^2 \le  C\left[2^{-n(1-1/\alpha)} + \Psi_n(t,x)\right].
\end{align}
\end{lemma}
\begin{proof}
 Notice that
\begin{align*}
\widehat{u}^n_{\mathbf{z}}(t,x)-u(t,x) =&
\int_0^t\int_\R G(t-s,x-y)
\rho(\widehat{u}^n_{\mathbf{z}}(t,x))\InPrd{\mathbf{z},\mathbf{h}_n(s,y)}\ud s\ud y\\
&+
\int_0^t\int_\R G(t-s,x-y)
\left[\rho(\widehat{u}^n_{\mathbf{z}}(t,x))-\rho(u(t,x))\right]W(\ud s,\ud y).
\end{align*}
Then applying the arguments that led  to \eqref{E:DriftBd} and by \eqref{E:SupU} and \eqref{E:PsiN}, we have that
\begin{multline*}
\sup_{|\mathbf{z}|\le\kappa} \Norm{\widehat{u}^n_{\mathbf{z}}(t,x) -u(t,x)}_p^2
\\ \le
 \kappa^2 C \Psi_n(t,x)^2
 +  C\int_0^t\ud s\int_\R  \ud y \:  G^2(t-s,x-y)
\sup_{|\mathbf{z}|\le\kappa}
\Norm{\widehat{u}^n_{\mathbf{z}}(s,y) -u(s,y)}_p^2,\qquad
\end{multline*}
for all $t\in [0,T]$.
Then by Lemma \ref{L:Mom} with $\vv=0$, we see that
\[
\sup_{|\mathbf{z}|\le\kappa} \Norm{\widehat{u}^n_{\mathbf{z}}(t,x)-u(t,x)}_p^2
\le
\kappa^2 C \Psi_n(t,x)^2 +
C \int_0^t \int_\R (t-s)^{-1/\alpha}G(t-s,x-y)
\Psi_n(s,y)\ud s \ud y.
\]
Hence, an application of \eqref{E:hG3} shows that
\begin{align}\label{E:SupOutIn}
\sup_{|\mathbf{z}|\le\kappa} \Norm{\widehat{u}^n_{\mathbf{z}}(t,x)-u(t,x)}_p^2
\le  C\left[2^{-n(1-1/\alpha)} + \Psi_n(t,x)^2\right]
\le
C\left[2^{-n(1-1/\alpha)} + \Psi_n(t,x)\right].
\end{align}
Thanks to \eqref{E:IncHatU} and the fact that $\Norm{\widehat{u}^n_0(t,x)-u(t,x)}_p\equiv 0$, one can apply the Kolmogorov continuity theorem (see Theorem \ref{T:KolCont}) to move the supremum inside the $L^p(\Omega)$-norm in \eqref{E:SupOutIn}. This proves Lemma \ref{L:Unz}.
\end{proof}

\bigskip

\begin{lemma}
\label{L:thetaRho}
For any $\kappa>0$, $1\le i\le d$ and $n\in\bbN$,
it holds that
\begin{align}
\label{E:thetaRho}
\Norm{
\sup_{|\mathbf{z}|\le \kappa}\left|
\theta_{\mathbf{z}}^{n,i}(T,x_i) - \rho(u(T,x_i))
\right|
}_p^2\le C \left[2^{-n(1-1/\alpha)}+\kappa^2\right].
\end{align}
As a consequence, for all $x\in\R$, when $\kappa=0$,
\begin{align}
 \label{E2:thetaRho}
\lim_{n\rightarrow\infty}
\theta^{n,i}_{0}(T,x)
=
\rho(u(T,x_i)) \one_{\{x=x_i\}}
\quad \text{a.s.}
\end{align}
\end{lemma}
\begin{proof}
Notice that by \eqref{E:cni-1},
\[
\theta_{\mathbf{z}}^{n,i}(T,x_i) = \rho(u(T,x_i)) + \int_0^T \int_\R G(T-s,x_i-y)
\left[\rho(\widehat{u}^n_{\mathbf{z}}(s,y))-\rho(u(T,x_i))\right] h_n^i(s,y)\ud s\ud y.
\]
By the same arguments as those led  to \eqref{E:DriftBd}, we have that
\begin{align*}
&\Norm{\sup_{|\mathbf{z}|\le \kappa}
\left|\theta_{\mathbf{z}}^{n,i}(T,x_i) - \rho(u(T,x_i))\right|
}_p^2\\
&\le C
\int_0^T\int_\R G(T-s,x_i-y)h_n^i(s,y)\:
\Norm{
\sup_{|\mathbf{z}|\le \kappa}\left|
\widehat{u}^n_{\mathbf{z}}(s,y)-u(T,x_i)
\right|
}_p^2 \ud s\ud y\\
&=
C
c_n \int_0^{2^{-n}}\int_{-2^{-n}}^{2^{-n}} G(s,y)
\Norm{
\sup_{|\mathbf{z}|\le \kappa}\left|
\widehat{u}^n_{\mathbf{z}}(T-s,x_i-y)-u(T,x_i)
\right|
}_p^2 \ud s\ud y\\
&\le
C
c_n \int_0^{2^{-n}}\int_{-2^{-n}}^{2^{-n}} G(s,y)
\Norm{
\sup_{|\mathbf{z}|\le \kappa}\left|
\widehat{u}^n_{\mathbf{z}}(T-s,x_i-y)-u(T-s,x_i-y)
\right|
}_p^2 \ud s\ud y\\
&\quad +
C
c_n \int_0^{2^{-n}}\int_{-2^{-n}}^{2^{-n}} G(s,y)
\Norm{
u(T-s,x_i-y)-u(T,x_i)
}_p^2 \ud s\ud y\\
&=: I_1^n + I_2^n,
\end{align*}
where $c_n$ is defined in \eqref{E:cni-1}.
By \eqref{E:UnU-Close}, we have that
\[
I_1^n \le
C
c_n \int_0^{2^{-n}}\int_{-2^{-n}}^{2^{-n}} G(s,y)
\left[2^{-n(1-1/\alpha)}+\kappa^2\right] \ud s\ud y
\le C \left[2^{-n(1-1/\alpha)}+\kappa^2\right].
\]
As for $I_2^n$, by \eqref{E:ScaleG} and \eqref{E:BddG} and
by the H\"older regularity of $u(t,x)$ (see \eqref{E:Holder}), we see that
\begin{align}\notag
I_2^n
\le&
Cc_n \int_0^{2^{-n}}\int_{-2^{-n}}^{2^{-n}} s^{-1/\alpha}
\Norm{
u(T-s,x_i-y)-u(T,x_i)
}_p^2 \ud s\ud y\\
\notag
\le&
C c_n \int_0^{2^{-n}} \ud s \int_{-2^{-n}}^{2^{-n}}\ud y \:
\left[s^{2(1-1/\alpha)-1/\alpha}+ \: s^{-1/\alpha} |y|^{2(\alpha-1)}\right]\\
\le&
C 2^{-2n(1-1/\alpha)} + C2^{-2n(\alpha-1)}.
\label{E_:HolderT}
\end{align}
Therefore, \eqref{E:thetaRho} is proved by
combining these two cases and by using the fact that $1-1/\alpha<\alpha-1$.
Finally, \eqref{E2:thetaRho} is proved by an application of the Borel-Cantelli lemma
thanks to \eqref{E:thetaRho} when $x=x_i$ and \eqref{E:theta-Lp0} when $x\ne x_i$.
This completes the proof of Lemma \ref{L:thetaRho}.
\end{proof}

\bigskip
\begin{lemma}
\label{L:UniU-Close}
For any $\kappa>0$, $1\le i\le d$ and $n\in\bbN$,
it holds that
\begin{align}
\label{E:UniU-Close}
\Norm{
\sup_{|\mathbf{z}|\le \kappa}\left|
\widehat{u}_{\mathbf{z}}^{n,i}(T,x_i) - \rho(u(T,x_i))
\right|
}_p^2\le  C\left[2^{-n(1-1/\alpha)} + \kappa^2\right].
\end{align}
As a consequence, for all $x\in\R$, when $\kappa=0$,
\begin{align}
 \label{E2:UniU-Close}
\lim_{n\rightarrow\infty}
\widehat{u}^{n,i}_{0}(T,x)
&=
\rho(u(T,x_i))\one_{\{x=x_i\}}
\quad \text{a.s.}
\end{align}
\end{lemma}
\begin{proof}
The proof for \eqref{E:UniU-Close} is similar to that for \eqref{E:UnU-Close}, where we need \eqref{E:thetaRho}.
From \eqref{E:hatUni}, we see that
\begin{align*}
 \widehat{u}^{n,i}_{\mathbf{z}}(T,x_i)-\rho(u(T,x_i))
 =& \quad \theta^{n,i}_{\mathbf{z}}(T,x_i)-\rho(u(T,x_i))\\
  &+\int_{T-2^{-n}}^T\int_\R   G(T-s,x_i-y)
 \rho'(\widehat{u}^{n}_{\mathbf{z}}(s,y))\widehat{u}^{n,i}_{\mathbf{z}}(s,y) W(\ud s,\ud y)\\
 &+
 \sum_{j=1, j\ne i}^n \int_{0}^T\ud s\int_\R  \ud y \:  G(T-s,x_i-y)z_j
  \rho'(\widehat{u}^{n}_{\mathbf{z}}(s,y))\widehat{u}^{n,i}_{\mathbf{z}}(s,y)h_n^j(s,y)\\
 &+
 \int_{0}^T\ud s\int_\R  \ud y \:  G(T-s,x_i-y)
 \rho'(\widehat{u}^{n}_{\mathbf{z}}(s,y))\widehat{u}^{n,i}_{\mathbf{z}}(s,y)
 z_i h_n^i(s,y)\\
 =:& \: U_0^n+U_1^n + U_2^n + U_3^n.
\end{align*}
By \eqref{E:thetaRho},
\[
\Norm{\sup_{|\mathbf{z}|\le \kappa} \left| U_0^n \right| }_p^2\le C\left[2^{-n(1-1/\alpha)}+\kappa^2\right].
\]
By \eqref{E:BGD}, \eqref{E:supxKui} and \eqref{E:G10},
\[
\sup_{|\mathbf{z}|\le \kappa} \Norm{ U_1^n}_p^2\le C \int_0^{2^{-n}} \int_\R G(s,y)^2 \ud s\ud y
\le C 2^{-n(1-1/\alpha)}.
\]
Write $U_1^n$ as $U^n_{1,\mathbf{z}}$.
The difference
$\Norm{U_{1,\mathbf{z}}-U_{1,\mathbf{z}'}}_p^2$ has been
estimated in Step 2 in the proof of Lemma \ref{L:HolderInZ}, which is equal to $I_5+I_6$ (with the notations there). Hence,
\[
\sup_{|\mathbf{z}|\vee |\mathbf{z}'|\le \kappa}\Norm{U_{1,\mathbf{z}}-U_{1,\mathbf{z}'}}_p^2\le C 2^{-n(1-1/\alpha)} |\mathbf{z}-\mathbf{z}'|^2.
\]
Then one can apply the Kolmogorov continuity theorem (see Theorem \ref{T:KolCont}) to see that
\[
\Norm{\sup_{|\mathbf{z}|\le \kappa} \left| U_{1,\mathbf{z}}^n\right| }_p^2\le C 2^{-n(1-1/\alpha)}.
\]
By the same arguments led  to \eqref{E:DriftBd} and by \eqref{E2:SupU} and \eqref{E:PsiN},
using the boundedness of $\rho'$, we see that
\[
\Norm{\sup_{|\mathbf{z}|\le \kappa}\left|U_2^n\right|}_p^2\le \kappa^2 C \sum_{j\ne i}\Psi_n^i(T,x_j)\le \kappa^2 C.
\]
As for $U_3^n$, by \eqref{E2:SupU} and by the boundedness of $\rho'$, we have that
\begin{align*}
\Norm{\sup_{|\mathbf{z}|\le \kappa} \left| U_3^n \right| }_p^2
& \le \kappa^2 C
\int_0^T\ud s\int_\R \ud y\:
G(T-s,x_i-y)\left[2^{-n(1-1/\alpha)} + 1\right] h_n^i(s,y)\\
&\le
\kappa^2 C\left[2^{-n(1-1/\alpha)} + \Psi_n^i(T,x_i)\right]\\
&=
\kappa^2 C\left[2^{-n(1-1/\alpha)} + 1\right],
\end{align*}
where in the last step we have used the fact \eqref{E:hG1}.
Therefore,
\eqref{E:UniU-Close} is proved by combining these three terms.
Finally, \eqref{E2:UniU-Close} is application of the Borel-Cantelli lemma
thanks to \eqref{E:UniU-Close} when $x=x_i$ and \eqref{E:UniLp0} when $x\ne x_i$.
This completes the proof of Lemma \ref{L:UniU-Close}.
\end{proof}

\bigskip
\begin{lemma}\label{L:HolderuHatx}
For all $n\in\bbN$, $t\in [T-2^{-n},T]$, $p\ge 2$, $1\le i,k\le d$, $K>0$, and $x,y\in(-K,K)$,
we have that
\begin{align} \label{E1:Space}
\Norm{\sup_{|\mathbf{z}|\le\kappa}\left| \widehat{u}^n_{\mathbf{z}}(t,x)-\widehat{u}^n_{\mathbf{z}}(t,y)\right| }_p^2
&\le C \left(|x-y|^2+|x-y|^{\alpha-1}+2^{n/\alpha}|x-y|\right),\\
\Norm{\sup_{|\mathbf{z}|\le\kappa}\left|
\widehat{u}^{n,i}_{\mathbf{z}}(t,x)-\widehat{u}^{n,i}_{\mathbf{z}}(t,y)
\right|
}_p^2
&\le C \left(|x-y|^{\alpha-1}+2^{n/\alpha}|x-y|\right),
\label{E2:Space}\\
\label{E3:Space}
\Norm{
\sup_{|\mathbf{z}|\le\kappa}\left|
\widehat{u}^{n,i,k}_{\mathbf{z}}(t,x)-\widehat{u}^{n,i,k}_{\mathbf{z}}(t,y)
\right|}_p^2
&\le C \left(|x-y|^{\alpha-1}+2^{n/\alpha}|x-y|\right).
\end{align}
\end{lemma}
\begin{proof}
We first prove \eqref{E1:Space}. Notice that
\begin{align*}
\widehat{u}^n_{\mathbf{z}}(t,x)-\widehat{u}^n_{\mathbf{z}}(t,y)
=&J_0(t,x)-J_0(t,y)\\
&+\int_0^t\int_\R \left[G(t-s,x-z)-G(t-s,y-z)\right]\rho\left(\widehat{u}^n_{\mathbf{z}}(s,z)\right)W(\ud s,\ud z)\\
&+
\sum_{i=1}^d z_i \int_0^t\int_\R \left[G(t-s,x-z)-G(t-s,y-z)\right]\rho\left(\widehat{u}^n_{\mathbf{z}}(s,z)\right)h_n^i(s,z)\ud s\ud z\\
=:&I_1 + I_2 +I_3.
\end{align*}
The continuity of $x\mapsto J_0(t,x)$ (see Lemma 4.10 in \cite{ChenDalang15FracHeat}) implies that $|I_1|\le C |x-y|$.
By \eqref{E:SupU} and by the same arguments as those in the proof of Theorem 1.6 in \cite{ChenKim14Comparison},
we see that
\[
\sup_{|\mathbf{z}|\le \kappa} \Norm{I_2}_p^2 \le C |x-y|^{\alpha-1}.
\]
Write $I_2$ as $I_{2,\mathbf{z}}$. Then by \eqref{E:IncHatU} and \eqref{E:G-x},
\begin{align*}
 \sup_{|\mathbf{z}|\vee |\mathbf{z}'|\le\kappa}
 \Norm{I_{2,\mathbf{z}} - I_{2,\mathbf{z}'}}_p^2
 &\le
 C \int_0^t \int_\R \left[G(t-s,x-w)-G(t-s,y-w)\right]^2\\
 &\qquad\qquad \times
 \sup_{|\mathbf{z}|\vee |\mathbf{z}'|\le\kappa}
 \Norm{\hat{u}^n_{\mathbf{z}}(s,w)-\hat{u}^n_{\mathbf{z}'}(s,w)}_p^2\ud s\ud w\\
 &\le
 C |\mathbf{z}-\mathbf{z}'|^2\int_0^t \int_\R \left[G(t-s,x-w)-G(t-s,y-w)\right]^2\ud s\ud w\\
 &\le
 C |x-y|^{\alpha-1}|\mathbf{z}-\mathbf{z}'|^2.
\end{align*}
Hence, an application of the Kolmogorov continuity theorem (see Theorem \ref{T:KolCont}) implies that
\[
\Norm{\sup_{|\mathbf{z}|\le \kappa} \left| I_{2,\mathbf{z}}\right| }_p^2 \le C |x-y|^{\alpha-1}.
\]
As for $I_3$, by \eqref{E:SupU}, we see that
\begin{align*}
\Norm{\sup_{|\mathbf{z}|\le \kappa}\left| I_3\right| }_p^2 \le &C \one_{\{t>T-2^{-n}\}}2^{n(2-1/\alpha)}\int_{T-2^{-n}}^t\int_{x_i-2^{-n}}^{x_i+2^{-n}}
\left|G(t-s,x-z)-G(t-s,y-z)\right| \ud s\ud z\\
\le &C 2^{n(2-1/\alpha)}\int_0^{2^{-n}} \int_{x_i-2^{-n}}^{x_i+2^{-n}}
\left|G(s,x-z)-G(s,y-z)\right| \ud s\ud z\\
\le &C 2^{n/\alpha} |x-y|,
\end{align*}
where in the last step we have applied Lemma \ref{L:IntG-G}.

Now we prove \eqref{E2:Space}.
From \eqref{E:hatUni}, write the difference $\hat{u}^{n,i}_{\mathbf{z}}(t,x)-
\hat{u}^{n,i}_{\mathbf{z}}(t,y)$ in three parts as above.
By moment bound \eqref{E:SupU}, we see that
the difference for $\theta^{n,i}_{\mathbf{z}}$ reduces to
\begin{align*}
 \Norm{\sup_{|\mathbf{z}|\le\kappa}\left|
 \theta^{n,i}_{\mathbf{z}}(t,x)-\theta^{n,i}_{\mathbf{z}}(t,y)
 \right|
 }_p^2
 \le&
 C \left(\int_0^t\int_\R \left|G(t-s,x-z)-G(t-s,y-z)\right|h_n^i(s,z)\ud s\ud z\right)^2.
\end{align*}
Then by Lemma \ref{L:IntG-G}, this part is bounded by $C 2^{n/\alpha} |x-y|$.
Thanks to the boundedness of $\rho'$, by the same arguments (using moment bound \eqref{E2:SupU}), the third part of  $\widehat{u}^{n,i}_{\mathbf{z}}$ in \eqref{E:hatUni}
has the same moment bound as $I_3$ above.
The second part can be proved in the same way as those for $I_2$.
This proves \eqref{E2:Space}.

The result \eqref{E3:Space} can be proved in the same way. We leave the details for interested readers.
This completes the proof of Lemma \ref{L:HolderuHatx}.
\end{proof}

\bigskip
Before proving Proposition \ref{P:thetaBdd}, we introduce the following functionals to alleviate the notation:
for any function $g:\R_+\mapsto\R$, $n\in\bbN$ and $t\in(0,T]$, define
\begin{align}\label{E:F}
F_n\left[g\right](t):=
g(t) + \textcolor{\myred}{\left(1-1/\alpha\right)}\one_{\{t>T-2^{-n}\}}2^{n(1-1/\alpha)}
\int_{T-2^{-n}}^t(t-s)^{-1/\alpha}
g(s)\ud s,
\end{align}
and for $m\ge 1$,
\begin{align}\label{E:Fm}
F_{n}^m\left[g\right](t):=
m\textcolor{\myred}{\left(1-1/\alpha\right)}\one_{\{t>T-2^{-n}\}}2^{mn(1-1/\alpha)}
\int_{T-2^{-n}}^t(t-s)^{m(1-1/\alpha)-1}
g(s)\ud s.
\end{align}
\begin{lemma}\label{L:F}
For all $n,m,m'\ge 1$ and $T>0$, the following properties hold for all $t\in (0,T]$:
\begin{gather}
\label{E:F1}
F_n[g](t)=g(t)+F_n^1[g](t),\\
\label{E:F2}
\lim_{n\rightarrow\infty}F_n^m[g](t)\one_{\{t<T\}} = 0,\\
\label{E:F5}
F_n^m[|g|](t)\le \left(F_n^m[|g|^{m'}](t)\right)^{1/m'},\\
F_n^m\left[F_n^{m'}[g]\right](t) = C_{m,m',\alpha}\: F_n^{m+m'}[g](t)\,.
\label{E:F3}
\end{gather}
If $g$ is left-continuous at $T$, then
\begin{align}\label{E:F6}
\lim_{n\rightarrow\infty}F_n^m[g](T) = g(T).
\end{align}
\end{lemma}
\begin{proof}
Properties \eqref{E:F1} and \eqref{E:F2} are clear from the definition.
Note that
\[
m\left(1-1/\alpha\right)2^{mn(1-1/\alpha)}
\int_{T-2^{-n}}^T(T-s)^{m(1-1/\alpha)-1}
\ud s=1.
\]
Hence, the H\"older inequality implies \eqref{E:F5}.
As for \eqref{E:F3}, we need only to consider the case when $t>T-2^{-n}$.
In this case, by the Beta integral,
\begin{align*}
 F_n^m\left[F_n^{m'}[g]\right](t) =&
 mm'(1-1/\alpha)^22^{(m+m')n(1-1/\alpha)}
 \int_{T-2^{-n}}^t\ud s\int_{T-2^{-n}}^s \ud r\\
 &\times
 (t-s)^{m(1-1/\alpha)-1}
 (s-r)^{m'(1-1/\alpha)-1} g(r)\\
 =&
 mm'(1-1/\alpha)^2
 2^{(m+m')n(1-1/\alpha)}
 \int_{T-2^{-n}}^t\ud r\: g(r) \int_{r}^t \ud s\\
 &\times
 (t-s)^{m(1-1/\alpha)-1}
 (s-r)^{m'(1-1/\alpha)-1} \\
 =& C_{m,m',\alpha}'(1-1/\alpha)
 2^{(m+m')n(1-1/\alpha)}\int_{T-2^{-n}}^t (t-r)^{(m+m')n(1-1/\alpha)-1} g(r)\ud r\\
 =&C_{m,m',\alpha}\: F_n^{m+m'}[g](t).
\end{align*}
Finally, if $g$ is left-continuous at $T$, then for any $\epsilon>0$, there exists $N\ge 1$ such that
when $T-s\le 2^{-N}$, then $|g(T)-g(s)|\le \epsilon$. Hence, for all $n\ge N$,
\[
\left|F_n^m[g](T)-g(T)\right|
=\left|F_n^m[g](T)-F_n^m[g(T)](T)\right|
\le F_n^m[|g(T)-g|](T) \le F_n^m[\epsilon ](T)=\epsilon.
\]
Since $\epsilon$ is arbitrary, we have proved \eqref{E:F6}.
\end{proof}

\bigskip
Now we are ready to prove our last proposition.
\begin{proof}[Proof of Proposition \ref{P:thetaBdd}]
In this  proof  we assume   $t\in (0,T]$ and we  use $\{M_n\}_{n\ge 1}$ to denote a generic sequence of nonnegative random variable that
goes to zero as $n\rightarrow\infty$.  Its value may vary at each occurrence.
We divide the proof to  four steps.

{\bigskip\noindent\bf Step 1.~}
We first prove \eqref{E:thetabdd} for $\widehat{u}^n_{\mathbf{z}}(T,x_i)$.
Notice that
\begin{align*}
\widehat{u}^n_{\mathbf{z}}(t,x_i)-u(t,x_i)
=&\int_0^t\int_\R G(t-s,x_i-y)\left[\rho\left(\widehat{u}^n_{\mathbf{z}}(s,y)\right)-
\rho\left(u(s,y)\right)\right]W(\ud s,\ud y)\\
&
+\sum_{j\ne i} z_j
\int_0^t\int_\R G(t-s,x_i-y)\rho\left(\widehat{u}^n_{\mathbf{z}}(s,y)\right)h_n^j(s,y)\ud s\ud y\\
&+ z_i
\int_0^t\int_\R G(t-s,x_i-y)
\left[
\rho\left(\widehat{u}^n_{\mathbf{z}}(s,y)\right)
-
\rho\left(\widehat{u}^n_{\mathbf{z}}(s,x_i)\right)
\right]
h_n^i(s,y)\ud s\ud y\\
&+ z_i
\int_0^t\int_\R G(t-s,x_i-y)
\left[
\rho\left(u(s,x_i)\right)
-
\rho\left(u(t,x_i)\right)
\right]
h_n^i(s,y)\ud s\ud y\\
&+z_i\Psi^i_n(t,x_i)\rho\left(u(t,x_i)\right) \\
&+ z_i
\int_0^t\int_\R G(t-s,x_i-y)
\left[
\rho\left(\widehat{u}^n_{\mathbf{z}}(s,x_i)\right)
-
\rho\left(u(s,x_i)\right)
\right]
h_n^i(s,y)\ud s\ud y\\
=:& \sum_{\ell=1}^6 I_\ell(t).
\end{align*}
Then we have that
\begin{align*}
\sup_{|\mathbf{z}|\le\kappa}
\left|
\widehat{u}^n_{\mathbf{z}}(t,x_i)-u(t,x_i)
\right|
\le &\sum_{\ell=1}^5 \sup_{|\mathbf{z}|\le\kappa} |I_\ell(t)|
+
C \one_{\{t>T-2^{-n}\}}2^{n(1-1/\alpha)}\\
&\times
\int_{T-2^{-n}}^t(t-s)^{-1/\alpha}
\sup_{|\mathbf{z}|\le\kappa}\left|
\widehat{u}^n_{\mathbf{z}}(s,x_i)
-
u(s,x_i)
\right|\ud s.
\end{align*}
By Lemma \ref{L:Contraction} (see \eqref{E3:Contraction}) and by \eqref{E:F},
\begin{align}\label{e1:uHat}
 \sup_{|\mathbf{z}|\le\kappa}\left|
\widehat{u}^n_{\mathbf{z}}(t,x_i)-u(t,x_i)
\right|
\le&
C \sum_{\ell=1}^5 F_n\left[
\sup_{|\mathbf{z}|\le\kappa}|I_\ell(\cdot)|
\right](t).
\end{align}
We estimate each term in the above sum separately. By \eqref{E:UnU-Close} and \eqref{E:hG3}, we see that
\[
\Norm{\sup_{|\mathbf{z}|\le\kappa}\left| I_1(t)\right|}_p^2\le  C 2^{-n(1-1/\alpha)}.
\]
By \eqref{E:SupU} and \eqref{E:hG2}, and since $i\ne j$, we see that
\[
\Norm{\sup_{|\mathbf{z}|\le\kappa}
\left| I_2(t)\right| }_p^2\le  C 2^{-n(1+1/\alpha)}.
\]
As for $I_3(t)$, we have that
\begin{align*}
\Norm{\sup_{|\mathbf{z}|\le\kappa}
\left|I_3(t)\right|}_p^2\le
& C2^{n(2-1/\alpha)}\one_{\{t>T-2^{-n}\}}\\
&\times \int_{T-2^{-n}}^t\int_{x_i-2^{-n}}^{x_i+2^{-n}} G(t-s,x_i-y)
\Norm{
\sup_{|\mathbf{z}|\le\kappa}
\left|
\widehat{u}^n_{\mathbf{z}}(s,y)-\widehat{u}^n_{\mathbf{z}}(s,x_i)
\right|}_p^2\ud s\ud y.
\end{align*}
Then since $|y-x_i|\le 2^{-n}$, by Lemma \ref{L:HolderuHatx},
\[
\Norm{\sup_{|\mathbf{z}|\le\kappa}
\left|\widehat{u}^n_{\mathbf{z}}(s,y)-\widehat{u}^n_{\mathbf{z}}(s,x_i)
\right|}_p^2\le
C 2^{-n(1-1/\alpha)},
\]
which implies that
\[
\Norm{\sup_{|\mathbf{z}|\le\kappa}\left|I_3(t)\right|}_p^2\le C 2^{-n(1-1/\alpha)}.
\]
Note that both $I_4(t)$ and $I_5(t)$ depend on $\mathbf{z}$ through a multiplicative factor $z_i$. By the H\"older continuity of $u(t,x)$, we see that
\begin{align*}
\Norm{
\sup_{|\mathbf{z}|\le\kappa}\left|
I_4(t)\right|}_p^2\le
& C2^{n(1-1/\alpha)}\int_0^t (t-s)^{-1/\alpha} \Norm{u(s,x_i)-u(t,x_i)}_p^2 \ud s\\
\le&
C2^{n(1-1/\alpha)}\int_0^t (t-s)^{-1/\alpha} (t-s)^{1-1/\alpha} \ud s\\
=& C 2^{-n(1-1/\alpha)}.
\end{align*}
Therefore,
\[
\Norm{\sum_{\ell=1}^4
F_n\left[\sup_{|\mathbf{z}|\le\kappa}|
I_\ell(\cdot)|
\right](t)}_p^2 \le C 2^{-n(1-1/\alpha)}.
\]
By the Borel-Cantelli lemma,
we see that,
\begin{align}\label{e2:uHat}
\lim_{n\rightarrow\infty}
\sum_{\ell=1}^4
F_n\left[
\sup_{|\mathbf{z}|\le\kappa}|I_\ell(\cdot)|
\right](t) = 0,\quad\text{a.s.}
\end{align}
The term $I_5$ does not depend on $\mathbf{z}$.
By the property of $\Psi_n(t,x_i)$ (see \eqref{E:hG0} and \eqref{E:hG1}),
\begin{align}
\label{e3-0:uHat}
|\Psi_n^i(t,x_i)\rho(u(t,x_i))|\le
C|\rho(u(t,x_i))|,\quad\text{a.s.}
\end{align}
Hence, for all $n\in\bbN$,
\begin{align}
F_n\left[|I_5(\cdot)|\right](t)
&\le C F_n\big[\left|\rho(u(\cdot,x_i))\right|\big](t).
\label{e3:uHat}
\end{align}
Note that the right-hand side  of \eqref{e3:uHat} does not depend on $n$.
Therefore, by combining \eqref{e1:uHat}, \eqref{e2:uHat} and \eqref{e3:uHat}
we have that
\begin{align}\label{e4:uHat}
\sup_{|\mathbf{z}|\le\kappa}
 \left|
 \widehat{u}_{\mathbf{z}}^n(t,x_i)
 -
 u(t,x_i)
 \right|\le &\:
 M_n+C\left|\rho(u(t,x_i))\right|+C F_n^1\big[\left|\rho(u(\cdot,x_i))\right|\big](t)
\quad
 \text{a.s. }
\end{align}
for all $n\in\bbN$.
Recall that $M_n$ is some generic nonnegative random quantity
that goes to zero a.s. as $n\rightarrow\infty$.
Finally, by letting $t=T$ and
sending $n\rightarrow\infty$, and by
the H\"older continuity of $s\mapsto \rho(u(s,x_i))$, we have that
\begin{align*}
\limsup_{n\rightarrow\infty}\sup_{|\mathbf{z}|\le\kappa}
 \left|
 \widehat{u}_{\mathbf{z}}^n(T,x_i)
\right|
 &\le
 \limsup_{n\rightarrow\infty}\sup_{|\mathbf{z}|\le\kappa}
 \left|
 \widehat{u}_{\mathbf{z}}^n(T,x_i)
 -
 u(T,x_i)
 \right| + |u(T,x_i)|\\
 & \le |u(T,x_i)|+ C |\rho(u(T,x_i))|, \quad\text{a.s.,}
\end{align*}
which proves Proposition \ref{P:thetaBdd} for $\widehat{u}^n_{\mathbf{z}}(T,x)$.

{\bigskip\noindent\bf Step 2.~}
Now we prove Proposition \ref{P:thetaBdd} for $\widehat{u}^{n,i}_{\mathbf{z}}(T,x)$.
Notice that
\begin{align*}
\widehat{u}^{n,i}_{\mathbf{z}}&(t,x_i) - \rho(u(t,x_i))\\
=& \:\: \theta^{n,i}_{\mathbf{z}}(t,x_i)\\
&+\one_{\{t>T-2^{-n}\}}\int_{T-2^{-n}}^t\int_\R G(t-s,x_i-y)
\rho'\left(\widehat{u}^{n}_{\mathbf{z}}(s,y)\right)
\widehat{u}^{n,i}_{\mathbf{z}}(s,y)W(\ud s,\ud y)\\
&+\sum_{j\ne i}z_j\int_0^t\int_\R G(t-s,x_i-y)
\rho'\left(\widehat{u}^{n}_{\mathbf{z}}(s,y)\right)
\widehat{u}^{n,i}_{\mathbf{z}}(s,y)
h_n^j(s,y)\ud s\ud y\\
&+z_i\int_0^t\int_\R G(t-s,x_i-y)
\rho'\left(\widehat{u}^{n}_{\mathbf{z}}(s,y)\right)
\left[\widehat{u}^{n,i}_{\mathbf{z}}(s,y)-\widehat{u}^{n,i}_{\mathbf{z}}(s,x_i)\right]
h_n^i(s,y)\ud s\ud y\\
&+z_i\int_0^t\int_\R G(t-s,x_i-y)
\rho'\left(\widehat{u}^{n}_{\mathbf{z}}(s,y)\right)
\left[\rho\left(u(s,x_i)\right)-\rho\left(u(t,x_i)\right)\right]
h_n^i(s,y)\ud s\ud y\\
&+z_i\rho\left(u(t,x_i)\right)\Psi_n^i(t,x_i)\\
&+z_i\int_0^t\int_\R G(t-s,x_i-y)
\rho'\left(\widehat{u}^{n}_{\mathbf{z}}(s,y)\right)
\left[\widehat{u}^{n,i}_{\mathbf{z}}(s,x_i)-\rho\left(u(s,x_i)\right)\right]
h_n^i(s,y)\ud s\ud y\\
=:&\sum_{\ell=0}^6 I_\ell(t).
\end{align*}
Then we have
\begin{align}
\label{e0:uHatD}
\begin{aligned}
\sup_{|\mathbf{z}|\le\kappa}\left|
\widehat{u}^{n,i}_{\mathbf{z}}(t,x_i)-\rho(u(t,x_i))
\right|
\le&  \sum_{\ell=0}^5 \sup_{|\mathbf{z}|\le\kappa} |I_\ell(t)|
+
C \one_{\{t>T-2^{-n}\}}2^{n(1-1/\alpha)}\\
&\times \int_{T-2^{-n}}^t(t-s)^{-1/\alpha}
\sup_{|\mathbf{z}|\le\kappa}
\left|
\widehat{u}^{n,i}_{\mathbf{z}}(s,x_i)
-
\rho(u(s,x_i))
\right|\ud s.
\end{aligned}
\end{align}
By Lemma \ref{L:Contraction} (see \eqref{E3:Contraction}),
\begin{align}\label{e1:uHatD}
\sup_{|\mathbf{z}|\le\kappa}
\left|
\widehat{u}^{n,i}_{\mathbf{z}}(t,x_i)-\rho(u(t,x_i))
\right|
\le
C \sum_{\ell=0}^5 F_n\left[
\sup_{|\mathbf{z}|\le\kappa}
|I_\ell(\cdot)|\right](t),\quad
\text{a.s.}
\end{align}
By the same arguments as those in Step 1 and using the fact that $\rho'$ is bounded, we see that
\[
\Norm{\sum_{\ell=1}^4F_n\left[
\sup_{|\mathbf{z}|\le\kappa}|I_\ell(\cdot)|\right](t)
}_p^2 \le C 2^{-n(1-1/\alpha)}.
\]
Hence, the Borel-Cantelli lemma implies that
\begin{align}\label{e2:uHatD}
\lim_{n\rightarrow\infty}
\sum_{\ell=1}^4
F_n\left[
\sup_{|\mathbf{z}|\le\kappa}
|I_\ell(\cdot)|
\right](t) = 0,\quad\text{a.s.}
\end{align}
The term $I_5$ part is identical to the term $I_5$ in Step 1
whence we have the bound \eqref{e3:uHat}.
As for $I_0(t)$, we decompose $z_i\theta^{n,i}_{\mathbf{z}}(t,x_i)$ into three parts
\begin{align*}
z_i\theta^{n,i}_{\mathbf{z}}(t,x_i)= &
\widehat{u}^{n}_{\mathbf{z}}(t,x_i)-u(t,x_i)\\
&+\int_0^t\int_\R G(t-s,x_i-y) \left[\rho\left(u(s,y)\right)-\rho\left(\widehat{u}^n_{\mathbf{z}}(s,y)\right)\right]W(\ud s,\ud y)\\
&-\sum_{j\ne i} z_j \theta_{\mathbf{z}}^{n,j}(t,x_i)\\
=:& I_{0,1}(t)+I_{0,2}(t) - I_{0,3}(t).
\end{align*}
Notice that $I_{0,2}(t)$ is equal to $I_1(t)$ in Step 1.
Hence,
\begin{align*}
 \lim_{n\rightarrow\infty}
F_n\left[\sup_{|\mathbf{z}|\le\kappa}
|I_{0,2}(\cdot)|\right](t)=0,\quad\text{a.s.}
\end{align*}
From \eqref{E4:SupU}, we have that
\begin{align*}
 \lim_{n\rightarrow\infty}
F_n\left[\sup_{|\mathbf{z}|\le\kappa}
|I_{0,3}(\cdot)|\right](t)=0,\quad\text{a.s.}
\end{align*}
As for $I_{0,1}(t)$,
by \eqref{e4:uHat} and \eqref{E:F3}, we see that
\begin{align*}
F_n\left[\sup_{|\mathbf{z}|\le\kappa}\left|I_{0,1}(\cdot)\right|\right](t)
\le&
M_n+C|\rho(u(t,x_i))|+C\sum_{\ell=1}^2 F_n^\ell \big[|\rho(u(\cdot,x_i))|\big](t)\quad\text{a.s.}
\end{align*}
Therefore,
\begin{align}\label{e4:uHatD}
F_n\left[\sup_{|\mathbf{z}|\le\kappa}\left|I_{0}(\cdot)\right|\right](t)
\le&
M_n+C|\rho(u(t,x_i))|+C\sum_{\ell=1}^2 F_n^\ell \big[|\rho(u(\cdot,x_i))|\big](t)\quad\text{a.s.}
\end{align}
Combining \eqref{e1:uHatD}, \eqref{e2:uHatD},
\eqref{e3:uHat} and \eqref{e4:uHatD} shows that
for all $t\in(0,T]$ and $n\in\bbN$,
\begin{align}
\label{e5:uHatD}
\sup_{|\mathbf{z}|\le\kappa}&\left|
 \widehat{u}_{\mathbf{z}}^{n,i}(t,x_i)
 - \rho(u(t,x_i))
\right|
\le
M_n+C\left|\rho(u(t,x_i))\right|+C \sum_{\ell=1}^2 F_n^\ell \left[\left|\rho(u(\cdot,x_i))\right|\right](t)\quad\text{a.s.}
\end{align}
Finally, by letting $t=T$ and sending $n\rightarrow\infty$,
and by the H\"older continuity of $s\mapsto\rho(u(s,x_i))$,
we have that
\[
\limsup_{n\rightarrow\infty}
\sup_{|\mathbf{z}|\le\kappa}\left|
 \widehat{u}_{\mathbf{z}}^{n,i}(T,x_i)
\right|\le C\left|\rho(u(T,x_i))\right|,
\]
which proves
Proposition \ref{P:thetaBdd} for
$\widehat{u}^{n,i}_{\mathbf{z}}(T,x_i)$.

{\bigskip\noindent\bf Step 3.~}
In this step, we will prove
\begin{align}\label{e1:thetaDD}
F_n\left[\sup_{|\mathbf{z}|\le\kappa}\left|\theta^{n,i,k}_{\mathbf{z}}(\cdot,x_i)\right|\right](t)
\le
M_n+C\left|\rho(u(t,x_i))\right|+C \sum_{\ell=1}^3 F_n^\ell \left[\left|\rho(u(\cdot,x_i))\right|\right](t)\quad\text{a.s.}
\end{align}
By \eqref{E:thetaIK-bd}, we need only to consider the case when $k=i$.
Notice from \eqref{E:hatUni} that
\begin{align*}
z_i \theta^{n,i,i}_{\mathbf{z}}(t,x_i)
=& \widehat{u}^{n,i}_{\mathbf{z}}(t,x_i) -
\theta^{n,i}_{\mathbf{z}}(t,x_i) -\sum_{j\ne i}z_j \theta^{n,i,k}_{\mathbf{z}}(t,x_i)-I(t,x_i),
\end{align*}
where
\[
I(t,x_i):=\one_{\{t>T-2^{-n}\}}\int_{T-2^{-n}}^t\int_\R G(t-s,x_i-y)
\rho'\left(\widehat{u}^{n}_{\mathbf{z}}(s,y)\right)
\widehat{u}^{n,i}_{\mathbf{z}}(s,y)W(\ud s,\ud y).
\]
By \eqref{e5:uHatD}, we see that
\begin{align*}
F_n\left[
\sup_{|\mathbf{z}|\le\kappa}\left|
 \widehat{u}_{\mathbf{z}}^{n,i}(\cdot,x_i)
\right|\right](t)
&\le
F_n\left[
\sup_{|\mathbf{z}|\le\kappa}\left|
 \widehat{u}_{\mathbf{z}}^{n,i}(\cdot,x_i)
 - \rho(u(\cdot,x_i))
\right|\right](t) + F_n\Big[|\rho(u(\cdot,x_i))|\Big](t)\\
&\le
M_n+C\left|\rho(u(t,x_i))\right|+C \sum_{\ell=1}^3 F_n^\ell \left[\left|\rho(u(\cdot,x_i))\right|\right](t)\quad\text{a.s.}
\end{align*}
Notice that $\theta_{\mathbf{z}}^{n,i}(t,x_i)$ is the $I_0$ term in Step 2, hence by \eqref{e4:uHatD}
\[
F_n\left[\sup_{|\mathbf{z}|\le\kappa}\left|\theta_{\mathbf{z}}^{n,i}(\cdot)\right|\right](t)
\le
M_n+C|\rho(u(t,x_i))|+C\sum_{\ell=1}^2 F_n^\ell \big[|\rho(u(\cdot,x_i))|\big](t)\quad\text{a.s.}
\]
Because $i\ne j$, from \eqref{E4:SupU} we see that
\[
\Norm{\sup_{|\mathbf{z}|\le\kappa}\sum_{j\ne i}\left|z_j \theta^{n,i,k}_{\mathbf{z}}(T,x_i) \right|}_p^2\le C 2^{-n(1-1/\alpha)}.
\]
Hence,
\[
\lim_{n\rightarrow\infty}
F_n\left[\sum_{j\ne i}\sup_{|\mathbf{z}|\le\kappa}\left|z_j \theta^{n,i,k}_{\mathbf{z}}(\cdot,x_i) \right|\right](t) = 0,\quad\text{a.s.}
\]
By the boundedness of $\rho'$, \eqref{E2:SupU} and \eqref{E:hG3}, we see that
\[
\Norm{\sup_{|\mathbf{z}|\le\kappa}\left|I(t,x_i)\right|}_p^2\le C 2^{-n(1-1/\alpha)},
\]
which implies that
\[
\lim_{n\rightarrow\infty} F_n\left[\sup_{|\mathbf{z}|\le\kappa}\left|I(\cdot,x_i)\right| \right](t)= 0,\quad\text{a.s.}
\]
Combining the above four terms proves \eqref{e1:thetaDD}.

{\bigskip\noindent\bf Step 4.~}
In this last step, we will prove Proposition \ref{P:thetaBdd} for $\widehat{u}^{n,i,k}_{\mathbf{z}}(T,x_i)$.
Write the six parts of $\widehat{u}^{n,i,k}_{\mathbf{z}}(t,x_i)$
in \eqref{E:hatUnik} as in \eqref{E:Unik6}, i.e.,
\begin{align}\label{e:Unik6}
\widehat{u}^{n,i,k}_{\mathbf{z}}(t,x_i) =
\theta^{n,i,k}_{\mathbf{z}}(t,x_i) +\theta^{n,k,i}_{\mathbf{z}}(t,x_i)+ \sum_{\ell=1}^4 U_\ell^n(t,x_i).
\end{align}
We first consider $U^n_4(t,x_i)$ because it contributes to the recursion. Write $U^n_4(t,x_i)$ in three parts
\begin{align*}
 U^n_4(t,x_i) = &\sum_{j\ne i}
 z_j\int_0^t\int_\R G(t-s,x_i-y)\rho'\left(\widehat{u}^n_{\mathbf{z}}(s,y)\right)
 \widehat{u}^{n,i,k}_{\mathbf{z}}(s,y)h_n^j(s,y)\ud s\ud y\\
 &+
 z_i\int_0^t\int_\R G(t-s,x_i-y)\rho'\left(\widehat{u}^n_{\mathbf{z}}(s,y)\right)
 \left[
 \widehat{u}^{n,i,k}_{\mathbf{z}}(s,y)-\widehat{u}^{n,i,k}_{\mathbf{z}}(s,x_i)
 \right]h_n^i(s,y)\ud s\ud y\\
 &+
 z_i\int_0^t\int_\R G(t-s,x_i-y)\rho'\left(\widehat{u}^n_{\mathbf{z}}(s,y)\right)
 \widehat{u}^{n,i,k}_{\mathbf{z}}(s,x_i)
 h_n^i(s,y)\ud s\ud y\\
 =:&U^n_{4,1}(t,x_i)+U^n_{4,2}(t,x_i)+U^n_{4,3}(t,x_i).
\end{align*}
Notice that
\[
\sup_{|\mathbf{z}|\le\kappa}
\left|U_{4,3}^n(t,x_i)\right|
\le C\one_{\{t>T-2^{-n}\}}2^{n(1-1/\alpha)}\int_{T-2^{-n}}^{t} (t-s)^{-1/\alpha}
\sup_{|\mathbf{z}|\le\kappa}
|\widehat{u}^{n,i,k}_{\mathbf{z}}(s,x_i)|  \ud s,\quad\text{a.s.}
\]
Therefore, by Lemma \ref{L:Contraction},
\begin{align}\label{E:2ndD}
\sup_{|\mathbf{z}|\le\kappa}
|\widehat{u}^{n,i,k}_{\mathbf{z}}(t,x_i)|
\le C \sum F_n \left[\sup_{|\mathbf{z}|\le\kappa} I_{\mathbf{z}}^n(\cdot,x_i)\right](t)
\quad\text{a.s.,}
\end{align}
where the summation is over all terms on the right-hand side  of \eqref{e:Unik6} except $U_{4,3}^n(t,x_i)$
and $I_{\mathbf{z}}^n$ stands such a generic term.

By the similar arguments as in the
 previous steps, using \eqref{E3:Space}, one can show that
\begin{align}\label{E_:U4-13}
\lim_{n\rightarrow\infty}
\sup_{|\mathbf{z}|\le\kappa}
F_n\left[\left|U^n_{4,\ell}(\cdot,x_i)\right|\right](t) = 0,\quad\text{a.s. for $\ell=1,2$.}
\end{align}

By \eqref{E2:SupU}, \eqref{E3:SupU}, the boundedness of both $\rho'$ and $\rho''$, and \eqref{E:hG3},
we can obtain moment bounds for both $\sup_{|\mathbf{z}|\le\kappa}\left|U_1^n(t,x_i)\right|$
and
$\sup_{|\mathbf{z}|\le\kappa}\left|U_3^n(t,x_i)\right|$
and then argue using the Borel-Cantelli lemma as above to conclude that
\begin{align}\label{E_:U13}
\lim_{n\rightarrow\infty}
F_n\left[\sup_{|\mathbf{z}|\le\kappa}
\left|U_\ell^n(\cdot,x_i)\right|\right](t) =0, \quad \text{a.s. for $\ell=1,3$.}
\end{align}
We claim that
\begin{align}\label{E_:U2}
F_n\left[\sup_{|\mathbf{z}|\le\kappa}
\left|U_{2}^{n}(\cdot,x_i)\right|\right](t)
\le &
M_n+C\rho^2(u(t,x_i))+C \sum_{\ell=1}^4 F_n^\ell \left[\rho^2(u(\cdot,x_i))\right](t)\quad\text{a.s.}
\end{align}
Write $U_2^n(t,x_i)$ as
\begin{align*}
 U_2^n(t,x_i)
 = \sum_{j=1}^d
 z_j U_2^{n,i,j,k}(t,x_i),
\end{align*}
where
\[
U_2^{n,i,j,k}(t,x_i)=\int_0^t
\int_\R G(t-s,x_i-y)\rho''\left(\widehat{u}_{\mathbf{z}}^n(s,y)\right)
\widehat{u}_{\mathbf{z}}^{n,i}(s,y)
\widehat{u}_{\mathbf{z}}^{n,k}(s,y)
h_n^j(s,y)\ud s\ud y.
\]
We need only take care of the case when $i=j=k$
because otherwise it is not hard to show that
\[
\lim_{n\rightarrow\infty} \sup_{|\mathbf{z}|\le\kappa}
F_n\left[|U_2^{n,i,j,k}(\cdot,x_i)|\right](t) =0 \quad\text{a.s.}
\]
As for $U_2^{n,i,i,i}(t,x_i)$, it is nonnegative and
\begin{align*}
U_2^{n,i,i,i}(t,x_i)
\le& C
\int_0^t
\int_\R G(t-s,x_i-y)
\widehat{u}_{\mathbf{z}}^{n,i}(s,y)^2
h_n^i(s,y)\ud s\ud y\\
\le&
\quad C
\int_0^t
\int_\R G(t-s,x_i-y)
\left[\widehat{u}_{\mathbf{z}}^{n,i}(s,y)-\widehat{u}_{\mathbf{z}}^{n,i}(s,x_i)\right]^2
h_n^i(s,y)\ud s\ud y\\
&+C
\int_0^t
\int_\R G(t-s,x_i-y)
\left[\widehat{u}_{\mathbf{z}}^{n,i}(s,x_i)-\rho(u(s,x_i))\right]^2
h_n^i(s,y)\ud s\ud y\\
&+C
\int_0^t
\int_\R G(t-s,x_i-y)
\left[\rho(u(s,x_i))-\rho(u(t,x_i))\right]^2
h_n^i(s,y)\ud s\ud y\\
=:&\quad
U_{2,1}^{n,i,i,i}(t,x_i)
+U_{2,2}^{n,i,i,i}(t,x_i)
+U_{2,3}^{n,i,i,i}(t,x_i).
\end{align*}
By \eqref{E2:Space} and the H\"older continuity of $s\mapsto \rho(u(s,x_i))$
one can prove in the same way as before that
\[
\lim_{n\rightarrow\infty}
F_n\left[
\sup_{|\mathbf{z}|\le\kappa}
U_{2,\ell}^{n,i,i,i}(\cdot,x_i)
\right](t)=0,\quad\text{a.s. for $\ell=1,3$.}
\]
Notice that
\[
\sup_{|\mathbf{z}|\le\kappa}U_{2,2}^{n,i,i,i}(t,x_i)
\le
C F_n^1 \left[\sup_{|\mathbf{z}|\le\kappa}
\left[\widehat{u}_{\mathbf{z}}^{n,i}(\cdot,x_i)-\rho(u(\cdot,x_i))\right]^2
\right](t).
\]
By applying \eqref{E:F5} on \eqref{e5:uHatD} with $m'=2$, we see that
\begin{align*}
\sup_{|\mathbf{z}|\le\kappa}
\left[\widehat{u}_{\mathbf{z}}^{n,i}(s,x_i)-\rho(u(s,x_i))\right]^2
&\le
M_n+C\rho^2(u(t,x_i))+C \sum_{\ell=1}^2 F_n^\ell \left[\rho^2(u(\cdot,x_i))\right](t)\quad\text{a.s.}
\end{align*}
Hence, by \eqref{E:F3}, for all $n\in\bbN$,
\begin{align*}
\sup_{|\mathbf{z}|\le\kappa}U_{2,2}^{n,i,i,i}(t,x_i)
\le &
M_n+C\rho^2(u(t,x_i))+C \sum_{\ell=1}^3 F_n^\ell \left[\rho^2(u(\cdot,x_i))\right](t)\quad\text{a.s.}
\end{align*}
Then another application of \eqref{E:F3} shows that
\begin{align*}
F_n\left[\sup_{|\mathbf{z}|\le\kappa}U_{2,2}^{n,i,i,i}(\cdot,x_i)\right](t)
\le &
M_n+C\rho^2(u(t,x_i))+C \sum_{\ell=1}^4 F_n^\ell \left[\rho^2(u(\cdot,x_i))\right](t)\quad\text{a.s.}
\end{align*}
Combining these terms, we have thus proved \eqref{E_:U2}.
\bigskip

Finally, combining \eqref{e1:thetaDD}, \eqref{E_:U13}, \eqref{E_:U2}, and \eqref{E_:U4-13}, we see that
\begin{align*}
\sup_{|\mathbf{z}|\le\kappa}
|\widehat{u}^{n,i,k}_{\mathbf{z}}(t,x_i)|
\le & M_n+C\left|\rho(u(t,x_i))\right| + C \rho^2(u(t,x_i))\\
&+C \sum_{\ell=1}^3 F_n^\ell \left[\left|\rho(u(\cdot,x_i))\right|\right](t)
+C \sum_{\ell=1}^4 F_n^\ell \left[\rho^2(u(\cdot,x_i))\right](t)\quad\text{a.s.}
\end{align*}
Then by letting $t=T$ and sending $n\rightarrow\infty$ in the above inequality,
and by the H\"older continuity of $s\mapsto\rho(u(s,x_i))$, we have that
\[
\limsup_{n\rightarrow\infty}
\sup_{|\mathbf{z}|\le\kappa}\left|
 \widehat{u}_{\mathbf{z}}^{n,i,k}(T,x_i)
\right|\le C \left(\left|\rho(u(T,x_i))\right|+\rho^2(u(T,x_i))\right),
\]
which proves
Proposition \ref{P:thetaBdd} for
$\widehat{u}^{n,i,k}_{\mathbf{z}}(T,x)$.
Thus  we complete the   proof of Proposition \ref{P:thetaBdd}.
\end{proof}

\appendix
\section{Appendix}
\subsection{Some miscellaneous results}

The following lemma is a well-known result.
\begin{lemma}\label{L:NegMom}
Let $X$ be a real-valued nonnegative random variable.
The following two statements are equivalent:
\begin{enumerate}
 \item[(1)] For all $p>0$, there exists some finite constant $C_p>0$ such that
\[
\bbP\left(X<\epsilon\right)< C_p \: \epsilon^p\quad\text{for all $\epsilon>0$.}
\]
\item[(2)] $\E[X^{-q}]<\infty$ for all $q>0$.
\end{enumerate}
\end{lemma}


The following lemma can be viewed as a special case of
Propositions 3.2 and 4.3 in \cite{ChenDalang15FracHeat}.
This lemma is used in the proof of Theorem \ref{T:Pos}.
We will need the two-parameter {\it Mittag-Leffler
function} (see \cite[Section 1.2]{Podlubny99FDE} or
\cite[Section 1.9]{kilbas-fractional})  is defined as follows:
\begin{align}\label{E:Mittag-Leffler}
E_{\alpha,\beta}(z) := \sum_{k=0}^{\infty} \frac{z^k}{\Gamma(\alpha k+\beta)},
\qquad \alpha>0,\;\beta> 0.
\end{align}
\begin{lemma}\label{L:Mittag}
Suppose that $\alpha>1$, $\lambda>0$ and that $\beta:\R_+\mapsto\R$ is a locally integrable function.
\begin{enumerate}[parsep=0ex,topsep=1ex]
\item[(1)]  If $f$ satisfies
\begin{align}\label{E:IntEqf}
f(t)= \beta(t) + \lambda \int_0^t (t-s)^{-1/\alpha} f(s)\ud s \qquad \text{for $t\ge 0$,}
\end{align}
then
\begin{align}\label{E2:IntEqf}
f(t)= \beta(t)+\int_0^t \beta(s) K_\lambda(t-s)\ud s,
\end{align}
where
\begin{align}\label{E:Klambda}
K_\lambda(t) = t^{-1/\alpha}
\lambda \Gamma(1-1/\alpha) E_{1-1/\alpha,1-1/\alpha}\left(t^{1-1/\alpha}\lambda \Gamma(1-1/\alpha)\right).
\end{align}
Moreover, when $\beta(x)\ge 0$, if the equality in \eqref{E:IntEqf} is replaced by ``$\le$'' (resp. ``$\ge$''),
then the equality in the conclusion \eqref{E2:IntEqf} should be replaced by
``$\le$'' (resp. ``$\ge$'') accordingly.
\item[(2)]  If $\beta$ is a nonnegative constant, then there exist some nonnegative constants $C_\alpha$ and $\gamma_\alpha$ that depend on $\alpha$ only, such that for all $t\ge 0$,
\begin{align}\label{E3:IntEqf}
f(t) &=\beta \Gamma(1-1/\alpha) \lambda t^{1-1/\alpha} E_{1-1/\alpha,2-1/\alpha}(\Gamma(1-1/\alpha) \lambda t^{1-1/\alpha})\\
&\le \beta C_\alpha \exp\left(\gamma_\alpha \lambda^{\frac{\alpha}{\alpha-1}} t\right).
\label{E4:IntEqf}
\end{align}
\end{enumerate}
\end{lemma}
\begin{proof}
(1) Denote $g(t)=t^{-1/\alpha}$.
Let ``$*$'' be the convolution in time, i.e., $(g*f)(t)=\int_0^t g(s)f(t-s)\ud s$.
Define $k_0(t)= g(t)$ and for $n\ge 1$,
\[
k_n(t):=\Big(\underbrace{g*\cdots*g}_{\text{$(n+1)$'s $g$}}\Big)(t).
\]
Denote $b=1-1/\alpha$.
We claim that
\begin{align}\label{E_:Induction}
k_n(t)= \Gamma(b) t^{-1/\alpha} \frac{[t^{b}\Gamma(b)]^n}{\Gamma((n+1)b)},
\quad\text{for all $n\ge 1$.}
\end{align}
When $n=1$, by the definition of the Beta functions,
\begin{align*}
 (g*g)(t)= t^{1-2/\alpha}\int_0^1(s(1-s))^{-1/\alpha}\ud s = t^{1-2/\alpha}\frac{\Gamma(b)^2}{\Gamma(2b)} = k_1(t).
\end{align*}
If \eqref{E_:Induction} is true for $n$, then by the same reason,
\begin{align*}
 (k_n*g)(t) = & \frac{\Gamma(b)^{n+1}}{\Gamma((n+1)b)}\int_0^t s^{(n+1)b-1} (t-s)^{b-1} \ud s\\
 =&
t^{-1/\alpha} t^{(n+1)b} \frac{\Gamma(b)^{n+2}}{\Gamma((n+2)b)}=k_{n+1}(t).
\end{align*}
Hence, \eqref{E_:Induction} is true.
Therefore,
\begin{align*}
K_\lambda(t)&:=k_0(t)+\sum_{n=1}^\infty \lambda^{n+1} k_n(t) \\
& =
t^{-1/\alpha} + \lambda \Gamma(b) t^{-1/\alpha} \left[ E_{b,b}\left(t^b\lambda \Gamma(b)\right)- \frac{1}{\lambda\Gamma(b)}\right]\\
&=
\lambda \Gamma(b) t^{-1/\alpha} E_{b,b}\left(t^b\lambda \Gamma(b)\right).
\end{align*}
Finally, by successive replacements, we see that
\[
f(t)= \beta(t) + \sum_{n=0}^\infty (\beta * k_n)(t)
=\beta(t)+ (\beta*K_\lambda)(t).
\]
The remaining  part of (1) is clear.
\eqref{E3:IntEqf} is a direct consequence of (1) and the following integral (see (1.99) in \cite{Podlubny99FDE})
\[
\int_0^z E_{\alpha,\beta}(\lambda t^\alpha) t^{\beta-1}\ud t= z^\beta E_{\alpha,\beta+1}(\lambda z^\alpha),\quad \beta>0.
\]
Finally, for any $\gamma_\alpha>\Gamma(b)^{1/b}$, by the asymptotic property of
the Mittag-Leffler function (see Theorem 1.3 in \cite{Podlubny99FDE}),
\[
C_\alpha := \sup_{x\ge 0} x E_{b,1+b}(\Gamma(b) x^b )\exp\left(-\gamma_\alpha x \right)<\infty.
\]
This completes the proof of Lemma \ref{L:Mittag}.
\end{proof}

\begin{lemma}\label{L:Contraction}
Suppose that $\alpha\in (1,2]$, $\lambda>0$, $T>\epsilon>0$, and $\beta_\epsilon:\R_+\mapsto\R$ is a locally integrable function.
\begin{enumerate}[parsep=0ex,topsep=1ex]
\item[(1)] If $f$ satisfies
\begin{align}\label{E:Contraction}
f(t)= \beta_\epsilon(t) + \lambda \epsilon^{-(1-1/\alpha)}\one_{\{t>T-\epsilon\}}\int_{T-\epsilon}^t (t-s)^{-1/\alpha} f(s)\ud s
\qquad \text{for all $t\in (0,T]$,}
\end{align}
then
\begin{align}\label{E2:Contraction}
f(t)= \beta_\epsilon(t)+\one_{\{t>T-\epsilon\}}\int_{T-\epsilon}^t K_{\lambda \epsilon^{-(1-1/\alpha)}}(t-s)\beta_\epsilon(s)\ud s,
\qquad \text{for all $t\in (0,T]$,}
\end{align}
where $K_{\lambda \epsilon^{-(1-1/\alpha)}}(t)$ is defined in \eqref{E:Klambda}.
Moreover, when $\beta_\epsilon(t)\ge 0$, if the equality in \eqref{E:Contraction} is replaced by ``$\le$'' (resp. ``$\ge$''),
then the equality in the conclusion \eqref{E2:Contraction} should be replaced by
``$\le$'' (resp. ``$\ge$'') accordingly.
\item[(2)] When $\beta_\epsilon(t)\ge 0$, for some nonnegative constant $C$ that depends on $\alpha$, $\lambda$ and $T$ (not on $\epsilon$),
we have
\begin{align}\label{E3:Contraction}
f(t)\le \beta_\epsilon(t)+C \one_{\{t>T-\epsilon\}}\epsilon^{-(1-1/\alpha)}
\int_{T-\epsilon}^t (t-s)^{-1/\alpha}\beta_\epsilon(s)\ud s,
\end{align}
for all $t\in [0,T]$.
Moreover, if $\beta_\epsilon$ is a nonnegative constant, then for the same constant $C$, it holds that
\begin{align}
 f(t)&\le C \beta_\epsilon,\quad\text{for all $t\in[0,T]$}.
\label{E4:Contraction}
\end{align}
\end{enumerate}
\end{lemma}

\begin{proof}
(1) Write $\beta_\epsilon(t)$ in two parts
\[
\beta_\epsilon(t) =
\beta_\epsilon(t)\one_{\{t\le T-\epsilon\}}
+\beta_\epsilon(t)\one_{\{t> T-\epsilon\}}
=:
\beta_\epsilon^\dag(t)+ \beta_\epsilon^*(t).
\]
Denote accordingly $f^*(t)= f(t)\one_{\{t> T-\epsilon\}}$. Then it is clear that
$f(t) = \beta_\epsilon^\dag(t) + f^*(t)$ and $f^*$ satisfies
\[
f^*(t)= \beta_\epsilon^*(t) + \lambda \epsilon^{-(1-1/\alpha)}
\int_{0}^t (t-s)^{-1/\alpha} f^*(s)\ud s
\qquad \text{for all $t\ge 0$.}
\]
Hence, by Lemma \ref{L:Mittag},
\begin{align*}
f^*(t) &=
\beta_\epsilon^*(t)  + \int_0^t K_{\lambda\epsilon^{-(1-1/\alpha)}}(t-s)\beta_\epsilon^*(s)\ud s\\
&=
\beta_\epsilon^*(t)  + \one_{\{t>T-\epsilon\}}\int_{T-\epsilon}^t K_{\lambda\epsilon^{-(1-1/\alpha)}}(t-s)\beta_\epsilon(s)\ud s.
\end{align*}
Then adding $\beta_\epsilon^\dag(t)$ on
both sides proves \eqref{E2:Contraction}.

(2) By the property of the Mittag-Leffler function, we see that
for some constant $C$ that does not depend on $\epsilon$,
\[
K_{\lambda\epsilon^{-(1-1/\alpha)}}(t-s)\le
C \lambda\epsilon^{-(1-1/\alpha)}
(t-s)^{-1/\alpha} \exp\left(C \lambda\epsilon^{-(1-1/\alpha)}(t-s) \right)\,.
\]
Since the integral in \eqref{E2:Contraction}
is nonvanishing only when $T-t\le \epsilon$ and since $t-s\le T-t$, we see that
\[
K_{\lambda\epsilon^{-(1-1/\alpha)}}(t-s)\le
C \lambda\epsilon^{-(1-1/\alpha)}
(t-s)^{-1/\alpha} \exp\left(C \lambda\epsilon^{1/\alpha} \right)
\le C'\epsilon^{-(1-1/\alpha)} (t-s)^{-1/\alpha}.
\]
Putting this upper bound back into \eqref{E2:Contraction} proves \eqref{E3:Contraction}.
\eqref{E4:Contraction} is clearly  from \eqref{E3:Contraction}.
This completes the proof of Lemma \ref{L:Contraction}.
\end{proof}

\begin{lemma}\label{L:IntG-G}
Suppose that $\alpha\in (1,2]$ and $|\delta|\le 2-\alpha$. For all $\epsilon, h\in (0,1)$ and $y\in\R$,
it holds that
\begin{align}
\int_0^\epsilon \ud s \int_{-\epsilon}^\epsilon \ud z
\: \left|G(s,y-z)-G(s,z)\right|&\le C \epsilon^{2(1-1/\alpha)} |y|.
\end{align}
\end{lemma}
\begin{proof}
Denote the integral by $I$.
Choose any $\beta\in (1,\alpha)$ and let $\beta^*=\beta/(\beta-1)$.
It is clear that $\beta^*\ge 2$.
By the scaling property of $G$ and by H\"older's inequality,
\begin{align*}
 I &=\int_0^\epsilon \ud s \int_{-\epsilon}^{\epsilon} \ud z\:
 s^{-1/\alpha}\left|G\left(1,\frac{y-z}{s^{1/\alpha}}\right)-G\left(1,\frac{z}{s^{1/\alpha}}\right)\right|\\
&=
\left(
\int_0^\epsilon \ud s \int_{-\epsilon}^{\epsilon} \ud z\:
 s^{-\beta/\alpha}
\right)^{1/\beta}
\left(
 \int_0^\epsilon \ud s \int_{-\epsilon}^{\epsilon} \ud z\:\left|G\left(1,\frac{y-z}{s^{1/\alpha}}\right)-G\left(1,\frac{z}{s^{1/\alpha}}\right)\right|^{\beta^*}
 \right)^{1/\beta^*}\\
&\le
C \epsilon^{\left(2-\frac{\beta}{\alpha}\right)\frac{1}{\beta}}
\left(
 \int_0^\epsilon \ud s \int_{-\epsilon}^{\epsilon} \ud z\:\left(\frac{|y|}{s^{1/\alpha}}\right)^{\beta^*}
 \right)^{1/\beta^*}\\
&\le C \epsilon^{\left(2-\frac{\beta}{\alpha}\right)\frac{1}{\beta}} \epsilon^{\left(2-\frac{\beta^*}{\alpha}\right)\frac{1}{\beta^*}} |y|
= C \epsilon^{2(1-1/\alpha)} |y|,
\end{align*}
where we have used the fact that $G^{(1)}(1,x)$ is bounded (see \eqref{E:BddG}).
\end{proof}

\begin{proposition}[Proposition  4.4 of \cite{ChenDalang15FracHeat}]
\label{P:G}
For $\alpha\in(1,2]$ and $|\delta|\le 2-\alpha$,
there is a constant $C$ depends only on $\alpha$ and $\delta$ such that
for all $t\ge s>0$ and $x,y\in\R$,
\begin{align}\label{E:G-x}
 &\int_0^t\ud r\int_\R \ud z
\left[G(t-r,x-z)-G(t-r,y-z)\right]^2
\le C |x-y|^{\alpha-1},\\
\label{E:G-t1}
&\int_0^s\ud r\int_\R \ud z
\left[G(t-r,x-z)-G(s-r,x-z)\right]^2
\le C (t-s)^{1-1/\alpha},
\end{align}
and
\begin{align}\label{E:G-t2}
\int_s^t\ud r\int_\R \ud z \left[G(t-r,x-z)\right]^2
& \le C (t-s)^{1-1/\alpha}\;.
\end{align}
\end{proposition}

\bigskip
In the following theorem, we state the Kolmogorov continuity theorem with an emphasize on
the constants. The proof can be found, e.g., in \cite[Theorem 4.3]{Khoshnevisan09Mini}.

\begin{theorem}[The Kolmogorov continuity theorem]
\label{T:KolCont}
Suppose $\{X(\mathbf{t})\}_{\mathbf{t}\in T}$
is a stochastic process indexed by a compact cube $T:=[a_1,b_1]\times\dots \times [a_d,b_d]\subset \R^d$.
Suppose also that there exists constants $\Theta>0$, $p>0$, and $\gamma>d$ such that uniformly
for all $\mathbf{s}$, $\mathbf{t}\in T$,
\[
\Norm{X(\mathbf{t})-X(\mathbf{s})}_p \le \Theta |\mathbf{t}-\mathbf{s}|^{\gamma/p}.
\]
Then $X$ has a continuous modification $\bar{X}$. Moreover, if $0\le \theta<(\gamma-d)/p$, then
\begin{align}\label{E:KolCont}
\Norm{\sup_{\mathbf{t}\ne \mathbf{s}}\frac{\left|\bar{X}(\mathbf{t})-\bar{X}(\mathbf{s})\right|}{|\mathbf{t}-\mathbf{s}|^\theta}}_p
\le C_{\gamma,\theta,p,d}\: \Theta.
\end{align}
\end{theorem}

\bigskip
The following lemma can be used to give an   explicit form for $c_n$ defined  by
\eqref{E2:c2n} when $\alpha=2$ (see Remark \ref{R:c2n}).

\begin{lemma}\label{L:c2n}
For all $\nu>0$ and $t>0$,
\[
\int_0^t\ud s\int_{-t}^t \ud y \:  \frac{1}{\sqrt{2\pi \nu s}}\exp\left(-\frac{y^2}{2\nu s}\right)
=t \left[2 \left(\frac{t}{\nu}+1\right) \Phi\left(\sqrt{t/\nu}\right)-\frac{2
   t}{\nu}+\sqrt{2/\pi} \: e^{-\frac{t}{2\nu}}
   \sqrt{t/\nu}-1\right].
\]
\end{lemma}
\begin{proof}
Recall that $\Phi(x)=\int_{-\infty}^x \frac{1}{\sqrt{2\pi}}e^{-\frac{y^2}{2}}\ud y$.
Because $\int_{-t}^t \frac{1}{\sqrt{2\pi\nu s}}\exp\left(-\frac{y^2}{2\nu s}\right)\ud y=2\Phi\left(t/\sqrt{\nu s}\right)-1$,
by change of variable $t/\sqrt{\nu s}=u$, the double integral is equal to
\[
2t^2\nu^{-1}\int_{\sqrt{t/\nu}}^\infty \Phi(u)u^{-3}\ud u-t.
\]
Now we evaluate the $\ud u$ integral. Apply the integration-by-parts twice to obtain that
\begin{align*}
2\int_{\sqrt{t/\nu}}^\infty \Phi(u)u^{-3}\ud u
&=
-u^{-2}\Phi(u)\Big|_{\sqrt{t/\nu}}^{\infty} + \int_{\sqrt{t/\nu}}^\infty
u^{-2} \frac{1}{\sqrt{2\pi}}e^{-u^2/2}\ud u \\
&=-u^{-2}\Phi(u)\Big|_{\sqrt{t/\nu}}^{\infty} -
u^{-1}\frac{1}{\sqrt{2\pi}}e^{-u^2/2}\Big|_{\sqrt{t/\nu}}^{\infty}
+\int_{\sqrt{t/\nu}}^\infty
u^{-1}\frac{1}{\sqrt{2\pi}}e^{-u^2/2} (-u)\ud u\\
&=\frac{\nu}{t}\Phi(\sqrt{t/\nu}\:) +\sqrt{\frac{\nu}{2\pi t}} e^{-\frac{t}{2\nu}} -\left(1-\Phi(\sqrt{t/\nu}\:)\right).
\end{align*}
Lemma \ref{L:c2n} is proved after some simplification.
\end{proof}

\subsection{A general framework from Hu {\it et al} \texorpdfstring{\cite{HHNS14}}{AA}}
\label{S:HHNS}
In Hu, Huang, Nualart and Sun \cite{HHNS14}, 
under certain assumptions on the fundamental solutions and the noise, 
the existence and smoothness of density for the following SPDE have been studied
\begin{align}\label{E:Lubs}
L u(t,x) = b(u(t,x)) + \sigma(u(t,x))\W(t,x), \quad t> 0, \; x\in\R^d.
\end{align}
In this equation, $L$ denotes a second order partial differential operator and $\W$ is a centered Gaussian noise that is white in time and homogeneous in space. 
Formally, 
\begin{align}\label{E:f}
\E\left[\W(t,x)\W(s,y)\right]=\delta_0(t-s)f(x-y),
\end{align}
where $f$ is some nonnegative and nonnegative definite function. 
Let $\mu$ be the spectral measure, i.e., the Fourier transform of $f$.
Let $\calH$ be the Hilbert space obtained by the completion of  $C_0^\infty(\R^d)$ according to the inner product 
\[
\InPrd{\psi,\phi}_\calH =\int_{\R^d}\ud x\int_{\R^d}
\ud y \: \phi(x)f(x-y)\phi(y) 
=\int_{\R^d} \calF \phi(\xi) \overline{\calF\phi(\xi)} \mu(\ud\xi),\quad
\psi,\phi\in C_0^\infty(\R^d).
\]
The space $\calH$ may contain distributions.
Set $\calH_0=L^2([0,\infty);\calH)$.
\begin{itemize}
 \item[(H1)] The fundamental solution to $Lu=0$, denoted by $G$, satisfies that for all $t>0$, $G(t,\ud x)$ is a nonnegative measure with rapid decrease, such that for all $T>0$,
 \begin{align}
 \int_0^T\int_{\R^d} \left|\calF G(t)(\xi)\right|^2 \mu(\ud \xi)\ud t<\infty,
 \end{align}
 and 
 \begin{align}
 \sup_{t\in [0,T]} G(T,\R^d)\le C_T<\infty.
 \end{align}
 \item[(H2)] There exist positive constants $\kappa_1$ and $\kappa_2$
 such that for all $s,t\in[0,T]$, $x,y\in\R^d$, $T>0$ and $p\ge 1$,
 \begin{align}
  \E\left[\left|u(s,x)-u(t,x)\right|^p\right] &\le C_{p,T} |t-s|^{\kappa_1 p},\\
  \E\left[\left|u(t,x)-u(t,y)\right|^p\right] &\le C_{p,T} |x-y|^{\kappa_2 p},
 \end{align}
 for some constant $C_{p,T}$ which only depends on $p$ and $T$.
 \item[(H3)] There exist some constants $\eta>0$, $\epsilon_0>0$ and $C>0$ such that for all $0<\epsilon\le \epsilon_0$, 
 \begin{align}
  C \: \epsilon^\eta \le \int_0^\epsilon \Norm{G(r)}_\calH^2 \ud r.
 \end{align}
 \item[(H4)] Let $\eta$ be the constant defined in (H3) and $\kappa_1$ and $\kappa_2$ be the constants defined in (H2).
 \begin{enumerate}
  \item[(i)] There exist some constants $\eta_1>\eta$ and $\epsilon_1>0$ such that for all $0<\epsilon\le \epsilon_1$,
  \begin{align}
    \int_0^\epsilon r^{k_1}\Norm{G(r)}_\calH^2 \ud r\le C \epsilon^{\eta_1}.
  \end{align}
  \item[(ii)] There exists some constant $\eta_2>\eta$ such that for each fixed non-zero $w\in\R^d$, there exists two positive constants $C_w$ and $\epsilon_2$ satisfying 
  \begin{align}
   \int_0^\epsilon \InPrd{G(r,*),G(r,w+*)}_\calH\ud r \le C_w \epsilon^{\eta_2},\qquad \text{for all $0\le\epsilon\le \epsilon_2$.}
  \end{align}
  \item[(iii)] The measure $\Psi(t)$ defined by $|x|^{\kappa_2}G(t,\ud x)$ satisfies $\int_0^T\int_{\R^d}|\calF\Psi(t)(\xi)|^2\mu(\ud\xi)\ud t<\infty$ and there exists a constant $\eta_3>\eta$ such that for each fixed $w\in\R^d$, there exist two positive constants $C_w$ and $\epsilon_3$ satisfying 
  \begin{align}
   \int_0^\epsilon \InPrd{|*|^{\kappa_2}G(r,*),G(r,w+*)}_\calH\ud r \le C_w \epsilon^{\eta_3},\qquad \text{for all $0\le\epsilon\le \epsilon_3$.}
  \end{align}
 \end{enumerate}
\end{itemize}
\begin{theorem}[Theorem 3.1 of \cite{HHNS14}]\label{T:HHNS}
Assume that conditions (H1)-- (H4) hold, and the coefficients $\sigma$ and $b$ are smooth functions with bounded derivatives of all orders. 
Let $u(t,x)$ be the solution to \eqref{E:Lubs} with vanishing initial data.
Fix $t>0$ and let $x_1,\dots, x_n$ be $n$ distinct points in $\R^d$.
Assume that $u(t,x_i)$, $i=1,\dots, n$, satisfy the condition that 
for some positive constant $C_1$, 
\begin{align}\label{E:HHNS}
 |\sigma(u(t,x_i))|\ge C_1 \quad \text{$\bbP$-a.s. for any $i=1,\dots,n$.}
\end{align}
Then the law of the random vector $(u(t,x_1),\dots,u(t,x_n))$ has a smooth density with respect to the Lebesgue measure on $\R^n$.
\end{theorem}
\begin{remark}\label{R:HHNS}
In the remark 3.2 of \cite{HHNS14}, the authors commented that using a localization procedure developed in the proof of Theorem 3.1 in \cite{BP98}, one can obtain a version of Theorem \ref{T:HHNS} without assuming \eqref{E:HHNS}.
In this case, one concludes that the law of 
$(u(t,x_1),\dots,u(t,x_n))$ has a smooth density on $\{y\in\R: \sigma(y)\ne 0\}^n$. 
\end{remark}

\section*{Acknowledgements}
We appreciate some useful comments and stimulating discussions with Robert C. Dalang, Jin Feng, and Jingyu Huang.

\addcontentsline{toc}{section}{Bibliography}
\def\polhk#1{\setbox0=\hbox{#1}{\ooalign{\hidewidth
  \lower1.5ex\hbox{`}\hidewidth\crcr\unhbox0}}} \def\cprime{$'$}

\vspace{3em}
\hfill\begin{minipage}{0.55\textwidth}
{\bf Le CHEN}, {\bf Yaozhong HU}, {\bf David NUALART}\\[0.2em]
Department of Mathematics\\
University of Kansas\\
405 Snow Hall, 1460 Jayhawk Blvd,\\
Lawrence, Kansas, 66045-7594, USA.\\
E-mails: \: \url{chenle, yhu, nualart@ku.edu}
\end{minipage}

\end{document}